\documentclass[oneside,a4paper,reqno]{amsart}
\usepackage{pdfsync, enumitem,float}
\usepackage{amsmath,bm}
\usepackage{mathtools}
\usepackage{fullpage}
\usepackage{graphicx}
\usepackage[all,cmtip]{xy}
\usepackage{tikz-cd}
\usepackage{stmaryrd}
\usepackage{mathrsfs}
\usepackage{hyperref}
\usepackage{tikz-cd}
\usepackage{amssymb}
\usepackage{stackrel}
\usepackage{adjustbox}
\usepackage{multicol}
\usepackage{amsmath, amsthm, amscd, amssymb, latexsym, eucal}
\usepackage[all]{xy}
 \calclayout \makeatletter
 \calclayout \makeatletter
\def\serieslogo@{} \def\@setcopyright{} \makeatother

\usepackage[colorinlistoftodos]{todonotes}

\usepackage{hyperref}
\usepackage{color}
\usepackage{cite}
\usepackage{quiver}

\makeatletter
\renewcommand*\env@matrix[1][c]{\hskip -\arraycolsep
	\let\@ifnextchar\new@ifnextchar
	\array{*\c@MaxMatrixCols #1}}
\makeatother

\usepackage{color}

\numberwithin{equation}{section}
\newtheorem{thm}{Theorem}[section]
\newtheorem*{main-thm}{Theorem}
\newtheorem*{Auslander-thm}{Auslander's Theorem}

\newtheorem{cor}[thm]{Corollary}
\newtheorem{lem}[thm]{Lemma}
\newtheorem{prop}[thm]{Proposition}

\theoremstyle{definition}
\newtheorem{defn}[thm]{Definition}
\newtheorem{rem}[thm]{Remark}
\newtheorem{exmp}[thm]{Example}

\newtheorem*{connot}{Notation and Conventions}

\newcommand\padova{%
\mathrel{\ooalign{\hss{\scalebox{0.5}{$\longleftrightarrow$}}\hss\cr%
\kern0.9ex\raise0.55ex\hbox{\scalebox{0.5}{$\boldsymbol{\bm{\vert}}$}}}}}
\newcommand\hfpadova{%
\mathrel{\ooalign{\hss{\scalebox{0.5}{$\longleftrightarrow$}}\hss\cr%
\kern0.9ex\raise0.55ex\hbox{\scalebox{0.5}{$\boldsymbol{\mathsf{h}\mathsf{f}\ \ }$}}}}}
\newcommand\rpadova{%
\mathrel{\ooalign{\hss{\scalebox{0.5}{$\longrightarrow$}}\hss\cr%
\kern0.9ex\raise0.55ex\hbox{\scalebox{0.5}{$\boldsymbol{\bm{\vert}}$}}}}}
\newcommand\lpadova{%
\mathrel{\ooalign{\hss{\scalebox{0.5}{$\longleftarrow$}}\hss\cr%
\kern0.9ex\raise0.55ex\hbox{\scalebox{0.5}{$\boldsymbol{\bm{\vert}}$}}}}}

\newcommand{\T}{\mathcal T}
\newcommand{\U}{\mathcal U}
\newcommand{\V}{\mathcal V}

\newcommand{\X}{\mathcal X}
\newcommand{\Y}{\mathcal Y}

\newcommand{\Tb}{\mathcal{T}^{\mathsf{b}}}
\newcommand{\Tc}{\mathcal{T}^{\mathsf{c}}}

\DeclareMathOperator{\pd}{\mathsf{pdim}}

\DeclareMathOperator*{\gd}{\mathsf{gl.dim}}

\DeclareMathOperator*{\gld}{\mathsf{gl.dim}}

\DeclareMathOperator*{\Findim}{\mathsf{Fin.dim}}

\DeclareMathOperator*{\Mod}{\mathsf{Mod}-\!}

\DeclareMathOperator*{\smod}{\mathsf{mod}-\!}

\DeclareMathOperator*{\umod}{\underline{\mathsf{mod}}-\!}
\DeclareMathOperator*{\uMod}{\underline{\mathsf{Mod}}-\!}

\DeclareMathOperator*{\inj}{\mathsf{inj}-\!}
\DeclareMathOperator*{\proj}{\mathsf{proj}-\!}
\DeclareMathOperator*{\Gproj}{\mathsf{Gproj}-\!}
\DeclareMathOperator*{\GProj}{\mathsf{GProj}-\!}

\DeclareMathOperator*{\Inj}{\mathsf{Inj}-\!}

\DeclareMathOperator*{\Proj}{\mathsf{Proj}-\!}

\DeclareMathOperator*{\Add}{\mathsf{Add}}
\DeclareMathOperator*{\Product}{\mathsf{Prod}}

\DeclareMathOperator{\Hom}{\mathsf{Hom}}

\usepackage{stackengine}

\DeclareMathOperator*{\?}{\textnormal{\textsf{?}}}

\newsavebox{\proofbox}
\savebox{\proofbox}{\begin{picture}(7,7)%
	\put(0,0){\framebox(7,7){}}\end{picture}}

\usepackage{latexsym}
\usepackage{pstricks}
\usepackage{comment}

\usepackage[greek,english]{babel}

\begin{document}

\title{Intrinsic homological algebra for triangulated categories} 

\author[Kostas]{Panagiotis Kostas}
\address{Department of Mathematics, Aristotle University of Thessaloniki, Thessaloniki 54124, Greece}
\email{pkostasg@math.auth.gr}

\author[Psaroudakis]{Chrysostomos Psaroudakis}
\address{Department of Mathematics, Aristotle University of Thessaloniki, Thessaloniki 54124, Greece}
\email{chpsaroud@math.auth.gr}

\author[Vit\'oria]{Jorge Vit\'oria}
\address{Dipartimento di Matematica “Tullio Levi-Civita”, Università degli Studi di Padova, Torre Archimede, via Trieste 63, 35121 Padova, Italy}
\email{jorge.vitoria@unipd.it}

\thanks{\textbf{Acknowledgements:} P.K. and C.P. were supported in the framework of the H.F.R.I. call “Basic Research Financing (Horizontal Support of all Sciences)” under the National Recovery and Resilience Plan “Greece 2.0” funded by the European Union - NextGenerationEU (H.F.R.I. Project Number: 16785). J.V.~was supported by NextGenerationEU under NRRP, Call PRIN 2022  No.~104 of February 2, 2022 of Italian Ministry of University and Research; Project 2022S97PMY \textit{Structures for Quivers, Algebras and Representations (SQUARE)}}

\begin{abstract}
We propose a new framework for the study of homological properties for (compactly generated) triangulated categories such as regularity, finiteness of global or finitistic dimension, gorensteinness or injective generation and the relation between them. Our approach focuses on distinguished, intrinsically defined, subcategories and our main tool is the new notion of \textit{far-away orthogonality}. We observe that these homological properties generalise previously studied properties on derived categories of modules over rings, and we use the generality of our theory to also examine those same attributes for the homotopy category of injectives and the big singularity category (in the sense of Krause) of an Artin algebra, as well as the derived category of a non-positive differential graded algebra. Finally, using our theory we recover and generalise various results in the theory of recollements of triangulated categories. 
\end{abstract}

\maketitle

\setcounter{tocdepth}{1} \tableofcontents

\section{Introduction and Main Results}

There are many ways of measuring the homological complexity of a category. A first, very classical, approach is to establish a dimension function that associates an integer (or, possibly, $\infty$) to certain objects of the category, usually by measuring the length of a suitable resolution by special objects, and then take a supremum or an infimum of that function. This is the case, for example, when defining the global dimension or the finitistic dimension of an abelian category with enough projectives or injectives, by measuring projective or injective dimensions. A second approach for this computation is to try to reduce the class of objects on which to measure this complexity and still obtain enough data on the whole category. A concrete example of this is the well-known fact that the global dimension of the module category over a ring is, in fact, the supremum of projective/injective dimensions of cyclic modules. Alternatively, if the aim is to measure the complexity of the category, the mathematician that is so inclined might prefer to avoid picking individual objects and work instead with intrinsically defined subcategories and the relations between them. This is the approach we take in this paper to study triangulated categories, since they are, almost by default, a unifying framework for the formalisation of homological phenomena. An example of the type of relations that will pop up is the well-known statement that the finiteness of global dimension of a ring can be detected by an equality between a bounded derived category of modules and a bounded homotopy category of projective (or injective) modules.

Indeed, triangulated categories are central in all areas of mathematics using homological algebra, ranging from representation theory to algebraic geometry or algebraic topology. Examples include derived categories of (quasi)coherent sheaves over schemes, derived categories of modules over rings or differential graded modules over differential graded (dg for short) algebras, homotopy categories of relevant additive categories (such as projective or injective objects), singularity categories, homotopy categories of ring spectra, stable categories of group algebras, among others. We seek to redefine, generalise and verify in well-known categories some attributes that express good homological behaviour of triangulated categories. These are \textit{finite global dimension}, \textit{gorensteinness}, \textit{injective generation}, \textit{finiteness of big finitistic dimension} and \textit{regularity}. We do so intrinsically, using subcategories constructed in a way that tries to minimise or eliminate the choice of distinguished objects.

There have been various approaches to importing homological invariants into the context of triangulated categories, most of which defined in terms of a distinguished class of objects that are not intrinsically defined. We list a few examples.
\begin{itemize}[leftmargin=6mm]
\item In \cite{bondal_van-den-bergh} a notion of regularity is introduced, requiring the existence of a strong generator.
\item Similarly, ideas regarding strong generators allow for a definition of small finitistic dimension in \cite{krause}.
\item In \cite{kontsevich} and \cite{toen} a notion of smooth dg category is explored by requiring the perfectness of a certain bimodule.
\item In \cite{jin} and \cite{kuznetsov_shinder} the notion of Gorenstein dg algebra and Gorenstein dg category is discussed via certain subcategories suitably generated by the given data. 
\item In \cite{rickard}, a relation is established between finiteness of finitistic dimension of a finite-dimensional algebra and the generation of the derived category by injective objects.
\end{itemize}
The approaches listed require being given some data about an object or a subcategory to deduce something about the whole category. We propose a top-down approach: knowing the category, produce the subcategories that can tell us something intrinsic about it.

It has long been observed that in the derived category of a ring many subcategories of interest occur intrinsically \cite{koenig, rickard_morita}. First come the perfect complexes, which arise as the subcategory of compact objects. Then, the bounded (above or below) complexes, which can be characterised in terms of the compacts and lastly the bounded (below) complexes of injectives or the bounded (above) complexes of projectives which can be characterised in terms of the bounded complexes. This network of subcategories, and the corresponding analogues in algebraic geometry, has been studied through the new technique of completion with respect to a metric \cite{neeman_metrics}. The appropriate framework for this kind of study is that of (weakly) approximable triangulated categories (see, for example, \cite{CanonacoNeemanStellariHaesemeyer} and \cite{CanonacoNeemanStellari}), which demands a weak approximation data consisting a single compact generator and a t-structure. In this paper, we do not ask our categories to be weakly approximable, taking instead a different approach via what we call \textit{far-away orthogonals}, recovering a concept that has appeared before, namely in \cite{orlov2}. We refer to Sections \ref{section:faraway} and \ref{distinguished subcategories} for the relevant definitions. 


%



For every intrinsically defined subcategory, each one of its far-away orthogonals is also intrinsic, allowing us to define a whole collection of distinguished subcategories that we decorate with suitable superscripts and subscripts inspired by what these subcategories are in the derived category of rings. The following table illustrates this idea in the derived category $\mathsf{D}(R)$ of right modules over a ring $R$.
\begin{table}[H]
    \caption{Summary of subcategories}
      \centering
        \begin{tabular}{cc}
            \hline Subcategory of $\mathcal{T}$ & Subcategory of $\mathsf{D}(R)$ \\ 
            \hline $\mathcal{T}^+$ & $\mathsf{D}^+(R)$ \\ 
            $\mathcal{T}^-$ & $\mathsf{D}^-(R)$ \\ 
            $\mathcal{T}^{\mathsf{b}}$ & $\mathsf{D}^{\mathsf{b}}(R)$ \\ 
            $\mathcal{T}^{\mathsf{b}}_{p}$ & $\mathsf{K}^{\mathsf{b}}(\Proj R)$ \\ 
            $\mathcal{T}^{\mathsf{b}}_i$ & $\mathsf{K}^{\mathsf{b}}(\Inj R)$\\
            $\mathcal{T}^{\mathsf{b}}_c$ & $\mathsf{D}^{\mathsf{b}}_{\smod R}(\Mod R)$ ($R$ Noether algebra) 
        \end{tabular}
        \end{table}

Following this, on the basis of how these subcategories relate to each other, we define the previously mentioned homological attributes, making sure that they adequately generalise the notions existing already in the literature. For example, we will say that a triangulated category $\mathcal{T}$ has finite global dimension if $\mathcal{T}^{\mathsf{b}}=\mathcal{T}^{\mathsf{b}}_p$. Here is another table exemplifying how our properties, which are determined by equalities of distinguished subcategories, are interpreted (in given contexts) in the derived category of right $R$-modules.

\begin{table}[H]
\caption{Homological properties}
      \centering
            \begin{tabular}{ccc}
            \hline Property& Identification & What it means for $\mathsf{D}(R)$ \\ 
         \hline          Finite global dimension & $\mathcal{T}^{\mathsf{b}}=\mathcal{T}^{\mathsf{b}}_p$ & $\gd R<\infty$ \\ 
           Gorenstein & $\mathcal{T}^{\mathsf{b}}_p=\mathcal{T}^{\mathsf{b}}_i$ & $R$ Iwanaga-Gorenstein (if $R$ Artin algebra) \\ 
            Injective generation & $\mathsf{Loc}(\mathcal{T}^{\mathsf{b}}_i)=\mathcal{T}$ & Injectives generate for $R$\\ 
         Finite finitistic dimension &  $(\mathcal{T}^{\mathsf{b}}_i)^{\perp}\cap \mathcal{T}^+=\{0\}$ & $\Findim R<\infty$ (if $R$ Artin algebra) \\
 Regular & $\mathcal{T}^{\mathsf{c}}=\mathcal{T}^{\mathsf{b}}_c$ & $R$ regular (if $R$ commutative noetherian) \\ 
        \end{tabular}
 \end{table}
 
The attributes listed above satisfy the expected causality relations (see Theorem \ref{theorem1}):
\[\begin{xymatrix}{
\gd \T<\infty\ar@{=>}[r]& \T\ {\rm Gorenstein}\ar@{=>}[d]\ar@{=>}[r]&\T\ {\rm generated\ by\ injectives}\ar@{=>}[dl]^{}\\ &\Findim \T<\infty,&
}\end{xymatrix}\]
where the diagonal implication holds if we assume a sort of \textit{noetherianness} of the triangulated category. This is explained in detail following the theorem. 

In this paper we develop the theory behind these homological attributes while, simultaneously, verifying their occurrence in various triangulated categories. In particular we use them to study in detail the following triangulated categories:
\begin{itemize}[leftmargin=6mm]
\item $\mathsf{D}(R)$, the derived category of right $R$-modules, for a ring $R$;
\item $\mathsf{K}(\Inj\Lambda)$, the homotopy category of injective right $\Lambda$-modules, for $\Lambda$ an Artin algebra; 
\item $\mathsf{K}_{\mathsf{ac}}(\Inj \Lambda)$, the homotopy category of acyclic complexes of injective right $\Lambda$-modules, for $\Lambda$ an Artin algebra (often referred to as the \textit{big singularity category of $\Lambda$}, and introduced in \cite{krause2});
\item $\mathsf{D}(\Gamma)$, the derived category of right dg modules over a proper connective dg algebra $\Gamma$.
\end{itemize}

\noindent \textbf{Structure of the paper.} In Section~\ref{section:faraway} we collect preliminary notions and results that are used throughout the paper. In Subsection~\ref{far away orthogonality} we introduce and study far away orthogonal subcategories. In Section~\ref{distinguished subcategories} we introduce the main new intrinsic subcategories of triangulated categories that will be the focus of our study, see Definition~\ref{main_definition}. We immediately proceed to compute these subcategories for the categories listed in the paragraph above (see Proposition~\ref{examples}, Corollary \ref{T^b_p for dg algebras} and Proposition~\ref{examples2}). Along the way, we discuss our distinguished subcategories in the presence of silting objects, and how these silting objects induce bounded t-structures or co-t-structures in suitable subcategories (see Theorem \ref{silting far-away generating}). In Section~\ref{regulartriangcategories}, we set up in Definition~\ref{homological} the properties that we aim to study in the context of triangulated categories: finiteness of global dimension, gorensteinness, generation by injectives, and finiteness of finitistic dimension. We then study their causality relations (see Theorem \ref{theorem1}) and analyse their occurrence in our categories of interest. In Section \ref{Sec5}, we assume that a given triangulated category is $k$-linear over a commutative noetherian ring $k$ in order to define regularity with respect to $k$.  This requires a new subcategory which is intrinsic to the given $k$-linear structure, the bounded finite objects, as well as a quotient category that we call singularity category. Once again, we study how to read off the property of being regular in our categories of interest, and we also compare our definition with other notions of regularity present in the literature. Finally, in Section \ref{app to rec}, we discuss applications of our theory to recollements of triangulated categories. In particular, we generalise results from \cite{angeleri_koenig_liu_yang} concerning restricting recollements of triangulated categories to distinguished subcategories, and we end with a discussion of how our homological properties behave in a recollement, generalising previous work of \cite{angeleri_koenig_liu_yang} about the global dimension, of \cite{cummings} about injective generation and of \cite{chen_xi} about the finitistic dimension.


\begin{connot}
 All subcategories are strict and full. For a subcategory $\X$ of an additive category $\mathcal{A}$, we denote by $\mathsf{Add}(\X)$ (respectively, $\mathsf{Prod}(\X)$) the subcategory formed by the summands of the existing coproducts (respectively, products) of objects in $\X$. For an additive category $\mathcal{A}$ we write $\mathsf{K}(\mathcal{A})$ to be the homotopy category of (cochain) complexes in $\mathcal{A}$, and we consider the following subcategories: $\mathsf{K}^+(\mathcal{A})$ the homotopy category of complexes $X^{\bullet}$ for which $X^n=0$ for $n\ll0$; $\mathsf{K}^-(\mathcal{A})$ for the homotopy category of complexes $X^{\bullet}$ for which $X^n=0$ for $n\gg0$; and $\mathsf{K}^{\mathsf{b}}(\mathcal{A})$ for $\mathsf{K}^+(\mathcal{A})\cap\mathsf{K}^-(\mathcal{A})$. If $\mathcal{A}$ is abelian, we write $\mathsf{D}^{\?}(\mathcal{A})$ for the essential image of $\mathsf{K}^{\?}(\mathcal{A})$ under the quotient of $\mathsf{K}(\mathcal{A})\longrightarrow \mathsf{K}(\mathcal{A})/ \mathsf{K}_{\mathsf{ac}}(\mathcal{A})=\mathsf{D}(\mathcal{A})$, where $\mathsf{K}_{\mathsf{ac}}(\mathcal{A})$ denotes the subcategory of $\mathsf{K}(\mathcal{A})$ formed by the acyclic complexes and $\?$ is one of the symbols $+,-$ or $b$. If $\X$ and $\Y$ are subcategories of $\T$, we denote by $\X\ast\Y$ the subcategory of $\T$ whose objects are those $t$ for which there is $x$ in $\X$, $y$ in $\Y$ and a triangle in $\T$ of the form
\[x\longrightarrow t\longrightarrow y\longrightarrow x[1].\]
We denote by $\mathsf{thick}(\X)$ the smallest thick subcategory (i.e.~smallest subcategory closed under extensions and summands) containing $\X$. Given an object $x$ in a triangulated category $\T$ and an interval $I$ of integers, we consider the subcategories
$$x^{\perp_{I}}\coloneqq\{y\in\T\colon \Hom_\T(x,y[i])=0,\forall i \in I\}$$
where $I$ sometimes is represented by symbols $>0$ or $\leq 0$ with the obvious meaning. The symbol $\perp$ stands simply for $\perp_0$. For a ring $R$, we write $\Mod{R}$ for the category of all right $R$-modules and we consider its subcategories $\Proj{R}$ of projective modules, $\proj{R}$ of finitely generated projective modules, $\Inj R$ of injective modules, and $\smod{R}$ of finitely presented modules. We also use $\mathsf{D}^{\?}(R)$ and $\mathsf{K}^{\?}(R)$ as a short for $\mathsf{D}^{\?}(\Mod{R})$ and $\mathsf{K}^{\?}(\Mod{R})$ for $\?$ in $\{+,-,\mathsf{b}\}$. Unless otherwise stated, $k$ will always denote a commutative noetherian ring. By an Artin algebra (respectively, a Noether algebra) we mean a finitely generated algebra over a commutative artinian (respectively, noetherian) ring.
\end{connot} 

\section{Far-Away Orthogonality}
\label{section:faraway}

In this section we introduce far-away orthogonal subcategories. We start with some preliminaries results that we need in the sequel.

\subsection{Compactly generated triangulated categories}
\label{subsection:prelim}
An object $x$ in a triangulated category $\T$ is said to be \textbf{compact} if $\mathsf{Hom}_\T(x,-)$ commutes with all existing coproducts. Compact objects form a thick subcategory of $\T$, which we denote by $\mathcal{T}^{\mathsf{c}}$. If $\T$ is cocomplete, $\mathcal{T}^\mathsf{c}$ is skeletally small and $(\Tc)^\perp=0$, then we say that $\T$ is \textbf{compactly generated}.

\begin{exmp}\label{first example}
Here are some examples of compactly generated triangulated categories.
\begin{enumerate}
\item If $\Lambda$ is a self-injective Artin algebra, then the stable module category $\uMod\Lambda$ is compactly generated, and its subcategory of compact objects is the subcategory obtained as the stabilisation of finitely presented modules $\umod \Lambda$, see for instance \cite[Lemma 3.9]{beligiannis2}.
\item By \cite[Proposition 6.4]{bokstedt_neeman}, the derived category $\mathsf{D}(R)$ of the category of modules over an arbitrary ring $R$ is compactly generated, and its subcategory of compact objects is $\mathsf{K}^{\mathsf{b}}(\proj{R})$. 
\item More generally, the derived category $\mathsf{D}(\Gamma)$ of a differential graded algebra $\Gamma$ is compactly generated, see \cite[Subsection 4.2]{keller0}. The subcategory of compact objects is $\mathsf{thick}(\Gamma)$, which is sometimes denoted by $\mathsf{per}(\Gamma)$.
\item Let $A$ be a noetherian ring. By \cite[Corollary 4.3]{krause2} there is a recollement (see Section \ref{app to rec} for a definition):
\begin{equation}\label{rec}
\begin{tikzcd}
\mathsf{K}_{\mathsf{ac}}(\Inj A) \arrow[rr, "\mathsf{I}"] &  & \mathsf{K}(\Inj A) \arrow[rr, "\mathsf{Q}"] \arrow[ll, "\mathsf{I}_{\rho}", bend left] \arrow[ll, "\mathsf{I}_{\lambda}"', bend right] &  & \mathsf{D}(A) \arrow[ll, "\mathsf{Q}_{\lambda}"', bend right] \arrow[ll, "\mathsf{Q}_{\rho}", bend left]
\end{tikzcd}
\end{equation}
where the functor $\mathsf{Q}$ is just the restriction of the projection functor $\mathsf{K}(A)\longrightarrow \mathsf{D}(A)$ to  $\mathsf{K}(\Inj A)$. By \cite[Proposition 2.3]{krause2}, the homotopy category $\mathsf{K}(\Inj A)$ is compactly generated and 
\[\mathsf{K}(\Inj A)^{\mathsf{c}}\simeq \mathsf{Q}_\rho(\mathsf{D}^{\mathsf{b}}(\smod A)).\]
As a consequence, we get that  $\mathsf{K}_{\mathsf{ac}}(\Inj A)$ (sometimes called the \textit{big singularity category} of $A$), is compactly generated, and that its compact objects are summands of objects in the essential image of $\mathsf{D}^{\mathsf{b}}(\smod A)$ under $\mathsf{I}_{\lambda}\mathsf{Q}_\rho$. This essential image turns out to be, by standard Bousfield localisation arguments, equivalent to the \textit{singularity category of $A$} \cite{buchweitz, orlov} defined as follows:
$$\mathsf{D}^{\mathsf{b}}(\smod A)/\mathsf{D}(A)^{\mathsf{c}}\eqqcolon\mathsf{D}_{\mathsf{sg}}(A).$$
\end{enumerate}
\end{exmp}

A useful tool in compactly generated triangulated categories is the existence of duals of compact objects. Recall that given a compact object $t$ of a compactly generated triangulated category $\mathcal{T}$, there is a unique object of $\mathcal{T}$, up to isomorphism, representing the functor 
\[
\mathsf{Hom}_{\mathbb{Z}}(\mathsf{Hom}_{\mathcal{T}}(t,-),\mathbb{Q}/\mathbb{Z})\colon \mathcal{T}^{\mathsf{op}}\rightarrow \mathsf{Ab},
\]
which we denote by $t^*$. This object is often called the \textbf{Brown-Comenetz dual of $t$}. Similary, if $\X$ is a subcategory of $\T^c$, we denote by $\X^{\ast}$ the subcategory of $\T$ whose objects are those $x^\ast$ for $x$ in $\X$. Note that, should the category $\T$ be endowed with a $k$-linear structure over a commutative noetherian ring $k$, it is natural to consider replacing $\mathsf{Hom}_\mathbb{Z}(-,-)$ with $\mathsf{Hom}_k(-,-)$ and $\mathbb{Q}/\mathbb{Z}$ with an injective cogenerator in $\Mod{k}$ in the formula above. In particular, when $k$ is a field, one often uses $k$ as a substitute for both $\mathbb{Z}$ and $\mathbb{Q}/\mathbb{Z}$.

\begin{rem}\label{Nak}
Brown-Comenetz duals occur quite naturally in representation theory. The following follows from the work of Keller in \cite{keller0} (we refer also to \cite[Subsection 1.2]{jin}). Let $k$ be a field and suppose that $\Gamma$ is a differential graded $k$-algebra (dg $k$-algebra, or dg algebra over $k$, for short). The functor
$$\nu\colon\mathsf{D}(\Gamma)\rightarrow \mathsf{D}(\Gamma)\ \ \ \ \ \ t\mapsto t\otimes^{\mathbb{L}}_\Gamma \mathbb{R}\mathsf{Hom}_k(\Gamma,k)$$
is called the \textbf{Nakayama functor} and it can be shown to send a compact object $t$ in $\mathsf{D}(A)$ to its Brown-Comenetz dual. The isomorphism inherent to the Brown-Comenetz duality is known, in this setting, as the \textit{Auslander-Reiten formula}.  We say that $\Gamma$ is \textbf{proper} if $\mathsf{dim}_k\oplus_{n\in\mathbb{Z}}H^n(\Gamma)<\infty$. In the case of a proper dg $k$-algebra, $\nu$ restricts to an equivalence of subcategories of $\mathsf{D}(\Gamma)$ of the form $\mathsf{D}(\Gamma)^{\mathsf{c}}=\mathsf{thick}(\Gamma)\rightarrow \mathsf{thick}( \nu \Gamma)$. 
\end{rem}

Finally, we want to write a brief reminder on t-structures and co-t-structures. 

\begin{defn}
A pair of subcategories $(\U,\V)$ in $\T$ is said to be a \textbf{torsion pair} if $\U^\perp=\V$, ${}^\perp \V=\U$ and $\U\ast\V=\T$. Such a torsion pair is said to be a \textbf{t-structure} if $\U[1]\subseteq \U$ or a \textbf{co-t-structure} if $\U[-1]\subseteq \U$. 
\end{defn}

It is well-known (see \cite{beilinson_bernstein_deligne}) that if $(\U,\V)$ is a t-structure then $\mathcal{H}\coloneqq\U\cap\V[1]$ is an abelian category, called the \textbf{heart} of $(\U,\V)$. If $(\U,\V)$ is a co-t-structure, the intersection $\U\cap\V[-1]$ is called the \textbf{co-heart} of $(\U,\V)$, and in general this is only an additive category. 

\subsection{Broadness}

In this subsection we introduce some nomenclature that will be useful later on.

\begin{defn}
We say that a thick subcategory $\mathcal{S}$ of a triangulated category $\mathcal{T}$ is 
\begin{itemize}
\item[(i)] \textbf{$\prod$-broad} (respectively, \textbf{$\Sigma$-broad})  if it is closed under existing self-products (respectively, self-coproducts), i.e. for every object $x$ in $\mathcal{S}$ and any set $I$, the product $x^{I}$ (respectively, the coproduct $x^{(I)}$) lies in $\mathcal{S}$; 
\item[(ii)] \textbf{broad} if it is both $\Sigma$-broad and $\prod$-broad.
\end{itemize}
\end{defn}

Note that any localising (respectively, colocalising) subcategory of a cocomplete (respectively, complete) triangulated category is $\Sigma$-broad (respectively, $\prod$-broad), but the converse is clearly not true - as the following examples show.

\begin{exmp}
If we consider $\mathcal{T}=\mathsf{D}(R)$ for some ring $R$, then we have that 
\begin{itemize}
\item[(i)] $\mathsf{K}^{\mathsf{b}}(\Proj{R})$ and $\mathsf{K}^-(\Proj{R})$ are $\Sigma$-broad subcategories;
\item[(ii)] $\mathsf{K}^{\mathsf{b}}(\Inj R)$ and $\mathsf{K}^+(\Inj R)$ are $\prod$-broad subcategories;
\item[(iii)] $\mathsf{D}^{\mathsf{b}}(R)$, $\mathsf{D}^+(R)$ and $\mathsf{D}^-(R)$ are broad subcategories.
\end{itemize}
Note, however, that none of the above subcategories of $\mathsf{D}(R)$ are neither localising nor colocalising.
\end{exmp}

\begin{lem}\label{broad generation}
Let $\mathcal{X}$ be a set of objects in a triangulated category $\mathcal{T}$. 
\begin{enumerate}
\item[\textnormal{(i)}] If $\T$ is cocomplete, then $\mathsf{thick}(\cup_{x\in\mathcal{X}}\mathsf{Add}(x))$ is the smallest $\Sigma$-broad subcategory of $\mathcal{T}$ containing $\mathcal{X}$, and we denote it by $\mathsf{broad}_{\Sigma}(\mathcal{X})$.
\item[\textnormal{(ii)}] If $\T$ is complete, then $\mathsf{thick}(\cup_{x\in\mathcal{X}}\mathsf{Prod}(x))$ is the smallest $\prod$-broad subcategory of $\mathcal{T}$ containing $\mathcal{X}$, and we denote it by $\mathsf{broad}_{\prod}(\mathcal{X})$.
\end{enumerate}
\end{lem}
\begin{proof}
Any $\Sigma$-broad subcategory containing $x$ must contain $\cup_{x\in\mathcal{X}}\mathsf{Add}(x)$ and, hence, $\mathsf{thick}(\cup_{x\in\mathcal{X}}\mathsf{Add}(x))$. It then suffices to show that the latter is a $\Sigma$-broad subcategory. It is thick by construction, and so it remains to show that for any object $y$ in $\mathsf{thick}(\cup_{x\in\mathcal{X}}\mathsf{Add}(x))$, $y^{(I)}$ is also in $\mathsf{thick}(\cup_{x\in\mathcal{X}}\mathsf{Add}(x))$ for any set $I$. Now $y$ is a summand of an object that admits a finite filtration by objects in appropriate shifts of $\cup_{x\in\mathcal{X}}\mathsf{Add}(x)$. Since coproducts preserve triangles and splits monomorphisms in $\mathcal{T}$, $y^{(I)}$ is a summand of another object admiting such a filtration, thus showing that it lies in $\mathsf{thick}(\cup_{x\in\mathcal{X}}\mathsf{Add}(x))$. Assertion (ii) can be proved dually.
\end{proof}

\begin{exmp}
Let $R$ be a ring, $E$ an injective cogenerator in $\Mod{R}$ and $\mathcal{T}=\mathsf{D}(R)$. It is easy to see that $\mathsf{K}^{\mathsf{b}}(\Proj{R})=\mathsf{broad}_{\Sigma}(R)$ and that $\mathsf{K}^{\mathsf{b}}(\Inj R)=\mathsf{broad}_{\prod}(E)$. If $R$ is right noetherian, then every injective right $R$-module is a coproduct of indecomposable injective modules and, in particular, we have $\mathsf{Add}(E)=\mathsf{Prod}(E)$. Thus, if $R$ is right noetherian, we have that $\mathsf{K}^{\mathsf{b}}(\Inj R)=\mathsf{broad}_{\prod}(E)=\mathsf{broad}_{\Sigma}(E)$ is a broad subcategory. Dually, if $R$ is artinian, then the subcategory of projective right $R$-modules is closed under products (by Chase's theorem) and, therefore, $\mathsf{K}^{\mathsf{b}}(\Proj{R})$ is broad.
\end{exmp}

An easy application of the lemma above gives us the following useful observation.

\begin{cor}\label{broad Db}
Let $\Lambda$ be an Artin algebra and $\mathcal{S}_\Lambda$ its subcategory of simple objects. Then we have 
$$\mathsf{D}^{\mathsf{b}}(\Mod\Lambda)=\mathsf{broad}_\Sigma(\mathsf{D}^{\mathsf{b}}(\smod\Lambda))=\mathsf{broad}_\Sigma(\smod\Lambda)=\mathsf{broad}_\Sigma(\mathcal{S}_\Lambda).$$
\end{cor}
\begin{proof}
Any $\Lambda$-module is a finite extension of semi-simple modules (consider the radical series, see \cite[Proposition 3.1]{auslander_reiten_smalo} and \cite[Corollary 15.18]{anderson_fuller}). Writing $\Add\mathcal{S}_\Lambda$ for the subcategory of semi-simple $\Lambda$-modules, we have 
$$\mathsf{D}^{\mathsf{b}}(\Mod\Lambda)=\mathsf{thick}(\Mod \Lambda) = \mathsf{thick}(\Add\mathcal{S}_\Lambda).$$ 
Now, note that since there are finitely many simple $\Lambda$-modules, in fact we have that 
\[
\mathsf{thick}(\Add\mathcal{S}_\Lambda)=\mathsf{thick}(\cup_{x\in\mathcal{S}_\Lambda}\Add x)=\mathsf{broad}_\Sigma(\mathcal{S}_\Lambda).\qedhere
\]

\end{proof}

\subsection{Far-away orthogonality}\label{far away orthogonality}
We discuss orthogonality in triangulated categories for \textit{sufficiently large} or \textit{sufficiently small} integers, and we observe that those relations give intrinsic descriptions of various subcategories of interest. Given two objects $x$ and $y$ in a triangulated category $\mathcal{T}$ we write 
$$\mathsf{Hom}_{\mathcal{T}}(x,y[\gg])=0\ \ \rm{(respectively,}\ \  \mathsf{Hom}_{\mathcal{T}}(x,y[\ll])=0)$$ 
if there is $N$ in $\mathbb{Z}$ such that for all $n>N$ (respectively for all $n<N$), we have $\mathsf{Hom}_{\mathcal{T}}(x,y[n])=0$. For a class $\mathcal{X}$ of objects in $\mathcal{T}$, we define the following subcategories of $\mathcal{T}$: 
\[
\begin{aligned}
\mathcal{X}^{\perp_{\gg}}&\coloneqq&\{t\in\mathcal{T}: \   \forall x\in\mathcal{X}, \mathsf{Hom}_{\mathcal{T}}(x,t[\gg])=0 \}\ \ \ \ \ 
\mathcal{X}^{\perp_{\ll}}&\coloneqq&\{t\in\mathcal{T}: \  \forall x\in\mathcal{X}, \mathsf{Hom}_{\mathcal{T}}(x,t[\ll])=0 \}\\
^{\perp_\gg}\mathcal{X}&\coloneqq&\{t\in\mathcal{T}: \ \forall x\in\mathcal{X}, \mathsf{Hom}_{\mathcal{T}}(t,x[\gg])=0 \}\ \ \ \ \ 
^{\perp_\ll}\mathcal{X}&\coloneqq&\{t\in\mathcal{T}: \ \forall x\in\mathcal{X}, \mathsf{Hom}_{\mathcal{T}}(t,x[\ll])=0 \}
\end{aligned}
\]

\[
\begin{aligned}
\mathcal{X}^{\padova}&\coloneqq&\mathcal{X}^{\perp_{\ll}}\cap\mathcal{X}^{\perp_{\gg}}\ \ \ \ \ 
{}^{\padova}\mathcal{X}&\coloneqq& {}^{\perp_\ll}\mathcal{X}\cap {}^{\perp_{\gg}}\mathcal{X}
\end{aligned}
\]
If $\X$ is made of a single object, i.e.~if $\X=\{x\}$, we write $x^{\?}$ instead of $\X^{\?}$, for any symbol $\?$ in $\{\perp_\gg,\perp_\ll,\padova\}$.
Before we discuss the relevance of the use of these orthogonals, let us extract some easy structural properties.

\begin{lem} \label{triangulated_subcategories} 
Let $\mathcal{X}$ and $\mathcal{Y}$ be classes of objects in $\mathcal{T}$, and let $\?$ be one of the symbols $\{\perp_{\ll},\perp_\gg, \padova\}$.
\begin{enumerate}
\item The subcategories $\mathcal{X}^{\?}$ are $\prod$-broad and the subcategories ${}^{\?}\mathcal{X}$ are $\Sigma$-broad.
\item $\mathcal{X}\subseteq {}^{\?}(\mathcal{X}^{\?})$ and $\mathcal{X}\subseteq ({}^{\?}\mathcal{X})^{\?}$.
\item $\mathcal{X}^{\?}=({}^{\?}(\mathcal{X}^{\?}))^{\?}$ and ${}^{\?}\mathcal{X}={}^{\?}(({}^{\?}\mathcal{X})^{\?})$
\item $(\mathcal{X}\ast\mathcal{Y})^{\?}=\mathcal{X}^{\?}\cap \mathcal{Y}^{\?}$ and ${}^{\?}(\mathcal{X}\ast\mathcal{Y})={}^{\?}\mathcal{X}\cap {}^{\?}\mathcal{Y}$.
\item[\textnormal{(v)}] If $\mathcal{X}$ is a set of objects, then we have
\begin{enumerate}
\item[\textnormal{(a)}] $\mathcal{X}\subseteq \mathsf{broad}_{\Sigma}(\mathcal{X})\subseteq {}^{\?}(\mathcal{X}^{\?})$ and $\mathcal{X}\subseteq \mathsf{broad}_{\prod}(\mathcal{X})\subseteq ({}^{\?}\mathcal{X})^{\?}$.
\item[\textnormal{(b)}] $\mathcal{X}^{\?}=\mathsf{broad}_{\Sigma}(\mathcal{X})^{\?}=\mathsf{thick}(\mathcal{X})^{\?}$ and ${}^{\?}\mathcal{X}={}^{\?}\mathsf{broad}_{\prod}(\mathcal{X})={}^{\?}\mathsf{thick}(X)$.
\end{enumerate}
\end{enumerate}
\end{lem}
\begin{proof}
  (i)  Subcategories of the form $\mathcal{X}^{\?}$ are clearly closed under shifts and summands. If we consider a triangle 
    $$u\rightarrow v\rightarrow w\rightarrow u[1]$$
    in $\mathcal{T}$ with $u$ and $w$ in $\mathcal{X}^{\?}$ and we apply to it the cohomological functor $\mathsf{Hom}_\mathcal{T}(x,-)$, for $x$ an object of $\mathcal{X}$, using the fact that $\mathcal{X}^{\?}$ is closed under shifts, it follows that $v$ lies in $\mathcal{X}^{\?}$. Finally, $\mathcal{X}^{\?}$ is closed under self-products since $\mathsf{Hom}_\mathcal{T}(x,-)$ commutes with products. The other case is dual.
    
(ii) This follows straight from the definition.
    
    (iii) We show that $\mathcal{X}^{\?}=({}^{\?}(\mathcal{X}^{\?}))^{\?}$; the other equality is dual. By the second inclusion in (i) applied to the class $\mathcal{X}^{\?}$, we have $\mathcal{X}^{\?}\subseteq ({}^{\?}(\mathcal{X}^{\?}))^{\?}$. Applying $(-)^{\?}$ to the first inclusion in (ii) we also get that $({}^{\?}(\mathcal{X}^{\?}))^{\?}\subseteq \mathcal{X}^{\?}$, thus proving the equality.

    (iv) We have $\mathcal{X}\subseteq \mathcal{X}\ast \mathcal{Y}$ and $\mathcal{Y}\subseteq \mathcal{X}\ast \mathcal{Y}$, implying that $(\mathcal{X}\ast \mathcal{Y})^{\?}\subseteq \mathcal{X}^{\?}$ and $(\mathcal{X}\ast \mathcal{Y})^{\?}\subseteq \mathcal{Y}^{\?}$. Consequently, we get $(\mathcal{X}\ast \mathcal{Y})^{\?}\subseteq \mathcal{X}^{\?}\cap \mathcal{Y}^{\?}$. For an object $t$ in $\mathcal{X}\ast \mathcal{Y}$ there is a triangle $x\rightarrow t\rightarrow y\rightarrow x[1]$ with $x$ in $\mathcal{X}$ and $y$ in $\mathcal{Y}$. Then, for an object $s$ of $\mathcal{X}^{\?}\cap \mathcal{Y}^{\?}$, using the long exact sequence 
    \[
    \dots\rightarrow \mathsf{Hom}_{\mathcal{T}}(y,s[n])\rightarrow \mathsf{Hom}_{\mathcal{T}}(t,s[n])\rightarrow \mathsf{Hom}_{\mathcal{T}}(x,s[n])\rightarrow \mathsf{Hom}_{\mathcal{T}}(y,s[n+1])\rightarrow \cdots
    \]
    we infer that $s$ lies in $t^{\?}$ for every $t$ in $\mathcal{X}\ast \mathcal{Y}$, i.e.~$s$ lies in $(\mathcal{X}\ast \mathcal{Y})^{\?}$. 
    
    (v)(a) These inclusions follow straight from the definition and from (i) and (ii).
    
    (v)(b) We show that $\mathcal{X}^{\?}=\mathsf{broad}_{\Sigma}(\mathcal{X})^{\?}=\mathsf{thick}(\mathcal{X})^{\?}$; the other statement is dual. Since, by (v)(a), we have $\mathsf{broad}_{\Sigma}(\mathcal{X})\subseteq {}^{\?}(\mathcal{X}^{\?})$, we get the equality $\mathcal{X}^{\?}=({}^{\?}(\mathcal{X}^{\?}))^{\?}\subseteq \mathsf{broad}_{\Sigma}(\mathcal{X})^{\?}$, by (iii). For the converse inclusion, observe that since $\mathcal{X}\subseteq \mathsf{broad}_{\Sigma}(\mathcal{X})$, we have that $\mathsf{broad}_{\Sigma}(\mathcal{X})^{\?}\subseteq \mathcal{X}^{\?}$. This proves the first equality. Now observe that clearly $\mathsf{thick}(\mathcal{X})^{\?}\subseteq \mathcal{X}^{\?}$ and, furthermore, since every object in $\mathsf{thick}(\mathcal{X})$ is a summand of an object admitting a finite filtration by shifts of objects of $\mathcal{X}$, it follows that any object in $\mathcal{X}^{\?}$ must also lie in $\mathsf{thick}(\mathcal{X})^{\?}$, thus concluding that $\mathcal{X}^{\?}=\mathsf{thick}(\mathcal{X})^{\?}$. 
\end{proof}

\begin{exmp}\label{exmp artin}
Let $\Lambda$ be an Artin algebra. Recall that from Corollary \ref{broad Db} 
$$\mathsf{broad}_\Sigma(\smod\Lambda)=\mathsf{thick}(\Mod\Lambda)=\mathsf{D}^\mathsf{b}(\Mod\Lambda).$$ 
It then follows from Lemma \ref{triangulated_subcategories}(v)(b) that $(\smod\Lambda)^{\?}=(\Mod\Lambda)^{\?}$ for any symbol $\?$ in $\{\perp_\ll,\perp_\gg,\padova\}$.
\end{exmp}

%

It will be useful in the following to describe how far-away orthogonality behaves under adjoint pairs and fully faithful functors.

\begin{lem}\label{adjoints and far-away orthogonality} \label{fully_faithful_functor}
    Let $\mathsf{F}$ be a triangle functor $\mathsf{F}\colon\mathcal{D}\rightarrow \mathcal{T}$ and let $\mathcal{X}$ and $\mathcal{Y}$ be classes of objects in $\mathcal{D}$ and $\mathcal{T}$ respectively such that $\mathsf{F}(\X)\subseteq \Y$. For every $\?$ in $\{\perp_{\ll},\perp_\gg, \padova\}$, the following statements hold\textnormal{:} 
    \begin{enumerate}
        \item if $\mathsf{F}$ admits a right adjoint $\mathsf{G}$, then $\mathsf{G}(\mathcal{Y}^{\?})\subseteq \mathcal{X}^{\?}$;
        \item if $\mathsf{F}$ admits a left adjoint $\mathsf{G}$, then $\mathsf{G}({}^{\?}\mathcal{Y})\subseteq {}^{\?}\mathcal{X}$;
        \item if $\mathsf{F}$ is fully faithful and $\mathsf{F}(x)$ lies in $\mathcal{Y}^{\?}$ for some $x$ in $\mathcal{D}$, then $x$ lies in $\mathcal{X}^{\?}$;
        \item if $\mathsf{F}$ is fully faithful and $\mathsf{F}(x)$ lies in $^{\?}\mathcal{Y}$ for some $x$ in $\mathcal{D}$, then $x$ lies in $^{\?}\mathcal{X}$.
    \end{enumerate}
\end{lem}
\begin{proof}
Statements (i) and (ii) follow directly by the definitions and the adjunctions considered. To prove (iii), assume that $\mathsf{F}(x)$ lies in $\mathcal{Y}^{\?}$ for some $x$ and some $\?$ in $\{\perp_{\gg},\perp_{\ll},\padova\}$. Consider an object $x'$ in $\mathcal{X}$. Since $\mathsf{F}$ is fully faithful, it follows that 
    \[
    \mathsf{Hom}_{\mathcal{D}}(x',x[n])\cong \mathsf{Hom}_{\mathcal{T}}(\mathsf{F}(x'),\mathsf{F}(x)[n])
    \]
for all $n$ in $\mathbb{Z}$. Since $\mathsf{F}(x')$ lies in $\mathcal{Y}$ and $\mathsf{F}(x)$ lies in $\mathcal{Y}^{\?}$, it follows in particular that the right hand side vanishes for appropriate values of $n$ depending on $\?$. In each case it follows that $x$ lies in $x'^{\?}$, as wanted. Assertion (iv) follows similarly.
\end{proof}

Sometimes far-away orthogonal subcategories will come in pairs, as we shall see in the following sections.

\begin{defn}
Given a triangulated category $\mathcal{T}$, a pair $(\mathcal{X},\mathcal{Y})$ of subcategories is called a \textbf{far-away pair} if $\mathcal{X}^{\padova}=\mathcal{Y}$ and $^{\padova}\mathcal{Y}=\mathcal{X}$. Further, a triple $(\mathcal{X},\mathcal{Y},\mathcal{Z})$ of subcategories of $\mathcal{T}$ is called a \textbf{far-away triple} if $(\mathcal{X},\mathcal{Y})$ and $(\mathcal{Y},\mathcal{Z})$ are far-away pairs.
\end{defn}

\begin{rem}\label{generation} It follows directly from Lemma \ref{triangulated_subcategories}(iii) that, for any class of objects $\mathcal{X}$, we can build two far-away pairs
\begin{enumerate}
\item[\textnormal{(i)}] the far-away pair generated by $\mathcal{X}$ given by $({}^{\padova}(\mathcal{X}^{\padova}),\mathcal{X}^{\padova})$, and 
\item[\textnormal{(ii)}] the far-away pair cogenerated by $\mathcal{X}$ given by $({}^{\padova}\mathcal{X},({}^{\padova}\mathcal{X})^{\padova})$. 
\end{enumerate}
Note that by Lemma \ref{triangulated_subcategories}(ii), we have $\X\subseteq {}^{\padova}(\mathcal{X}^{\padova})$ and $\X\subseteq ({}^{\padova}\mathcal{X})^{\padova}$.
\end{rem}

Examples of far-away pairs will be evident after the definitions at the start of the following section.

\section{Distinguished Subcategories}\label{Section 3}
\label{distinguished subcategories}

In the following definition we introduce the main subcategories of study in this paper. They are all built starting from the distinguished subcategory of the compact objects in a triangulated category. Their significance is, therefore, of particular relevance for compactly generated triangulated categories.

\begin{defn} \label{main_definition}
    Let $\mathcal{T}$ be a triangulated category. We define the following distinguished subcategories:
    \begin{itemize}
        \item the subcategory of \textbf{bounded above} objects $\mathcal{T}^-\coloneqq(\mathcal{T}^\mathsf{c})^{\perp_{\gg}}$.
        \item the subcategory of \textbf{bounded below} objects $\mathcal{T}^+\coloneqq(\mathcal{T}^{\mathsf{c}})^{\perp_\ll}$.
        \item the subcategory of \textbf{bounded} objects $\mathcal{T}^{\mathsf{b}}\coloneqq(\mathcal{T}^{\mathsf{c}})^{\padova}=\mathcal{T}^+\cap\mathcal{T}^-$.
          \item the subcategory of \textbf{bounded above projective} objects $\mathcal{T}^-_p\coloneqq{}^{\perp_\ll}(\mathcal{T}^{\mathsf{b}})$.
        \item the subcategory of \textbf{bounded below projective} objects $\mathcal{T}^+_p\coloneqq{}^{\perp_\gg}(\mathcal{T}^{\mathsf{b}})$.
        \item the subcategory of \textbf{bounded projective} objects $\mathcal{T}^{\mathsf{b}}_p\coloneqq$$^{\padova}(\mathcal{T}^{\mathsf{b}})=\mathcal{T}^+_p\cap \mathcal{T}^-_p$. 
        \item the subcategory of \textbf{bounded above injective} objects $\mathcal{T}^-_i\coloneqq(\mathcal{T}^{\mathsf{b}})^{\perp_\gg}$.
        \item the subcategory of \textbf{bounded below injective} objects $\mathcal{T}^+_i\coloneqq(\mathcal{T}^{\mathsf{b}})^{\perp_\ll}$.
        \item the subcategory of \textbf{bounded injective} objects $\mathcal{T}^{\mathsf{b}}_i\coloneqq(\mathcal{T}^{\mathsf{b}})^{\padova}=\mathcal{T}^+_i\cap \mathcal{T}^-_i$.
    \end{itemize}
\end{defn}

Note that, for any choice of $\?$ in $\{+,-,\mathsf{b}\}$, the subcategories $(\mathcal{T}^{\mathsf{c}})^{\?}$ are broad, the subcategories $\mathcal{T}_i^{\?}$ are $\prod$-broad, and the subcategories $\mathcal{T}_p^{\?}$ are $\Sigma$-broad. As promised, the following lemma gives us some examples of far-away pairs. 

\begin{lem} \label{common facts of the subcategories}
Let $\mathcal{T}$ be a complete and cocomplete triangulated category. The following statements hold. 
\begin{enumerate}
\item $(\mathcal{T}^{\mathsf{b}}_p,\mathcal{T}^{\mathsf{b}})$ is the far-away pair generated by $\T^{\mathsf{c}}$; in particular we have $\T^{\mathsf{c}}\subseteq \T^{\mathsf{b}}_p$.
\item If $\mathcal{T}$ is compactly generated, then the far-away pair cogenerated by $(\T^{\mathsf{c}})^\ast$ coincides with $(\T^{\mathsf{b}},\T^{\mathsf{b}}_i)$. In particular, we have a far-away triple of the form $(\mathcal{T}^{\mathsf{b}}_p,\mathcal{T}^{\mathsf{b}},\T^{\mathsf{b}}_i)$.
\end{enumerate}
If $\Tc=\mathsf{thick}(M)$ for some $M$, then $M^{\padova}=\Tb={}^{\padova}(M^\ast)$, ${}^{\padova}(M^{\padova})=\Tb_p$ and $({}^{\padova}(M^\ast))^{\padova}=\Tb_i$.
\end{lem}
\begin{proof}
(i) This follows by the definitions $\T^{\mathsf{b}}\coloneqq(\T^{\mathsf{c}})^{\padova}$ and $\mathcal{T}^{\mathsf{b}}_p\coloneqq{}^{\padova}((\mathcal{T}^\mathsf{c})^{\padova})$. 

(ii) Suppose that $\T$ is compactly generated. By definition of the Brown-Comenetz duals, it is clear that ${}^{\padova}((\Tc)^\ast)=(\Tc)^{\padova}=\Tb$. By definition we then have that 
$({}^{\padova}((\Tc)^\ast))^{\padova}=\Tb_i$, finishing the proof.

The final statement follows from Lemma \ref{triangulated_subcategories}.
\end{proof}

\subsection{First example: the derived category of a ring} Let us compute some examples of the subcategories of Definition \ref{main_definition} for some familiar triangulated categories. The first example, the case of the derived category of modules over a ring, helps us justify the nomenclature.

\begin{prop}\label{subcategories for the derived category}
\label{examples}
The following table contains a list of some distinguished subcategories for the triangulated category $\T=\mathsf{D}(R)$ for any ring $R$.
\small\begin{table}[h!]
\begin{tabular}{|c|c|c|c|c|c|c|c|c|c|c|}
\hline
$\mathcal{T}$ & $\mathcal{T}^{\mathsf{c}}$ & $\mathcal{T}^-$ & $\mathcal{T}^+$ & $\mathcal{T}^{\mathsf{b}}$ & $\mathcal{T}^-_p$ & $\mathcal{T}^{\mathsf{b}}_p$ & $\mathcal{T}^+_i$ & $\mathcal{T}^{\mathsf{b}}_i$   \\ \hline
$\mathsf{D}(R)$ &
$\mathsf{K}^{\mathsf{b}}(\proj{R})$ & $\mathsf{D}^-(R)$ &  $\mathsf{D}^+(R)$ & $\mathsf{D}^{\mathsf{b}}(R)$ &$\mathsf{K}^-(\Proj{R})$ &$\mathsf{K}^{\mathsf{b}}(\Proj{R})$ &$\mathsf{K}^+(\Inj{R})$ &  $\mathsf{K}^{\mathsf{b}}(\Inj{R})$  \\ \hline
\end{tabular}
\end{table}
\end{prop}
\begin{proof}
Let $R$ be an arbitrary ring and fix $\mathcal{T}=\mathsf{D}(R)$. These subcategories of $\mathsf{D}(R)$ were computed in \cite{koenig, rickard_morita} and a detailed account of the respective proofs can be found for example in \cite[Lemma 2.12]{cummings}. We include the arguments for the convenience of the reader. The fact that $\mathcal{T}^{\mathsf{c}}=\mathsf{K}^{\mathsf{b}}(\proj R)$ is well-known and can be found in \cite[Proposition 6.4]{bokstedt_neeman}. Recall that $R^*$, the Brown-Comenetz dual of $R$, is an injective cogenerator of $\Mod{R}$.
\begin{itemize}[leftmargin=*]
\item $\mathcal{T}^{\?}=\mathsf{D}^{\?}(R)$, for any $\?$ in $\{+,-,\mathsf{b}\}$. Since $\mathcal{T}^\mathsf{c}=\mathsf{thick}(R)$, it follows from Lemma \ref{triangulated_subcategories} (v)(b) that $\mathcal{T}^{-}=R^{\perp_{\gg}}$, i.e. $\mathcal{T}^-$ consists of the objects $x$ for which $H^n(x)\cong \mathsf{Hom}_{\mathsf{D}(R)}(R,x[n])\cong 0$ for $n\gg0$. This proves the result for $\?=-$; the case $\?=+$ is shown dually and $\?=\mathsf{b}$ follows by combining the two previous statements.
\item $\mathcal{T}^{-}_{p}=\mathsf{K}^{-}(\Proj{R})$ and $\mathcal{T}^{+}_{i}=\mathsf{K}^+(\Inj{R})$. If $x$ lies in $\mathcal{T}^{-}_{p}$, then for $n\ll0$, we have that $\mathsf{Hom}_{\mathcal{T}}(x,R^*[n])=0$ and, thus, $H^{-n}(x)=0$. Hence, $x$ lies in $\mathsf{D}^{-}(R)$, and this argument is reversible. We get, therefore, that $\mathcal{T}^-_p=\mathsf{D}^{-}(R)\simeq\mathsf{K}^{-}(\Proj R)$. The second equality is shown dually.

\item $\mathcal{T}^{\mathsf{b}}_{p}=\mathsf{K}^{\mathsf{b}}(\Proj{R})$ and $\mathcal{T}^{\mathsf{b}}_i=\mathsf{K}^{\mathsf{b}}(\Inj{R})$. If $x$ lies in $\mathsf{K}^{\mathsf{b}}(\Proj R)$, then for every $y$ in $\mathsf{D}^{\mathsf{b}}(R)$ we have  
$$\mathsf{Hom}_{\mathsf{D}(R)}(x,y[n])\cong \mathsf{Hom}_{\mathsf{K}(R)}(x,y[n])$$ 
and since both $x$ and $y$ are bounded complexes, the latter is $0$ for $|n|$ large enough. This means that $\mathsf{K}^{\mathsf{b}}(\Proj R)\subseteq $$^{\padova}(\mathsf{D}^{\mathsf{b}}(R))=$$^{\padova}(\mathcal{T}^{\mathsf{b}})=\mathcal{T}^{\mathsf{b}}_p$. If, on the other hand, $x$ is an object in $\mathcal{T}^{\mathsf{b}}_p$, then there is $N$ such that $\mathsf{Hom}_{\mathsf{D}(R)}(x,R^*[n])=0$ for $|n|\geq N$. However, since $H^{-n}(x)=0$ if and only if $\mathsf{Hom}_{\mathsf{D}(R)}(x,R^*[n])=0$, we see that $x$ has bounded cohomology. Therefore, there is a bounded above complex $P=(P^i,d^i)$ in $\mathsf{K}^{-}(\Proj R)$ that is quasi-isomorphic to $x$ and has bounded cohomology. We claim that for $n\ll0$, $\mathsf{Coker}(d^{n-1})$ is projective. Indeed, since $P\cong x$ lies in ${}^{\padova}(\mathsf{D}^{\mathsf{b}}(R))$, we have that $\mathsf{Hom}_{\mathsf{D}(R)}(x,c[\gg])\cong\mathsf{Hom}_{\mathsf{K}(R)}(P,c[\gg])=0$ for any $R$-module $c$. Take $c$ to be the coproduct $\oplus_{n\in\mathbb{Z}}\mathsf{Coker}(d^{n-1})$, and let $k_c$ be an integer such that $\mathsf{Hom}_{\mathsf{K}(R)}(P,c[k])=0$ for all $k\geq k_c$. Let now $k$ be an integer such that $k\geq \mathsf{max}\{N,k_c\}$, guaranteeing that for such $k$, we have both $\mathsf{Hom}_{\mathsf{D}(R)}(P,c[k])=0$ and $H^{-k}(P)=0$. Consider the natural map of complexes induced by the projection map $\pi\colon P^{-k}\longrightarrow \mathsf{Coker}(d^{-k-1})$. The fact that this is a null-homotopic map implies the existence of $f\colon P^{-k+1}\longrightarrow \mathsf{Coker}(d^{-k-1})$ such that $f\circ d^{-k}=\pi$, obtaining the diagram below.
        \[
  \begin{tikzcd}
\cdots \arrow[r] & P^{-k-1} \arrow[r, "d^{-k-1}"] \arrow[d] & P^{-k} \arrow[r, "d^{-k}"] \arrow[d, "\pi"]    & P^{-k+1} \arrow[dl, "f"]\arrow[r] \arrow[d]  & \cdots \\
\cdots \arrow[r] & 0 \arrow[r]                        & \mathsf{Coker}(d^{-k-1}) \arrow[r] & 0  \arrow[r]                                 & \cdots \\
\end{tikzcd}
    \]
Now, since $H^{-k}(P)=0$, there is a natural isomorphism $\mathsf{Im}(d^{-k})\cong \mathsf{Coker}(d^{-k-1})$ and, therefore, we may consider a canonical embedding $\epsilon\colon \mathsf{Coker}(d^{k-1})\longrightarrow P^{-k+1}$ such that $\epsilon\circ \pi=d^{-k}$. Observe that $f\circ \epsilon\circ \pi=f\circ d^{-k}=\pi$ and, therefore, $f$ is a split epimorphism and $\mathsf{Coker}(d^{k-1})$ is projective. This shows that the truncated complex
    \[
    \cdots\rightarrow 0\rightarrow \mathsf{Coker}(d^{-k-1})\rightarrow P^{-k+1}\rightarrow P^{-k+2}\rightarrow \cdots
    \]
is a bounded complex of projectives quasi-isomorphic to $x$. The second statement is shown dually. \qedhere
\end{itemize}
\end{proof}


\subsection{Example: compactly generated triangulated categories with silting objects}
We now consider compactly generated triangulated categories with a special type of objects.

\begin{defn}\cite{NSZ}\cite{PV}
    Let $\mathcal{M}$ be an additively closed subcategory of a triangulated category $\mathcal{T}$, i.e.~$\Add{\mathcal{M}}=\mathcal{M}$. We say that $\mathcal{M}$ is \textbf{silting} if $(\mathcal{M}^{\perp_{>0}},\mathcal{M}^{\perp_{\leq 0 }})$ is a t-structure in $\T$.
An object $M$ in $\mathcal{T}$ is \textbf{silting} if $\mathsf{Add}(M)$ is a silting subcategory. Two silting objects $M$ and $N$ are said to be \textbf{equivalent} if $\Add{M}=\Add{N}$.
\end{defn}

\begin{rem}
If $\T$ is a compactly generated triangulated category and $M$ is a silting object of $\T$ lying in $\Tc$ (such $M$ is often referred to as a \textit{compact silting object}), then $\Tc=\mathsf{thick}(M)$, as stated in \cite[Proposition 4.2]{aihara_iyama}, following \cite{neeman2} and \cite{ravenel}.
\end{rem}

\begin{exmp}\label{example compact silting}
A prototypical example of silting object arises in the derived category $\mathsf{D}(\Gamma)$ of a dg algebra $\Gamma$ for which $H^n(\Gamma)=0$ for all $n>0$.
In fact, every algebraic compactly generated triangulated category $\T$ that has a silting object contained in $\Tc$ is of this form (\!\!\cite{keller0}).
\end{exmp}

\begin{exmp}
A non-algebraic example of a compactly generated triangulated category with a compact silting object is the homotopy category of modules over a connective ring spectrum, see \cite{schwede_shipley}. 
\end{exmp}


If $\T$ is a compactly generated triangulated category, there is also a co-t-structure associated to a silting object $M$, namely $({}^{\perp}(M^{\perp_{>0}}),M^{\perp_{>0}})$, with coheart $\Add(M)$, see \cite[Corollary 3.10]{angeleri_marks_vitoria}.

\begin{prop}\label{restrictions of t-structure and co-t-structure}
Let $\T$ be a triangulated category with coproducts and let $M$ be a silting object in $\T$ with associated heart $\mathcal{H}_M=M^{\perp_{\neq 0}}$. Then we have that\textnormal{:}
\begin{enumerate}
\item The t-structure associated to $M$ restricts to a bounded t-structure in $M^{\padova}$, and $M^{\padova}=\mathsf{thick}(\mathcal{H}_M)$.
\item If $\T$ is moreover a compactly generated triangulated category, then the co-t-structure associated to $M$ restricts to a co-t-structure in ${}^{\padova}(M^{\padova})$.
\end{enumerate}
\end{prop}
\begin{proof}
(i) Consider the t-structure $(M^{\perp_{>0}},M^{\perp_{\leq 0}})$ associated to $M$, let $y$ be an object in $M^{\padova}$ and consider the truncation triangle
\[
u\rightarrow y\rightarrow v\rightarrow u[1]
\]
with $u$ in $M^{\perp_{>0}}\subseteq M^{\perp_{\gg}}$ and $v$ in $M^{\perp_{\leq 0}}\subseteq M^{\perp_{\ll}}$. Since both $u$ and $y$ both lie in $M^{\perp_{\gg}}$, then so does $v$ and, therefore, $v$ lies in $M^{\padova}$. In turn, this shows that also $u$ lies in $M^{\padova}$, showing therefore that the t-structure restricts to $M^{\padova}$. To check that this restriction is bounded, it suffices to observe that any $x$ in $M^{\padova}$ lies in $M^{\perp_{>a}}\cap M^{\perp_{\leq b}}$ for some $a$ and $b$. Observe furthermore that $\mathcal{H}_M=M^{\perp_{\neq 0}}$ is contained in $M^{\padova}$ and, therefore, the heart of $(M^{\perp_{>0}}\cap M^{\padova},M^{\perp_{\leq 0}}\cap M^{\padova})$ is $\mathcal{H}_M$. Since this t-structure is bounded in $M^{\padova}$ it follows that $M^{\padova}=\mathsf{thick}(\mathcal{H}_M)$. 

(ii) Let $\mathcal{T}$ be compactly generated, and consider the co-t-structure $(^{\perp}(M^{\perp_{>0}}),M^{\perp_{>0}})$ in $\mathcal{T}$ associated to $M$. We will show that it restricts to a co-t-structure in ${}^{\padova}(M^{\padova})$, which by (i) coincides with ${}^{\padova}(\mathcal{H}_{M})$. Let $y$ be an object in $^{\padova}(\mathcal{H}_M)$ and consider the truncation triangle 
\[
u\rightarrow y\rightarrow v\rightarrow u[1]
\]
with $u$ in ${^{\perp}}(M^{\perp_{>0}})$ and $v$ in $M^{\perp_{>0}}$. We have that $\mathsf{Hom}(M^{\perp_{>0}},\mathcal{H}_M[<0])=0$ since $\mathcal{H}_M[<0]$ is contained in the coaisle $M^{\perp_{\leq 0}}$. In particular, we have that $M^{\perp_{>0}}$ is contained in $^{\perp_{\ll}}(\mathcal{H}_M)$. It follows by the triangle above that $u$ belongs in $^{\perp_{\ll}}(\mathcal{H}_M)$. Moreover, we also have that $\mathsf{Hom}(^{\perp}(M^{\perp_{>0}}),\mathcal{H}_M[>0])=0$ since $\mathcal{H}_M[>0]$ is contained in the aisle $M^{\perp_{>0}}$.  This implies that $^{\perp}(M^{\perp_{>0}})$ lies in $^{\perp_{\gg}}(\mathcal{H}_M)$. Therefore $u$ belongs to $^{\padova}(\mathcal{H}_M)$ and by the truncation triangle the same follows for $v$, showing that the given co-t-structure indeed restricts to a co-t-structure of $^{\padova}(\mathcal{H}_M)$. 
\end{proof}

The statement above tells us in particular that the far-away pair generated by a silting object $M$ coincides with the far-away pair cogenerated by its associated heart $\mathcal{H}_M$ and that, moreover, 
$$M^{\padova}=({}^{\padova}\mathcal{H}_M)^{\padova}=\mathsf{thick}(\mathcal{H}_M).$$ 

It is well-known that if $\T$ is compactly generated and $M$ is a compact silting object in $\T$ then $M$ is what is sometimes referred to as a \textit{classical silting object} in $\T^{\mathsf{c}}$, i.e.~$M$ has the property that $\mathsf{Hom}_{\T^{\mathsf{c}}}(M,M[>0])=0$ and $\mathsf{thick}(M)=\T^{\mathsf{c}}$.

\begin{cor} \label{existence of silting compact}
    Let $0\neq \mathcal{T}$ be a compactly generated triangulated category with $\mathcal{T}^{\mathsf{b}}=0$. Then $\mathcal{T}$ has no compact silting object. 
\end{cor}
\begin{proof}
    If $\mathcal{T}$ had a compact silting object $M$, then by \cite[Theorem 1.3]{hoshino_kato_miyachi} we would have $\mathcal{H}_M=\Mod\mathsf{End}_{\mathcal{T}}(M)$. However, from Proposition \ref{restrictions of t-structure and co-t-structure} together with the fact that $\mathcal{T}^{\mathsf{b}}=M^{\padova}$, it would follow that $0=\mathcal{T}^{\mathsf{b}}=\mathsf{thick}(\mathcal{H}_M)$, showing that $\mathsf{End}_{\mathcal{T}}(M)=0$, a contradiction with the fact that $M$ is a generator of $0\neq \mathcal{T}$.
\end{proof}

If we assume more about the silting object $M$, we will be able to say more about the far-away pair mentioned above. We introduce the following definition in an arbitrary triangulated category.

\begin{defn} \label{far_away_generated} 
A subcategory $\mathcal{M}$ of a triangulated category $\mathcal{T}$ is said to be a 
\textbf{far-away generating subcategory} for $\mathcal{T}$ if $\mathcal{M}\subseteq \mathcal{M}^{\padova}$ and $\mathsf{thick}(\mathcal{M})=\mathcal{T}$. 
In this case we say that $\mathcal{T}$ is \textbf{far-away generated}.
\end{defn}


\begin{exmp}\label{about silting and tilting1}
If $\mathcal{M}$ is a silting subcategory with the additional property that $\mathsf{Hom}_\T(\mathcal{M},\mathcal{M}[\neq 0])=0$ (such $\mathcal{M}$ are known as \textbf{tilting subcategories}) then trivially we have that $\mathsf{thick}(\mathcal{M})$ is far-away generated. This is the case for $\mathcal{M}=\mathsf{Add}(R)$ in the derived category $\mathsf{D}(R)$, for a ring $R$, showing that $\mathsf{K}^\mathsf{b}(\Proj{R})$ is far-away generated. The same holds for $\mathsf{K}^\mathsf{b}(\proj{R})$ with the far-away generator $R$. More generally, if $\mathcal{M}$ is a silting subcategory of a triangulated category $\mathcal{T}$, then $\mathsf{thick}(\mathcal{M})$ is far-away generated by $\mathcal{M}$ if and only if $\mathcal{M}\subseteq \mathcal{M}^{\perp_{\ll}}$. This is the case for $\mathcal{M}=\Add{\Gamma}$ in the derived category $\mathsf{D}(\Gamma)$, for a non-positively graded dg algebra with finitely many non-zero cohomologies, showing that $\mathsf{broad}_{\Sigma}(\Gamma)=\mathsf{thick}(\Add{\Gamma})$ is far-away generated (see Lemma \ref{broad generation}). \end{exmp}


We will see later (see Corollary \ref{far-away gen db and dsg}) that if $\Lambda$ is an Artin algebra of infinite global dimension, neither $\mathsf{D}^{\mathsf{b}}(\smod{\Lambda})$ nor $\mathsf{D}_{\mathsf{sg}}(\Lambda)$ are far-away generated. Note that it is already well-known that neither of these categories have classical silting objects (see for example \cite[Example 2.5(a)]{aihara_iyama} and \cite[Corollary 3.3]{chen_li_zhang_zhao}). 

Let us collect some structural properties of far-away generated categories.

\begin{lem}\label{far-away and containments}
Let $\mathcal{T}$ be a compactly generated triangulated category, and assume that $\mathcal{T}^\mathsf{c}$ is far-away generated. Then the subcategories $\T^{\mathsf{c}}$, $\T^\mathsf{b}_i$ and $\T^\mathsf{b}_p$ are contained in $\T^{\mathsf{b}}$.
\end{lem}
\begin{proof}
$\mathcal{T}^\mathsf{c}\subseteq \mathcal{T}^{\mathsf{b}}$: Let $x$ be an object in $\mathcal{T}^\mathsf{c}$ and let $\mathcal{M}$ be a subcategory far-away generating $\mathcal{T}^\mathsf{c}$. We will show that $\T^{\mathsf{c}}\subseteq x^{\padova}$. We have seen that $x^{\padova}$ is a broad (hence, thick) subcategory of $\mathcal{T}$ (see Lemma \ref{triangulated_subcategories}). Thus, we only need to show that $\mathcal{M}$ is contained in $x^{\padova}$ or, equivalently, that $x$ lies in ${}^{\padova}\mathcal{M}$. Note, however, that ${}^{\padova}\mathcal{M}$ is a thick subcategory containing $\mathcal{M}$ and, therefore, it contains $\mathcal{T}^c$. Hence, $x$ lies in ${}^{\padova}\mathcal{M}$ and the result follows.

$\T^\mathsf{b}_i\subseteq \T^{\mathsf{b}}$: This holds by considering right far-away orthogonals of the inclusion $\mathcal{T}^\mathsf{c}\subseteq \mathcal{T}^{\mathsf{b}}$.

$\T^\mathsf{b}_i\subseteq \T^{\mathsf{b}}$: This holds by considering left far-away orthogonals of the inclusion $\T^\mathsf{b}_i\subseteq \T^{\mathsf{b}}$, by Lemma \ref{common facts of the subcategories}.
\end{proof}

This lemma gives us the following diagram of inclusions of distinguished thick subcategories of a compactly generated triangulated category $\mathcal{T}$ for which the subcategory of compact objects $\mathcal{T}^\mathsf{c}$ is far-away generated. In the diagram, we write $\mathbb{BC}(\mathcal{T}^\mathsf{c})$ for the thick subcategory of $\mathcal{T}$ generated by $\{t^*\colon t\in\mathcal{T}^\mathsf{c}\}$.
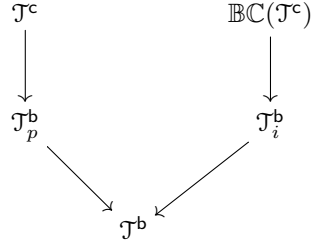
$$\xymatrix{\mathcal{T}^\mathsf{c}\ar[d]&&\mathbb{BC}(\mathcal{T}^\mathsf{c})\ar[d]\\ \mathcal{T}^{\mathsf{b}}_p\ar[dr]&&\mathcal{T}^{\mathsf{b}}_i\ar[dl]\\&\mathcal{T}^{\mathsf{b}}}$$
If a triangulated category $\T$ has a compact silting object which is simultaneously a far-away generator for $\Tc$, then we can improve Proposition \ref{restrictions of t-structure and co-t-structure} as follows.

\begin{thm}\label{silting far-away generating}
Let $\T$ be a compactly generated triangulated category and $M$ be a compact silting object which far-away generates $\Tc$. Then we have the following far-away triple
\[(\Tb_p,\Tb,\Tb_i)=(\mathsf{broad}_{\Sigma}(\Tc),\mathsf{thick}(\mathcal{H}_M),\mathsf{broad}_{\prod}((\Tc)^\ast))=(\mathsf{thick}(\Add{M}),\mathsf{thick}(\mathcal{H}_M),\mathsf{thick}(\Product{M^\ast}))\]
such that the t-structure induced by $M$ restricts to $\Tb$ as a bounded t-structure, the co-t-structure induced by $M$ restricts to $\Tb_p$ as a bounded co-t-structure and the co-t-structure induced by $M^\ast$ restricts to $\Tb_i$ as a bounded co-t-structure. 
\end{thm}

\begin{proof}
As seen in Example \ref{example compact silting}, since $M$ is a compact generator of $\T$, $\Tc=\mathsf{thick}(M)$ and by Lemma \ref{common facts of the subcategories} it follows that the far-away pair generated by $M$ is $(\Tb_p,\Tb)$ and the far-away pair cogenerated by $M^\ast$ is $(\Tb,\Tb_i)$. By Proposition \ref{restrictions of t-structure and co-t-structure}, since $M$ is silting we have that $\Tb=\mathsf{thick}(\mathcal{H}_M)$, the t-structure induced by $M$ restricts to $\Tb$ as a bounded t-structure and the co-t-structure induced by $M$ restricts to $\Tb_p$. Dual arguments show that also the co-t-structure induced by $M^\ast$ restricts to $\Tb_i$. To prove our theorem it will suffice to show the  restriction to $\Tb_p$ actually yields a bounded co-t-structure with co-heart $\Add{M}$. Indeed, it will follow dually that the restriction to $\Tb_i$ yields a bounded co-t-structure with co-heart $\Product{M^\ast}$. Then the boundedness condition guarantees that $\Tb_p$ (respectively $\Tb_i$) coincides with the thick closure of the co-heart (see \cite[Theorem 4.20]{bounded co t structures}).

Let us fix the notation $\mathcal{U}:={}^\perp(M^{\perp_{>0}})$ and $\mathcal{V}:=M^{\perp_{>0}}$ and consider the co-t-structure $(\mathcal{U}\cap\Tb_p,\mathcal{V}\cap\Tb_p)$ in $\Tb_p$. By Lemma \ref{far-away and containments}, we have that $\Tb_p\subseteq \Tb=\cup_{n\in\mathbb{Z}}\mathcal{V}[n]$. Recall that, if we denote by $H_M^0$ the cohomological functor $\T\longrightarrow \mathcal{H}_M$ associated to the t-structure induced by $M$, since the restriction of this t-structure to $\Tb$ is bounded, an object $y$ of $\Tb$ will have the property that $H^{-n}_M(y):=H^0_M(y[n])\neq 0$ for only finitely many $n$ in $\mathbb{Z}$. This leaves us to prove that 
\[\Tb_p=\cup_{n\in\mathbb{Z}}(\mathcal{U}\cap\Tb_p)[n]=\cup_{n>0}(\mathcal{U}\cap\Tb_p)[n]\]
(where the last equality follows trivially from the fact that $\mathcal{U}\cap \Tb_p$ is cosuspended), i.e. that for any $x$ in $\Tb_p$, there is an integer $N>0$ such that $x$ lies in $(\mathcal{U}\cap \Tb_p)[N]={}^\perp(\mathcal{V}\cap\Tb_p)[N]={}^\perp(\mathcal{V}[N]\cap\Tb_p)$. Suppose that such $N$ does not exist. This means that for each $n>0$ there is an object $y_n$ in $\mathcal{V}[n]\cap\Tb_p$ such that $\mathsf{Hom}_\T(x,y_n)\neq 0$. As noted above, since $y_n$ lies in $\Tb_p$ and hence in $\Tb$, $y_n$ lies in $\mathcal{H}_M[n+k]\ast \cdots \ast \mathcal{H}_M[n]$ for some $k\geq 0$. Thus, for each $n>0$, there is $k_n\geq 0$ such that $\mathsf{Hom}_\T(x,\mathcal{H}_M[n+k_n])\neq 0$. For each $n>0$, let $a_n$ be an object in $\mathcal{H}_M$ such that $\mathsf{Hom}_\T(x,a_n[n+k_n])\neq 0$ and let $a=\prod_{n>0}a_n$ in $\mathcal{H}_M$. Observe that then we have $\mathsf{Hom}_\T(x,a[\gg])\neq 0$, thus contradicting the fact that $x$ lies in $\Tb_p={}^{\padova}\mathcal{H}_M$. Hence, the postulated $N$ must exist, proving that the co-t-structure is indeed bounded.
%
%
%
\end{proof}

Note that some distinguished subcategories for the derived category of a dg algebra behave exactly like in the case of the derived category of modules over a ring.

\begin{rem} \label{dg algebras}
  Let $\Gamma$ be a dg $k$-algebra and consider its derived category $\mathcal{T}=\mathsf{D}(\Gamma)$. Using Lemma \ref{triangulated_subcategories}(v), since $H^0(-)\cong\mathsf{Hom}_{\mathsf{D}(\Gamma)}(\Gamma,-)$, we can see immediately see that: 
    \begin{itemize}
        \item[(i)] $\mathcal{T}^+=\{x\in\mathsf{D}(\Gamma):  H^n(x)=0, \ n\ll0\}$;
        \item[(ii)] $\mathcal{T}^-=\{x\in\mathsf{D}(\Gamma): H^n(x)=0, \ n\gg0\}$;
        \item[(iii)] $\mathcal{T}^{\mathsf{b}}=\{x\in\mathsf{D}(\Gamma): H^n(x)=0, \ |n|\gg0\}$. 
    \end{itemize}
\end{rem}

\begin{rem}\label{approx}
Another setup for the examples discussed above is provided by Neeman in \cite{neeman5}. If $\mathcal{T}$ is a triangulated category and $\tau=(\mathcal{T}^{\leq 0},\mathcal{T}^{\geq 0})$ is a t-structure in $\mathcal{T}$, we may then consider the following triangulated subcategories of $\mathcal{T}$: 
     \[
     \mathcal{T}^{-,\tau}\coloneqq \bigcup_{n>0}\mathcal{T}^{\leq n}, \ \ \mathcal{T}^{+,\tau}\coloneqq \bigcup_{n>0}\mathcal{T}^{\geq -n} \text{  and  } \mathcal{T}^{\mathsf{b},\tau}\coloneqq \mathcal{T}^{-,\tau}\cap \mathcal{T}^{+,\tau}. 
     \]
If $\T$ is a triangulated category with a single compact generator $G$, a particular choice of $\tau$ is the t-structure generated by $G$. When $\mathcal{T}$ is \emph{weakly approximable}, i.e.\ the t-structure generated by $G$ satisfies several ``approximability conditions'' (see \cite[Definition 0.25]{neeman5} for details) then the subcategories defined above are intrinsic and independent of the chosen generator. It follows by \cite[Corollary 3.10]{burke_neeman_pauwels} that the subcategories $\mathcal{T}^{-,\tau}, \mathcal{T}^{+,\tau}$ and $\mathcal{T}^{\mathsf{b},\tau}$ as defined above, are the same as $\mathcal{T}^-,\mathcal{T}^+$ and $\mathcal{T}^{\mathsf{b}}$ respectively, of Definition \ref{main_definition}. 
\end{rem}

We end this subsection by applying Theorem \ref{silting far-away generating} to the case of the derived category of a non-positive dg algebra $\Gamma$, i.e.~a dg algebra $\Gamma$ for which $H^n(\Gamma)=\mathsf{Hom}_{\mathsf{D}(\Gamma)}(\Gamma,\Gamma[n])=0$ for all $n>0$. Such dg algebras are often referred to as \textbf{connective}.

\begin{cor} \label{T^b_p for dg algebras}
Let $\Gamma$ be a connective differential graded algebra over a commutative ring $k$ with $H^n(\Gamma)\neq 0$ for only finitely many integers $n$. Considering $\T$ to be $\mathsf{D}(\Gamma)$, we have that $\Tb_p=\mathsf{broad}_{\Sigma}(\mathsf{per}(\Gamma))=\mathsf{thick}(\Add{\Gamma})$ and $\Tb_i=\mathsf{broad}_{\prod}((\mathsf{per}(\Gamma)^\ast)=\mathsf{thick}(\Product{\Gamma^\ast})$.
\end{cor}
\begin{proof}
This follows immediately from Theorem \ref{silting far-away generating}, taking into account that, under the assumption of having finitely many non-zero cohomologies, $\Gamma$ is a far-away generator for $\mathsf{per}(\Gamma)$.
\end{proof}

\subsection{Example: homotopy categories for Artin algebras and the big singularity category} 
We will be considering the categories $\mathsf{K}(\Inj{\Lambda})$ and $\mathsf{K}(\Proj\Lambda)$ for an Artin algebra $\Lambda$. The \textit{big singularity category subcategory} under consideration will be that of acyclic complexes of injective $\Lambda$-modules $\mathsf{K}_{\mathsf{ac}}(\Inj{\Lambda})$. 
We will need the following statements about $\mathsf{K}(\Inj \Lambda)$.

\begin{lem}\label{some facts about KInj}
Let $A$ be a right noetherian ring and consider the recollement \textnormal{(\ref{rec})}. The following hold.
\begin{enumerate}
\item[\textnormal{(i)}] $\mathsf{K}^+(\Inj A)\subseteq \mathsf{Im}(\mathsf{Q}_\rho)$.
\item[\textnormal{(ii)}] $\mathsf{Hom}_{\mathsf{K}(\Inj A)}(-,x)\cong \mathsf{Hom}_{\mathsf{D}(A)}(\mathsf{Q}(-),\mathsf{Q}(x))$ for any $x$ in $\mathsf{Im}(\mathsf{Q}_\rho)$.
\item[\textnormal{(iii)}] For any $x$ in $\mathsf{K}^{\mathsf{b}}(\proj A)$, $\mathsf{Q}_\lambda(x)=\mathsf{Q}_\rho(x)$.
\item[\textnormal{(iv)}] For any family of objects $(x_i)_{i\in I}$ whose coproduct lies in $\mathsf{D}^+(A)$, we have $\mathsf{Q}_\rho(\coprod_{i\in I} x_i)\cong \coprod_{i\in I}\mathsf{Q}_\rho(x_i)$.
\end{enumerate}
\end{lem}
\begin{proof}
(i) is rather classic (see \cite{krause2}) and (ii) follows directly from adjunction. 

(iii) Observe that since $\mathsf{Q}$ preserves coproducts, $\mathsf{Q}_\lambda$ preserves compacts and, therefore, for any $x$ in $\mathsf{D}(A)^{\mathsf{c}}$, $\mathsf{Q}_\lambda(x)$ lies in $\mathsf{Im}(\mathsf{Q}_\rho)$. Therefore the unit of the adjoint pair $(\mathsf{Q},\mathsf{Q_\rho})$ applied to $\mathsf{Q}_\lambda(x)$ is an isomorphism, thus showing that $\mathsf{Q}_\lambda(x)\cong \mathsf{Q}_\rho(x)$.

(iv) First note that $\mathsf{Q}_\rho$ restricts to an equivalence of categories between $\mathsf{D}^+(A)$ and $\mathsf{K}^+(\Inj A)$. Now, since the coproduct of the family $(x_i)_{i\in. I}$ exists in $\mathsf{D}^+(A)$, and since $A$ is right noetherian, we then conclude that $\mathsf{Q}_\rho(\coprod_{i\in I} x_i)\cong \coprod_{i\in I}\mathsf{Q}_\rho(x_i)$.
\end{proof}

We will calculate these distinguished subcategories for $\mathsf{K}(\Inj{\Lambda})$ and $\mathsf{K}_{\mathsf{ac}}(\Inj{\Lambda})$ for an Artin algebra $\Lambda$. The following lemma will be important for our calculations.

\begin{lem} \label{far away of smod}
If $\Lambda$ is an Artin algebra, then for any symbol $\?$ in $\{\perp_{\gg},\perp_{\ll},\padova\}$, we have that, in $\mathsf{K}(\Inj \Lambda)$,
\begin{align*} \label{star condition}\tag{$*$}
    \mathsf{Q}_{\rho}(\smod \Lambda)^{\?}=\mathsf{Q}_{\rho}(\Mod \Lambda)^{\?}.
\end{align*}

\end{lem}
\begin{proof}
Since, by Corollary \ref{broad Db}, $\mathsf{broad}_\Sigma(\mathsf{D}^\mathsf{b}(\smod\Lambda))=\mathsf{D}^\mathsf{b}(\Mod\Lambda)$ and since $\mathsf{Q}_\rho$ preserves all coproducts existing in $\mathsf{D}^{\mathsf{b}}(\Mod\Lambda)$, we have that
$$\mathsf{broad}_\Sigma(\mathsf{Q}_\rho(\smod\Lambda))=\mathsf{Q}_\rho(\mathsf{broad}_\Sigma(\smod\Lambda))=\mathsf{Q}_\rho(\mathsf{D}^\mathsf{b}(\Mod\Lambda))=\mathsf{thick}(\mathsf{Q}_\rho(\Mod\Lambda))$$
By Lemma  \ref{triangulated_subcategories}(v)(b) we then get $\mathsf{Q}_\rho(\smod\Lambda)^{\?}=\mathsf{Q}_\rho(\Mod\Lambda)^{\?}$ for any symbol $\?$ in $\{\perp_\ll,\perp_\gg,\padova\}$.
\end{proof}

\begin{rem}\label{implication}
Observe that for an arbitrary right noetherian ring $A$, we have that if condition $(\ast)$ holds for $A$, then $(\smod A)^{\?}=(\Mod A)^{\?}$ for any symbol $\?$ in $\{\perp_\ll,\perp_\gg,\padova\}$. Indeed, using Lemma \ref{adjoints and far-away orthogonality}(i) with $\mathsf{F}=\mathsf{Q}$, $\mathsf{G}=\mathsf{Q}_\rho$, $\Y=\smod\Lambda$ and $\X=\mathsf{Q}_\rho(\smod A)$ to conclude that $\mathsf{Q}_\rho((\smod A)^{\?})\subseteq \mathsf{Q}_\rho(\smod A)^{\?}$. The latter is then equal to $\mathsf{Q}_\rho(\Mod A)^{\?}$. Now, given $y$ in $(\smod A)^{\?}$, we have that $\mathsf{Q}_\rho(y)$ lies in $\mathsf{Q}_\rho(\Mod A)^{\?}$ and, therefore, using the fact that $\mathsf{Q}_\rho$ is fully faithful, we get that $y$ lies in $(\Mod A)^{\?}$. The other inclusion is clear. Note that for an Artin algebra $\Lambda$ we already knew that $(\smod\Lambda)^{\?}=(\Mod\Lambda)^{\?}$ from Example \ref{exmp artin}.
\end{rem}

\begin{rem} \label{counterexample} 
The equality of Example \ref{exmp artin} does not hold for more general rings and, consequently, by Remark \ref{implication} the equality of Lemma \ref{far away of smod} will also not hold. Indeed, consider a ring $R$ for which every finitely generated module has finite projective dimension and $\gd R=\infty$. A particular example of such a ring was given by Nagata \cite{nagata}, which is in fact commutative noetherian. By the assumption that finitely generated $R$-modules have finite projective dimension, it follows that 
    \[
    (\smod R)^{\padova}\supseteq \Mod R
    \]
    in $\mathsf{D}(R)$. However, since in $\mathsf{D}(R)$ we have the equality 
   \[(\Mod R)^{\padova}=\mathsf{D}^{\mathsf{b}}(\Mod{R})^{\padova}=\mathsf{K}^{\mathsf{b}}(\Inj{R}),\]
   the subcategory $(\Mod R)^{\padova}$ cannot contain $\Mod R$, as it would imply that every $R$-module has finite injective dimension, contradicting the assumption. In other words, we have $(\mathsf{D}^{\mathsf{b}}(\smod R))^{\padova}\neq (\mathsf{D}^{\mathsf{b}}(\Mod R))^{\padova}$. In particular, here we have that $\mathsf{broad}_\Sigma(\smod R)\neq \mathsf{D}^\mathsf{b}(\Mod R)$, showing that also the statement of Corollary \ref{broad Db} fails for $R$.
\end{rem}


\begin{prop} 
\label{examples2}
The table below contains some distinguished subcategories of $\mathsf{K}(\Inj \Lambda)$ and $\mathsf{K}_{\mathsf{ac}}(\Inj \Lambda)$ where $\Lambda$ is an Artin algebra.

    \begin{table}[h!]
\begin{tabular}{|c|c|c|c|c|}

\hline
$\mathcal{T}$  & $\mathsf{K}(\Inj \Lambda)$           & $\mathsf{K}_{\mathsf{ac}}(\Inj \Lambda)$ \\ \hline
$\mathcal{T}^{\mathsf{c}}$   &  $\mathsf{D}^{\mathsf{b}}(\smod \Lambda)$ & $\mathsf{D}_{\mathsf{sg}}(\Lambda)$       \\ \hline
$\mathcal{T}^-$  &   &  $\mathsf{K}^-_{\mathsf{ac}}(\Inj \Lambda)$   \\ \hline
$\mathcal{T}^+$  &  $\mathsf{K}^+(\Inj \Lambda)$ & $\{0\}$                              \\ \hline
$\mathcal{T}^{\mathsf{b}}$  &$\mathsf{K}^{\mathsf{b}}(\Inj \Lambda)$ & $\{0\}$                                \\ \hline 
$\mathcal{T}^-_p$ &  $\{x\colon H^n(x)=0 \text{ for }n\gg0\}$  & $\mathsf{K}_{\mathsf{ac}}(\Inj \Lambda)$ \\ \hline
$\mathcal{T}^+_p$ &  $\{x\colon H^n(x)=0 \text{ for }n\ll0\}$  & $\mathsf{K}_{\mathsf{ac}}(\Inj \Lambda)$ \\ \hline
$\mathcal{T}^{\mathsf{b}}_p$  & $\{x\colon \#\{n\colon H^n(x)\neq 0\}<\infty \}$ & $\mathsf{K}_{\mathsf{ac}}(\Inj \Lambda)$ \\ \hline
$\mathcal{T}^-_i$  & see Lemma \textnormal{\ref{subcategories for IG}} & $\mathsf{K}_{\mathsf{ac}}(\Inj \Lambda)$ \\ \hline
$\mathcal{T}^+_i$  & see Lemma \textnormal{\ref{subcategories for IG}} & $\mathsf{K}_{\mathsf{ac}}(\Inj \Lambda)$ \\ \hline
$\mathcal{T}^{\mathsf{b}}_i$  &  see Lemma \textnormal{\ref{subcategories for IG}} & $\mathsf{K}_{\mathsf{ac}}(\Inj \Lambda)$ \\ \hline
\end{tabular}
\end{table}
\end{prop}
\begin{proof}
    Let us write $\mathcal{T}$ for $\mathsf{K}(\Inj \Lambda)$ and $\mathcal{T}_{\mathsf{ac}}$ for $\mathsf{K}_{\mathsf{ac}}(\Inj \Lambda)$.
    \begin{itemize}[leftmargin=*]
       \item $\mathcal{T}_{\mathsf{ac}}^+=\{0\}$. Let $x=(x^i,d^i)_{i\in\mathbb{Z}}$ be a complex in $\mathcal{T}_{\mathsf{ac}}$. Then we have that 
        \[
        \mathsf{Hom}_{\mathcal{T}_{\mathsf{ac}}}(\mathsf{I}_{\lambda}\mathsf{Q}_{\rho}(\smod \Lambda),x[\ll])\cong \mathsf{Hom}_{\mathcal{T}}(\mathsf{Q}_{\rho}(\smod \Lambda),\mathsf{I}(x)[\ll]).
        \]
      Since $\mathsf{thick}(\mathsf{I}_\lambda \mathsf{Q}_\rho(\smod\Lambda))=\Tc_{\mathsf{ac}}$, we see that $x$ lies in $\mathcal{T}_{\mathsf{ac}}^+$ if and only if $\mathsf{I}(x)$ lies in $\mathsf{Q}_{\rho}(\smod \Lambda)^{\perp_{\ll}}$, which coincides with $\mathsf{Q}_{\rho}(\Mod \Lambda)^{\perp_{\ll}}$ by Lemma \ref{far away of smod}. Assume then that $x$ lies in $\T_{\mathsf{ac}}^+$ and that, therefore, $\mathsf{I}(x)$ lies in $\T^+$. For simplicity we will write $x$ instead of $\mathsf{I}(x)$, thinking of $\mathsf{I}$ as the inclusion of the subcategory $\T_{\mathsf{ac}}$ into $\T$. By \cite[Lemma 2.1]{krause2}, we have an isomorphism 
        \[
        \mathsf{Hom}_{\mathcal{T}}(\mathsf{Q}_{\rho}(\oplus_j\mathsf{Ker}(d^j)),x[n])\cong \mathsf{Hom}_{\mathsf{K}(\Lambda)}(\oplus_j\mathsf{Ker}(d^j),x[n]),
        \]
and thus, since $x$ lies in $\T^{+}$, there is $N$ in $\mathbb{Z}$ such that for all $j$ and $n\leq N$, we have $\mathsf{Hom}_{\mathsf{K}(\Lambda)}(\mathsf{Ker}(d^j),x[n])=0$. In particular, we get $\mathsf{Hom}_{\mathsf{K}(\Lambda)}(\mathsf{Ker}(d^n),x[n])=0$ for all $n\leq N$. For such integers $n\leq N$, consider the morphism of complexes $f\colon \mathsf{Ker}(d^n)\rightarrow x[n]$ induced by the inclusion $f^n\colon\mathsf{Ker}(d^n)\rightarrow x^n$. Since $f$ is null-homotopic, there is $h\colon\mathsf{Ker}(d^n)\rightarrow x^{n-1}$ such that $f^n=d^{n-1}\circ h$. Since $x$ is acyclic, we have $\mathsf{Ker}(d^n)\cong\mathsf{Im}(d^{n-1})$ and therefore $d^{n-1}$ factors as a composition $f^n\circ i^n$, where $i^n\colon x^{n-1}\rightarrow \mathsf{Ker}(d^n)$. Hence, we have $f^n=d^{n-1}\circ h=f^n\circ i^n\circ h$ and since $f^n$ is a monomorphism, we get $i^n\circ h\simeq \mathsf{Id}_{\mathsf{Ker}(d^n)}$. In particular, it follows that $\mathsf{Ker}(d^n_x)$ is injective for all $n\leq N$, showing that $x$ is nullhomotopic. 
        \item $\mathcal{T}^-_{\mathsf{ac}}(\Inj \Lambda)=\mathsf{K}^{-}_{\mathsf{ac}}(\Inj \Lambda)$. By following verbatim the arguments above it follows that given a complex $x$ in $\mathcal{T}^{-}_{\mathsf{ac}}$, it must split in large degrees, implying that $\mathcal{T}^{-}_{\mathsf{ac}}\subseteq \mathsf{K}^{-}_{\mathsf{ac}}(\Inj \Lambda)$. For the inverse inclusion, observe that given $x$ in $\mathsf{K}^{-}_{\mathsf{ac}}(\Inj \Lambda)$, $\mathsf{I}(x)$ lies in $\mathsf{K}^{-}(\Inj \Lambda)$ and, therefore, it lies in  $\mathsf{Q}_{\rho}(\smod \Lambda)^{\perp_{\gg}}$, as wanted.

        \item $\mathcal{T}^{\mathsf{b}}_{\mathsf{ac}}=\{0\}$ follows from the fact that $\mathcal{T}^{\mathsf{b}}_{\mathsf{ac}}\subseteq \mathcal{T}^+_{\mathsf{ac}}$. Then, by considering appropriate far-away orthogonals, it follows that $(\mathcal{T}_{\mathsf{ac}})^{\?}_\alpha=\mathsf{K}_{\mathsf{ac}}(\Inj \Lambda)$ for any $\?$ in $\{+,-,\mathsf{b}\}$ and any $\alpha$ in $\{p,i\}$.
        \item $\mathcal{T}^+=\mathsf{K}^+(\Inj \Lambda)$. First observe that $\mathsf{K}^+(\Inj \Lambda)\subseteq \mathcal{T}^+$, since 
        \[
        \mathsf{Hom}_{\mathcal{T}}(\mathsf{Q}_{\rho}(\smod \Lambda),\mathsf{Q}_{\rho}(x)[\ll])\cong\mathsf{Hom}_{\mathsf{D}(\Lambda)}(\smod \Lambda,x[\ll])
        \]
        and $(\smod \Lambda)^{\perp_{\ll}}$ in $\mathsf{D}(\Lambda)$ contains $\mathsf{K}^+(\Inj \Lambda)$. Now let $x$ lie in $\mathcal{T}^+$ and consider the following triangle 
        \[
        \mathsf{II}_{\rho}(x)\rightarrow x\rightarrow \mathsf{Q}_{\rho}\mathsf{Q}(x)\rightarrow \mathsf{II}_{\rho}(x)[1]. 
        \]
        Since the functor $\mathsf{Q}_{\lambda}$ preserves compact objects, applying Lemma \ref{adjoints and far-away orthogonality} to the adjoint pair $(\mathsf{Q}_\lambda,\mathsf{Q})$ it follows that $\mathsf{Q}(x)$ lies in $\mathsf{D}^{+}( \Lambda)$. Consequently, $\mathsf{Q}_{\rho}\mathsf{Q}(x)$ lies in $\mathsf{K}^{\mathsf{+}}(\Inj \Lambda)\subseteq \mathcal{T}^+$ and from the above triangle it follows that $\mathsf{I}\mathsf{I}_{\rho}(x)$ lies in $\mathcal{T}^+$. We have 
        \[
        \mathsf{Hom}_{\mathcal{T}_{\mathsf{ac}}}(\mathsf{I}_{\lambda}\mathsf{Q}_{\rho}(\smod \Lambda),\mathsf{I}_{\rho}(x)[\ll])=\mathsf{Hom}_{\mathcal{T}}(\mathsf{Q}_{\rho}(\smod \Lambda),\mathsf{I}\mathsf{I}_{\rho}(x)[\ll])
        \]
        from which we infer that $\mathsf{I}_{\rho}(x)$ lies in $\mathcal{T}^{+}_{\mathsf{ac}}=\{0\}$. Therefore, we have $x\cong\mathsf{Q}_{\rho}\mathsf{Q}(x)$, which then lies in $\mathsf{K}^{+}(\Inj \Lambda)$.
        \item $\mathcal{T}^{\mathsf{b}}=\mathsf{K}^{\mathsf{b}}(\Inj \Lambda)$. Since $\mathcal{T}^\mathsf{b}$ is $\prod$-broad, $\mathsf{K}^{\mathsf{b}}(\Inj \Lambda)=\mathsf{broad}_{\prod}(E)$ for an injective cogenerator $E$ of $\Mod\Lambda$ and since $E$ lies in $\T^\mathsf{b}$, we get that $\mathsf{K}^\mathsf{b}(\Inj\Lambda)$ lies in $\T^\mathsf{b}$. For the converse inclusion, let $x$ be an object of $\T^\mathsf{b}\subseteq \T^+=\mathsf{K}^+(\Inj\Lambda)$. Then we have that $x\cong \mathsf{Q}_\rho\mathsf{Q}(x)$. From Lemma \ref{far away of smod} we have
$$\mathcal{T}^{\mathsf{b}}=\mathsf{Q}_\rho(\smod\Lambda)^{\padova}=\mathsf{Q}_\rho(\Mod\Lambda)^{\padova}.$$
Now, for any $y$ in $\Mod\Lambda$, we have
        $$        \mathsf{Hom}_{\mathcal{T}}(\mathsf{Q}_{\rho}(y),x[n])\cong \mathsf{Hom}_{\mathcal{T}}(\mathsf{Q}_\rho(y),\mathsf{Q}_\rho\mathsf{Q}(x)[n])\cong \mathsf{Hom}_{\mathsf{D}(\Lambda)}(y,\mathsf{Q}(x)[n])$$
showing that if $x$ lies in $\mathcal{T}^{\mathsf{b}}$ then $\mathsf{Q}(x)$ lies in the subcategory $(\Mod \Lambda)^{\padova}$ of $\mathsf{D}(\Lambda)$, i.e.~ $\mathsf{Q}(x)$ lies in $\mathsf{K}^{\mathsf{b}}(\Inj \Lambda)$ and, therefore, $x\cong \mathsf{Q}_\rho\mathsf{Q}(x)$ lies in $\mathsf{K}^{\mathsf{b}}(\Inj\Lambda)$ as a subcategory of $\T$.
        \item $\mathcal{T}^-_p=\{x\colon H^n(x)=0 \text{ for }n\gg0\}$. For any $x$ in $\mathsf{K}(\Inj \Lambda)$ and any $y\cong \mathsf{Q}_\rho\mathsf{Q}(y)$ in $\mathcal{T}^{\mathsf{b}}$ we have
        \[
        \mathsf{Hom}_{\mathcal{T}}(x,y[n])\cong \mathsf{Hom}_{\mathcal{T}}(x,\mathsf{Q}_\rho\mathsf{Q}(y)[n])\cong \mathsf{Hom}_{\mathsf{D}(\Lambda)}(\mathsf{Q}(x),\mathsf{Q}(y)[n]).
        \]
        So $x$ lies in $\mathcal{T}^{-}_p$ if and only if $\mathsf{Q}(x)$ lies in $^{\perp_{\ll}}(\mathsf{K}^{\mathsf{b}}(\Inj \Lambda))$ inside $\mathsf{D}(\Lambda)$. However, if $E$ is an injective cogenerator of $\Mod\Lambda$, we know that $H^{-n}(\mathsf{Q}(x))=0$ if and only if $\mathsf{Hom}_{\mathsf{D}(\Lambda)}(\mathsf{Q}(x),E[n])=0$. Hence, we conclude that  $x$ lies in $\mathcal{T}^{-}_p$  if and only if $H^{-n}(\mathsf{Q}(x))=0$ for all $n\ll 0$, i.e. $\mathsf{Q}(x)$ lies in $\mathsf{D}^{-}(\Lambda)$. This means that $x$ lies in $\mathsf{K}_{\mathsf{ac}}(\Inj \Lambda)\ast \mathsf{Q}_{\rho}(\mathsf{D}^-(\Lambda))$, i.e.~$\mathcal{T}^{-}_p$ consists of the complexes in $\mathsf{K}(\Inj \Lambda)$ with bounded above cohomology. 

       \item $\mathcal{T}^+_p=\{x\colon H^n(x)=0 \text{ for }n\ll0\}$. By the same argument as above we see that $x$ lies in $\mathcal{T}^{+}_p$ if and only if $\mathsf{Q}(x)$ lies in $^{\perp_{\gg} }(\mathsf{K}^{\mathsf{b}}(\Inj \Lambda))$ in $\mathsf{D}(\Lambda)$, which coincides with the subcategory of complexes with bounded below cohomology. 
        
        \item $\mathcal{T}^{\mathsf{b}}_p=\{x\colon \#\{n\colon H^n(x)\neq 0\}<\infty \}$. By the same argument as above we see that $x$ lies in $\mathcal{T}^{\mathsf{b}}_p$ if and only if $\mathsf{Q}(x)$ lies in $^{\padova}\mathsf{K}^{\mathsf{b}}(\Inj \Lambda)=\mathsf{D}^{\mathsf{b}}(\Lambda)$ in $\mathsf{D}(\Lambda)$. This means that $\mathcal{T}^{\mathsf{b}}_p $ coincides with $\mathsf{K}_{\mathsf{ac}}(\Inj \Lambda)\ast \mathsf{Q}_\rho(\mathsf{D}^{\mathsf{b}}(\Lambda))$, i.e. with the subcategory of the complexes in $\mathsf{K}(\Inj \Lambda)$ with only finitely many non-zero cohomologies. \qedhere
    \end{itemize}
\end{proof}

As immediate consequences of the proposition above we have the following corollaries. Note that the first of these corollaries, as observed in the paragraph before Lemma \ref{far-away and containments}, is already known (see \cite[Corollary 3.3]{chen_li_zhang_zhao}). 

\begin{cor}
    For an Artin algebra $\Lambda$, the singularity category $\mathsf{D}_{\mathsf{sg}}(\Lambda)$ has no silting object. 
\end{cor}
\begin{proof}
    This follows from Corollary \ref{existence of silting compact}, since $\mathsf{D}_{\mathsf{sg}}(\Lambda)$ is the subcategory of compact objects of $\mathcal{T}=\mathsf{K}_{\mathsf{ac}}(\Inj \Lambda)$ and $\mathcal{T}^{\mathsf{b}}=0$. 
\end{proof}

\begin{cor}
For an Artin algebra $\Lambda$, $\mathsf{K}_{\mathsf{ac}}(\Inj\Lambda)$ \textnormal{(}and, consequently, $\mathsf{D}_{\mathsf{sg}}(\Lambda)$\textnormal{)} is intrinsically determined in $\mathsf{K}(\Inj\Lambda)$, i.e.~it does not depend on $\Lambda$.
\end{cor}
\begin{proof}
Consider $\T=\mathsf{K}(\Inj\Lambda)$. We observe that the pair $(\mathsf{K}_{\mathsf{ac}}(\Inj\Lambda), \mathsf{Q}_\rho(\mathsf{D}^{\mathsf{b}}(\Mod{\Lambda})))$ is a stable t-structure in $\Tb_p$, and from the proof of Lemma \ref{far away of smod} we have that $\mathsf{Q}_{\rho}(\mathsf{D}^{\mathsf{b}}(\Mod{\Lambda}))=\mathsf{broad}_\Sigma(\T^{\mathsf{c}})$. This means that $\mathsf{K}_{\mathsf{ac}}(\Inj\Lambda)$ is a colocalisation of $\Tb_p$ and therefore equivalent to the Verdier quotient $\Tb_p/\mathsf{broad}_\Sigma(\T^{\mathsf{c}})$. Finally note that $\mathsf{D}_{\mathsf{sg}}(\Lambda)$ is the subcategory of compact objects in the latter, i.e.~$\mathsf{D}_{\mathsf{sg}}(\Lambda)=(\Tb_p/\mathsf{broad}_\Sigma(\T^{\mathsf{c}}))^{\mathsf{c}}$.
\end{proof}

\begin{rem}
    Observe now that the computations of Proposition \ref{examples2} do not hold for every noetherian ring. Consider $R$ to be as in Remark \ref{counterexample}. We then claim that for $\mathcal{T}=\mathsf{K}(\Inj R)$, we have
    \[
    \mathcal{T}^{\mathsf{b}}\neq \mathsf{K}^{\mathsf{b}}(\Inj R). 
    \]
    Indeed, we have $\mathsf{K}^{\mathsf{b}}(\Inj R)=(\mathsf{Q}_{\rho}(\Mod R))^{\padova}$ (see the proof of Proposition \ref{examples2}), so if $\mathcal{T}^{\mathsf{b}}=(\mathsf{Q}_{\rho}(\smod R))^{\padova}=(\mathsf{Q}_{\rho}(\Mod R))^{\padova}$ would hold, then we would get $(\smod R)^{\padova}=(\Mod R)^{\padova}$ in $\mathsf{D}(R)$ - and this is not the case.
\end{rem}

As promised earlier the computations of Proposition \ref{examples2} allow us to observe that neither $\mathsf{D}^{\mathsf{b}}(\smod\Lambda)$ nor $\mathsf{D}_{\mathsf{sg}}(\Lambda)$, for an Artin algebra of infinite global dimension, are far-away generated.

\begin{cor}\label{far-away gen db and dsg}
Let $\Lambda$ be an Artin algebra of infinite global dimension. Then $\mathsf{D}^{\mathsf{b}}(\smod\Lambda)$ and $\mathsf{D}_{\mathsf{sg}}(\Lambda)$ are not far-away generated.
\end{cor}
\begin{proof}
This follows from Lemma \ref{far-away and containments} and Proposition \ref{examples2} by noticing that both triangulated categories are the compact objects in a compactly generated triangulated category (namely for $\mathcal{T}=\mathsf{K}(\Inj\Lambda)$ and $\mathcal{T}=\mathsf{K}_{\mathsf{ac}}(\Inj\Lambda)$ respectively) and the inclusion $\mathcal{T}^\mathsf{c}\subseteq \mathcal{T}^{\mathsf{b}}$ is not verified. Note that $\mathsf{D}_{\mathsf{sg}}(\Lambda)$ is idempotent-complete (see \cite[Corollary 2.4]{chen}), therefore indeed coinciding with the compact objects in $\mathsf{K}_{\mathsf{ac}}(\Inj{\Lambda})$.
\end{proof}

The next lemma completes the previous table for Artin algebras with an additional condition. Recall that $\Lambda$ is \textbf{Iwanaga-Gorenstein} if $\mathsf{K}^{\mathsf{b}}(\Proj\Lambda)$ and $\mathsf{K}^{\mathsf{b}}(\Inj\Lambda)$ coincide as subcategories of $\mathsf{D}(\Lambda)$ (see \cite{happel} for details). 

\begin{lem} \label{subcategories for IG}
If $\Lambda$ is an Iwanaga-Gorenstein Artin algebra, for $\mathcal{T}=\mathsf{K}(\Inj \Lambda)$ and $\?$ in $\{-,+,\mathsf{b}\}$, we have
    \[
    \mathcal{T}^{\?}_i=\mathsf{K}_{\mathsf{ac}}(\Inj\Lambda)\ast \mathsf{Q}_{\rho}(\mathsf{D}^{\?}(\Lambda)).
    \]
\end{lem}
\begin{proof}
    Let $e$ be an injective $\Lambda$-module. We first show that if $x$ is acyclic, then $\mathsf{Hom}_{\mathsf{K}(\Inj \Lambda)}(e,x)=0$. In fact, since $\Lambda$ is noetherian, $e$ is a direct sum of indecomposable $\Lambda$-modules. Therefore, it suffices to show our claim for $e$ indecomposable injective. Since $\Lambda$ is an Artin algebra, $e$ is then finitely presented and therefore it lies in $\T^c=\mathsf{Q}_\rho(\mathsf{D}^\mathsf{b}(\smod \Lambda))$. Since $\Lambda$ is Iwanaga-Gorenstein, $\mathsf{Q}(e)$ is a finitely presented $\Lambda$-module of finite projective dimension and, thus, it lies in $\mathsf{K}^\mathsf{b}(\proj\Lambda)$.  Hence, we have  by Lemma \ref{some facts about KInj} that $e\cong \mathsf{Q}_{\lambda}\mathsf{Q}(e)$. Write $x=\mathsf{I}(x)$ for an acyclic complex of injectives. Then, as claimed, we have 
    \[
    \mathsf{Hom}_{\mathsf{K}(\Inj \Lambda)}(e,x)= \mathsf{Hom}_{\mathsf{K}(\Inj \Lambda)}(\mathsf{Q}_{\lambda}\mathsf{Q}(e),\mathsf{I}(x))=\mathsf{Hom}_{\mathsf{K}_{\mathsf{ac}}(\Inj \Lambda)}(\mathsf{I}_{\lambda}\mathsf{Q}_{\lambda}\mathsf{Q}(e),x)=0. 
    \]
    Now, since $\mathsf{K}^{\mathsf{b}}(\Inj\Lambda)=\mathsf{broad}_\Sigma(\Lambda^*)$, then $(\T^{\mathsf{b}})^{\?}=(\Lambda^*)^{\?}$ and, clearly, we get $\mathsf{K}_{\mathsf{ac}}(\Inj \Lambda)\subseteq (\mathcal{T}^{\mathsf{b}})^{\?}$ for every $\?$ in $\{\perp_{\ll}, \perp_{\gg},\padova\}$. For an object $x$ in $(\mathcal{T}^{\mathsf{b}})^{\?}$, consider the triangle 
    \[
    \mathsf{II}_{\rho}(x)\rightarrow x\rightarrow \mathsf{Q}_{\rho}\mathsf{Q}(x)\rightarrow \mathsf{II}_{\rho}(x)[1]. 
    \]
    By the above, we know that $\mathsf{II}_{\rho}(x)$ lies in $(\mathcal{T}^{\mathsf{b}})^{\?}$ and so $\mathsf{Q}_{\rho}\mathsf{Q}(x)$ lies in $(\mathcal{T}^{\mathsf{b}})^{\?}$. Given $y$ in $\mathsf{K}^{\mathsf{b}}(\Inj \Lambda)$ we have 
    \[
    \mathsf{Hom}_{\mathsf{D}(\Lambda)}(\mathsf{Q}(y),\mathsf{Q}(x)[n])\cong \mathsf{Hom}_{\mathsf{K}(\Inj \Lambda)}(y,\mathsf{Q}_{\rho}\mathsf{Q}(x)[n])=0
    \]
    for appropriate values of $n$. In other words, $\mathsf{Q}(x)$ belongs to the corresponding far-away orthogonal $\mathsf{K}^{\mathsf{b}}(\Inj \Lambda)^{\?}$ in $\mathsf{D}(\Lambda)$. However, since $\Lambda$ is assumed to be Iwanaga-Gorenstein, we have $\mathsf{K}^{\mathsf{b}}(\Inj \Lambda)=\mathsf{K}^{\mathsf{b}}(\Proj \Lambda)$ and thus $\mathsf{Q}(x)$ lies in $\mathsf{D}^-(\Lambda)$ when $x$ is an object of $\mathcal{T}^-_i$, $\mathsf{Q}(x)$ lies in $\mathsf{D}^+(\Lambda)$ when $x$ is an object of $\mathcal{T}^+_i$ and $\mathsf{Q}(x)$ lies in $\mathsf{D}^{\mathsf{b}}(\Lambda)$ when $x$ is an object of $\mathcal{T}^{\mathsf{b}}_i$. 
\end{proof}

The distinguished subcategories calculated above can be reinterpreted in $\mathsf{K}(\Proj \Lambda)$ for an Artin algebra $\Lambda$. Recall that for an Artin algebra, there is a well-known equivalence between $\Proj\Lambda$ and $\Inj\Lambda$ that extends to an equivalence between $\mathsf{K}(\Proj\Lambda)$ and $\mathsf{K}(\Inj\Lambda)$, given by tensoring with $\Lambda^*$. This is known as the \textit{Nakayama functor} and $\Lambda^*$ plays a role known as \textit{the dualising complex}. To see this translation explicitly, let us first consider a noetherian ring $R$ which admits a dualising complex $D_R$ which, by \cite{iyengar_krause}, yields an equivalence 
\begin{align} \label{krause iyengar}
  T:=  -\otimes_{R}D_R\colon\mathsf{K}(\Proj R)\longrightarrow\mathsf{K}(\Inj R).
\end{align}
A quasi-inverse of $T$ is then given by $\rho\circ \mathsf{Hom}_R(D_R,-)$ where the functor $\rho$ denotes the right adjoint to the inclusion $\mathsf{K}(\Proj R)\hookrightarrow\mathsf{K}(\Mod R)$ (whose existence follows from \cite[Theorem 2.4]{jorgensen2}). 

\begin{lem} \label{restriction of T}
    Let $R$ be a noetherian ring that admits a dualising complex $D_R$. The following hold. 
    \begin{enumerate}
        \item[\textnormal{(i)}] $T$ restricts to an equivalence $\mathsf{K}^+(\Proj R)\simeq \mathsf{K}^+(\Inj R)$. 
        \item[\textnormal{(ii)}] $T$ restricts to an equivalence $\mathsf{K}^{\mathsf{b}}(\Proj R)\simeq \mathsf{K}^{\mathsf{b}}(\Inj R)$. 
        \item[\textnormal{(iii)}] For a complex $x$ that lies in $\mathsf{K}(\Proj R)$, we have $H^n(\mathsf{Hom}_R(x,P))=0$ for every $P\in\Proj R$ and $n\gg 0$ if and only if $H^n(T(x))=0$ for $n\gg 0$. 
        \item[\textnormal{(iv)}] For a complex $x$ that lies in $\mathsf{K}(\Proj R)$, we have $H^n(\mathsf{Hom}_R(x,P))=0$ for every $P\in\Proj R$ and $n\ll 0$ if and only if $H^n(T(x))=0$ for $n\ll 0$. 

        \item[\textnormal{(v)}] For a complex $x$ that lies in $\mathsf{K}(\Proj R)$, we have $H^n(\mathsf{Hom}_R(x,P))=0$ for every $P\in\Proj R$ and $|n|\gg 0$ if and only if $H^n(T(x))=0$ for $|n|\gg 0$. 
    \end{enumerate}
\end{lem}
\begin{proof}
    (i) Observe that for a projective $R$-module $P$, the module $P\otimes_R D_R$ is injective. From the latter follows that the functor $T$ maps $\mathsf{K}^{+}(\Proj R)$ to $\mathsf{K}^+(\Inj R)$. Let now $x$ be an object in $\mathsf{K}^+(\Inj R)$. Then the complex $\mathsf{Hom}_R(D_R,x)$ consists of flat modules (see for instance \cite[Lemma 2.13]{OPS}). Since flat modules have finite projective dimension (see \cite{jorgensen}), it follows that $\rho(\mathsf{Hom}_R(D_R,x))$ is a bounded below (see \cite[Proposition 2.5]{OPS}). 
    
    (ii) This is similar to (i), see also \cite[Lemma 2.14]{OPS}. 
    
    (iii) Let $\mathcal{V}$ denote the subcategory of $\mathsf{K}(\Inj R)$ that consists with complexes that are acyclic in sufficiently large degrees. Then we have 
    \[
    \mathcal{V}={^{\perp_{\ll}}(\mathsf{K}^{\mathsf{b}}(\Inj R))}\simeq {^{\perp_{\ll}}(\mathsf{K}^{\mathsf{b}}(\Proj R))}={^{\perp_{\ll}}(\Proj R)},
    \]
    where the first equality holds by the proof of Proposition \ref{examples2} (note that we do not need $R$ to be an Artin algebra for this equality) and the equivalence that follows holds by (ii). We will now show that ${^{\perp_{\ll}}(\Proj R)}$ coincides with the wanted subcategory of $\mathsf{K}(\Proj R)$, for which it is enough to show that, for any $P$ in $\Proj R$, and for a complex $x$ in $\mathsf{K}(\Proj R)$, the complex $\mathsf{Hom}_R(x,P)$ is acyclic in large enough degrees if and only if $x$ lies in ${^{\perp_{\ll}}P}$. Write 
$$x=\cdots\rightarrow{P^{i-1}}\xrightarrow{d^{i-1}}P^i\xrightarrow{d^i}P^{i+1}\rightarrow \cdots.$$ 
Note that, for all $i$ in $\mathbb{Z}$, we have $\mathsf{Ker}(\mathsf{Hom}_R(d^{i-1},P))\cong \mathsf{Hom}_{\mathsf{C(R)}}(x,P[-i]).$ Now, we have $H^i(\mathsf{Hom}_R(x,P))=0$ if and only if $\mathsf{Ker}(\mathsf{Hom}_R(d^{i-1},P))=\mathsf{Im}(\mathsf{Hom}_R(d^i,P))$, i.e. if and only if every morphism in $\mathsf{Hom}_{\mathsf{C(R)}}(x,P[-i])$ factors through $d^i$. This is equivalent to say that there is $h^{i+1}\colon P^{i+1}\longrightarrow P$ such that $f=h^{i+1}d^i$. In particular, the complex $\mathsf{Hom}_R(x,P)$ is acyclic in large degrees if and only if every map of complexes $x\longrightarrow P[-i]$ is null-homotopic, for $i\gg0$.


    (iv) and (v) are similar to (iii).
\end{proof}

Taking into account the proposition above, we can translate our computations of distinguished subcategories of $\mathsf{K}(\Inj\Lambda)$ to the corresponding subcategories of $\mathsf{K}(\Proj\Lambda)$.

\begin{cor} \label{subcategories for K(Proj)}
    The following hold for $\mathcal{T}=\mathsf{K}(\Proj \Lambda)$ where $\Lambda$ is an Artin algebra.   
    \begin{itemize}
        \item[\textnormal{(i)}] $\mathcal{T}^+=\mathsf{K}^+(\Proj \Lambda)$. 
        \item[\textnormal{(ii)}] $\mathcal{T}^{\mathsf{b}}=\mathsf{K}^{\mathsf{b}}(\Proj \Lambda)$. 
        \item[\textnormal{(iii)}] $\mathcal{T}^-_p=\{x: H^n(\mathsf{Hom}_R(x,P))=0 \text{ for }n\gg0 \text{ and every }P\in\Proj \Lambda\}$. 
        \item[\textnormal{(iv)}] $\mathcal{T}^+_p=\{x: H^n(\mathsf{Hom}_R(x,P))=0 \text{ for }n\ll0 \text{ and every }P\in\Proj \Lambda\}$.
        \item[\textnormal{(v)}] $\mathcal{T}^{\mathsf{b}}_p=\{x: H^n(\mathsf{Hom}_R(x,P))=0 \text{ for }|n|\gg0 \text{ and every }P\in\Proj \Lambda\}$. 
    \end{itemize}
    In particular, when $\Lambda$ is Iwanaga-Gorenstein, the following hold. 
    \begin{itemize}
       \item[\textnormal{(vi)}] $\mathcal{T}^-_i=\{x: H^n(\mathsf{Hom}_R(x,P))=0 \text{ for }n\gg0 \text{ and every }P\in\Proj \Lambda\}$. 
        \item[\textnormal{(vii)}] $\mathcal{T}^+_i=\{x: H^n(\mathsf{Hom}_R(x,P))=0 \text{ for }n\ll0 \text{ and every }P\in\Proj \Lambda\}$. 
        \item[\textnormal{(viii)}] $\mathcal{T}^{\mathsf{b}}_i=\{x: H^n(\mathsf{Hom}_R(x,P))=0 \text{ for }|n|\gg0 \text{ and every }P\in\Proj \Lambda\}$. 
    \end{itemize}
\end{cor}
\begin{proof}
    The equivalence $\mathsf{K}(\Proj \Lambda)\simeq \mathsf{K}(\Inj \Lambda)$ restricts to an equivalence between all the distinguished subcategories of $\mathsf{K}(\Proj \Lambda)$ and $\mathsf{K}(\Inj \Lambda)$ respectively. Thus the result follows from Proposition \ref{examples2} and Lemma \ref{subcategories for IG}, using Lemma \ref{restriction of T}.
\end{proof}

\subsection{Example: homotopy categories beyond Artin algebras}
The computations of the previous section for Artin algebras were possible due to condition (\ref{star condition}) of Lemma \ref{far away of smod}, which itself depends on Corollary \ref{broad Db}. As seen in Example \ref{counterexample}, there are noetherian counterexamples to these statements. In this subsection we will comment on other settings where these two results hold and where, therefore, we might be able to compute the distinguished subcategories of the corresponding homotopy category of injectives/projectives.

\begin{exmp}\label{finite global dimension}
We start with a somewhat dull example. Let $R$ be a \textbf{right noetherian ring with finite right global dimension}. Then we have the equalities
 \[
    \mathsf{D}^{\mathsf{b}}(\Mod R)=\mathsf{K}^{\mathsf{b}}(\Proj R)=\mathsf{broad}_{\Sigma}(\mathsf{K}^{\mathsf{b}}(\proj R))=\mathsf{broad}_{\Sigma}(\mathsf{D}^{\mathsf{b}}(\smod R))
    \]
in $\mathsf{K}(\Inj R)\simeq \mathsf{D}(R)$, and both condition (\ref{star condition}) and Corollary \ref{broad Db} hold for such rings. Of course that, since $\mathsf{K}(\Inj R)\simeq \mathsf{D}(R)$, the distinguished subcategories for $\mathsf{K}(\Inj R)$ are already computed in Proposition \ref{subcategories for the derived category}.
\end{exmp}

\begin{exmp} \label{gorenstein orders}
    
We can verify that condition (\ref{star condition}) holds for any \textbf{Gorenstein $R$-order of finite CM type}, where $R$ is a commutative noetherian complete CM local ring. For such an algebra, say $A$, denote by $\Gproj A$ the category of finitely generated Gorenstein-projective $A$-modules, by $\GProj A$ the category of all Gorenstein projective $A$-modules and by $\mathcal{IG}_{A}$ the class of indecomposable objects in $\Gproj A$. It follows that 
    \[
    \mathsf{D}^{\mathsf{b}}(\Mod A)=\mathsf{thick}(\GProj A)=\mathsf{thick}(\Add (\Gproj A))=\mathsf{thick}(\mathsf{Add}(\mathcal{IG}_A))=\mathsf{thick}(\cup_{x\in \mathcal{IG}_A}\Add x)=\mathsf{broad}_{\Sigma}(\mathcal{IG}_A)
    \]
    where the first equality follows by taking minimal Gorenstein projective resolutions of $A$-modules, the second follows by the fact that every Gorenstein projective module is a direct sum of finitely generated Gorenstein projectives (see \cite[Theorem 1.2]{nakamura}) and the fourth follows from the fact that there are finitely many indecomposable objects among $\Gproj A$, up to isomorphism. In particular, since $\mathsf{thick}(\mathcal{IG}_A)=\mathsf{thick}(\Gproj A)=\mathsf{D}^{\mathsf{b}}(\smod A)$, the analogous of Corollary \ref{broad Db} follows for $A$, i.e. 
    \[
    \mathsf{D}^{\mathsf{b}}(\Mod A)=\mathsf{broad}_{\Sigma}(\mathsf{D}^{\mathsf{b}}(\smod A))
    \]
    as subcategories of $\mathsf{D}(A)$. By arguing as in the proof of Lemma \ref{far away of smod} we see that condition (\ref{star condition}) holds. It follows that the computation of distinguished subcategories for $\mathsf{K}(\Inj A)$ follows exactly the table of Proposition \ref{examples2}.
    \end{exmp}

We change strategy to provide another class of examples, building on results of Aoki \cite{aoki} and Neeman \cite{neeman5, neeman4}. Let us first recall some notation, for which we fix a triangulated category $\mathcal{T}$ that has coproducts. Given a collection of objects $\mathcal{X}$ in $\mathcal{T}$, we denote by $\mathsf{Coprod}(\mathcal{X})$ the subcategory $\mathcal{T}$ that consists of coproducts of objects of $\mathcal{X}$. Then, inductively, we define 
\[
\mathsf{Coprod}_1(\mathcal{X})\coloneqq \mathsf{Coprod}(\mathcal{X}) \text{   and   } \mathsf{Coprod}_n(\mathcal{X})\coloneqq \mathsf{Coprod}_{1}(\mathcal{X})\ast\mathsf{Coprod}_{n-1}(\mathcal{X})
\]
and set $\overline{\langle \mathcal{X}\rangle}_n\coloneqq \mathsf{smd}(\mathsf{Coprod}_n(\mathcal{X}))$, where $\mathsf{smd}$ stands for the closure under summands, see \cite{neeman5} for details. When $\mathcal{X}$ is given by $\{x[i],\  i\in I\}$ for some object $x$ and $I\subseteq \mathbb{Z}$, then we denote $\overline{\langle \mathcal{X}\rangle}_n$ by $\overline{\langle x[I]\rangle}_n$. If $x$ is an object for which there is an integer $n\geq 1$ such that $\overline{\langle x[\mathbb{Z}]\rangle}_n=\mathcal{T}$, then we say that $x$ is a \textbf{strong generator for $\mathcal{T}$} and $\mathcal{T}$ is said to be \textbf{strongly generated by $x$}.


\begin{prop} \label{bounded objects when T admits a strong gen}
    Let $\mathcal{T}$ be a triangulated category with coproducts that has a compact silting object. If $\mathcal{T}$ admits a strong generator $x$ which lies in $\mathcal{T}^{\mathsf{b}}$, then $\mathcal{T}^{\mathsf{b}}=\mathsf{broad}_{\Sigma}(x)$. 
\end{prop}
\begin{proof}
 Let $n\geq 1$ be such that $\overline{\langle x[\mathbb{Z}]\rangle}_n=\mathcal{T}$. Since $x$ lies in $\mathcal{T}^{\mathsf{b}}$, it is clear that $\mathsf{broad}_{\Sigma}(x)$ is contained in $\mathcal{T}^{\mathsf{b}}$. Denote by $M$ the compact silting generator of $\mathcal{T}$ and consider an object of $\mathcal{T}^{\mathsf{b}}$, say $t$. We know from Proposition \ref{restrictions of t-structure and co-t-structure} that the t-structure associated to $M$ restricts to a bounded t-structure in $\mathcal{T}^{\mathsf{b}}$. Thus, there are $a,b$ in $\mathbb{Z}$ such that $t$ belongs to $M^{\perp_{> a}}\cap M^{\perp_{\leq b}}$. It follows by the assumption and \cite[Lemma 9.4]{neeman5} that $t$ lies in $\overline{\langle x[[-d,d]]\rangle}_n$ for some $d\geq 1$ (that depends on $a,b$). By noticing that $\overline{\langle x[[-d,d]]\rangle}_n\subseteq \mathsf{broad}_{\Sigma}(x)$, the result follows. 
\end{proof}

The proposition above allows us to derived the final example of this subsection.
\begin{exmp}
Let $R$ be a \textbf{commutative noetherian quasiexcellent ring of finite Krull dimension}. Rather than defining this class of rings, it suffices to say that most commutative noetherian rings appearing \textit{naturally} in algebra and geometry are of this kind. They include, for example, any finitely generated commutative algebra over a Dedekind domain. For such rings, it follows from \cite[Main Theorem]{aoki} and its proof that there exists an object $x$ in $\mathsf{D}^{\mathsf{b}}(\smod R)$ for which $\mathsf{D}(R)=\overline{\langle x[\mathbb{Z}]\rangle}_n$ for some $n\geq 1$. By \cite[Lemma 2.7]{neeman4} (see also \cite[Proposition 9.8]{neeman5}), $x$ is a strong generator of $\mathsf{D}^{\mathsf{b}}(\smod R)$ and in particular, $\mathsf{D}^{\mathsf{b}}(\smod R)=\mathsf{thick}(x)$. Finally, Proposition \ref{bounded objects when T admits a strong gen} allows us to then conclude that $\mathsf{D}^{\mathsf{b}}(\Mod R)=\mathsf{broad}_{\Sigma}(x)$ and therefore the analogue of Corollary \ref{broad Db}, i.e.
    \[
    \mathsf{D}^{\mathsf{b}}(\Mod R)=\mathsf{broad}_{\Sigma}(\mathsf{D}^{\mathsf{b}}(\smod R)).
    \]
We then have that Lemma \ref{far away of smod} and, consequently, the computation of distinguished subcategories of Proposition \ref{examples2} hold for $\mathsf{K}(\Inj R)$. We also refer to \cite[Theorem 7.26 and Theorem 8.12]{manali-rahul} where some of the subcategories of $\mathsf{K}(\Inj R)$ are computed using different techniques. 

    \end{exmp}



\section{Homological properties for triangulated categories}
\label{regulartriangcategories} 

In this section we introduce various homological properties for a compactly generated triangulated category $\mathcal{T}$ which depend solely on the distinguished subcategories we defined in the previous section. This means in particular that all these notions are intrinsic and are, therefore, preserved under triangle equivalences. 

Recall that a ring has finite (right) global dimension if every (right) $R$-module has finite projective dimension or, equivalently, if every (right) $R$-module has finite injective dimension. This is a property that can be read in the derived category: indeed, as mentioned in the introduction, $R$ has finite global dimension if and only if $\mathsf{D}^{\mathsf{b}}(\Mod{R})=\mathsf{K}^{\mathsf{b}}(\Proj{R})$ or, equivalently, if and only if $\mathsf{D}^{\mathsf{b}}(\Mod{R})=\mathsf{K}^{\mathsf{b}}(\Inj{R})$, as subcategories of $\mathsf{D}(R)$. The equivalence of these two statements is completely general, as shown in the following lemma.

\begin{lem} \label{regularity via injectives}
For a compactly generated triangulated category $\mathcal{T}$, we have $\Tb=\Tb_p$ if and only if $\mathcal{T}^{\mathsf{b}}=\mathcal{T}^{\mathsf{b}}_i$. 
\end{lem}
\begin{proof}
This follows from Lemma \ref{common facts of the subcategories}, since $(\Tb_p,\Tb,\Tb_i)$ is a far-away triple.
\end{proof}

Recall that by a \textbf{Gorenstein} ring we mean a ring $R$ for which $\mathsf{K}^{\mathsf{b}}(\Proj{R})=\mathsf{K}^{\mathsf{b}}(\Inj{R})$ as subcategories of $\mathsf{D}(R)$. When $R$ is two-sided noetherian, we revert to the earlier terminology of calling $R$ Iwanaga-Gorenstein (see Section \ref{Section 3}). 
Recall also that according to \cite{rickard}, \textbf{injectives generate for a ring $R$} if $\mathsf{Loc}(\Inj R)=\mathsf{D}(R)$. In case $R$ is a finite-dimensional algebra over a field, if injectives generate for $R$, then $R$ has \textbf{finite finitistic dimension} ($\Findim R<\infty$), i.e.~there is a bound to the projective dimensions of $R$-modules of finite projective dimension (see \cite[Theorem 4.3]{rickard}). These are the four homological concepts we want to discuss: finiteness of the global dimension, gorensteinness, generation by injectives and finiteness of the finitistic dimension. 

\begin{defn} \label{homological}
    We say that a compactly generated triangulated category $\mathcal{T}$ is  
    \begin{enumerate} 
        \item[(i)] of \textbf{finite global dimension} ($\gd \mathcal{T}<\infty$ for short) if $\mathcal{T}^{\mathsf{b}}_p=\mathcal{T}^{\mathsf{b}}$ or, equivalently, if $\Tb=\Tb_i$; 
        \item[(ii)] \textbf{Gorenstein} if $\mathcal{T}^{\mathsf{b}}_p=\mathcal{T}^{\mathsf{b}}_i$;  
        \item[(iii)] \textbf{generated by injectives} if $\mathcal{T}=\mathsf{Loc}(\mathcal{T}^{\mathsf{b}}_i)$; 
        \item[(iv)] of \textbf{finite finitistic dimension} ($\Findim\mathcal{T}<\infty$ for short) if $(\mathcal{T}^{\mathsf{b}}_i)^{\perp}\cap \mathcal{T}^+=\{0\}$.  
    \end{enumerate}
\end{defn}

While the first three definitions have rather clear motivations, we base ourselves in \cite[Theorem 4.4]{rickard} and \cite[Theorem B.1]{krause} to motivate the fourth definition (see Example \ref{bad example} below for a discussion of its adequateness). Let us show that these definitions are good generalisations of the classical concepts in homological algebra.
%
%

\begin{thm}\label{properties}
Let $R$ be a ring. Then we have 
\begin{itemize}
\item $\gd \mathsf{D}(R)<\infty$ if and only if $\gd R<\infty$;
\item $\mathsf{D}(R)$ is Gorenstein if and only if $R$ is Gorenstein;
\item $\mathsf{D}(R)$ is generated by injectives if and only if injectives generate for $R$.
\end{itemize}
If $R=\Lambda$ is an Artin algebra, then we also have
\begin{itemize}
\item $\Findim \mathsf{D(\Lambda)}<\infty$ if and only if $\Findim \Lambda<\infty$;
\item $\gd \mathsf{K}(\Inj \Lambda)<\infty$ if and only if $\gd \Lambda<\infty$;
\item $\gd \mathsf{K}_{\mathsf{ac}}(\Inj\Lambda)<\infty$ if and only if $\mathsf{K}_{\mathsf{ac}}(\Inj\Lambda)=0$;
\item $\mathsf{K}_{\mathsf{ac}}(\Inj\Lambda)$ is always Gorenstein, generated by injectives and $\Findim \mathsf{K}_{\mathsf{ac}}(\Inj\Lambda)<\infty$.

\end{itemize}
\end{thm}

\begin{proof}
        The statements about global dimension follow easily for all categories from the tables Proposition \ref{examples} and Proposition \ref{examples2}. A similar argument applies for the relation between the gorensteinness of $\mathsf{D}(R)$ and that of $R$ proves our second claim. For the injective generation of $\mathsf{D}(R)$, observe that $\mathsf{K}^{\mathsf{b}}(\Inj R)\subseteq \mathsf{Loc}(\Inj R)$ and thus $\mathsf{Loc}(\mathsf{K}^{\mathsf{b}}(\Inj R))=\mathsf{D}(R)$ if and only if $\mathsf{Loc}(\Inj R)=\mathsf{D}(R)$. The injective generation  property for $\mathsf{K}_{\mathsf{ac}}(\Inj\Lambda)$, when $\Lambda$ is an Artin algebra (over a commutative artinian ring $k$), is trivial. The fact that $\mathsf{D}(\Lambda)$ has finite finitistic dimension if and only $\Findim \Lambda<\infty$, for $\Lambda$ an Artin algebra, follows from \cite[Theorem 4.4(b)]{rickard}, which states that $\Findim R<\infty$ if and only if $\mathsf{Hom}_k(R,k)^{\perp}\cap \mathsf{D}^+=\{0\}$ (see also \cite[Theorem B.1.]{krause}). 
\end{proof}

Regarding the Gorenstein property, the generation by injectives or the finiteness of the finitistic dimension for $\mathsf{K}(\Inj{\Lambda})$ ($\Lambda$ being an Artin algebra), we are only able to discuss them in the cases where we know something about the subcategory $\Tb_i$ (see, for example, Lemma \ref{subcategories for IG}). We begin with the Gorenstein property.

\begin{lem} \label{gorensteinness of K(Inj)}
    Let $\Lambda$ be an Artin algebra and consider $\mathcal{T}=\mathsf{K}(\Inj \Lambda)$. The following hold\textnormal{:} 
    \begin{enumerate}
        \item If $\Lambda$ is Iwanaga-Gorenstein, then $\mathcal{T}$ is Gorenstein. 
        \item If $\mathcal{T}$ is Gorenstein, then $\mathsf{K}^\mathsf{b}(\Inj{\Lambda})\subseteq \mathsf{K}^\mathsf{b}(\Proj{\Lambda})$. 
    \end{enumerate}
    In particular, if the Gorenstein symmetry conjecture holds, then $\Lambda$ is Iwanaga-Gorenstein if and only if $\mathsf{K}(\Inj\Lambda)$ is Gorenstein.
\end{lem}
\begin{proof}
The first statement follows from Lemma \ref{subcategories for IG} and from Proposition \ref{examples2} by observing that
\[\mathsf{K}_{\mathsf{ac}}(\Inj\Lambda)\ast \mathsf{Q}_{\rho}(\mathsf{D}^{\mathsf{b}}(\Lambda))=\{x\in\mathsf{K}(\Inj\Lambda)\colon \#\{n\colon H^n(x)\neq 0\}<\infty\}.\]
Regarding (ii), the assumption $\mathcal{T}^{\mathsf{b}}_i=\mathcal{T}^{\mathsf{b}}_p$ shows in particular that $\mathsf{Q}_{\rho}(\mathsf{D}^{\mathsf{b}}(\Lambda))\subseteq \mathcal{T}^{\mathsf{b}}_i$ and therefore for every $x$ in $\mathsf{K}^{\mathsf{b}}(\Inj \Lambda)$ and $y$ in $\mathsf{D}^{\mathsf{b}}(\Lambda)$ we have 
    \[
    \mathsf{Hom}_{\mathsf{D}(\Lambda)}(\mathsf{Q}(x),y[n])=\mathsf{Hom}_{\mathsf{K}(\Inj \Lambda)}(x,\mathsf{Q}_{\rho}(y)[n])=0
    \]
    for $|n|\gg0$. In other words, we have $\mathsf{K}^{\mathsf{b}}(\Inj \Lambda)\subseteq {}^{\padova}(\mathsf{D}^{\mathsf{b}}(\Lambda))=\mathsf{K}^{\mathsf{b}}(\Proj \Lambda)$ in $\mathsf{D}(\Lambda)$ and so injective $A$-modules have finite projective dimension.  For the last assertion, note that  if the Gorenstein symmetry conjecture would hold (i.e.~there is an inclusion $\mathsf{K}^\mathsf{b}(\Inj{\Lambda})\subseteq \mathsf{K}^\mathsf{b}(\Proj{\Lambda})$ if and only if there in an inclusion $\mathsf{K}^\mathsf{b}(\Inj{\Lambda})\supseteq\mathsf{K}^\mathsf{b}(\Proj{\Lambda})$ in $\mathsf{D}(\Lambda)$) then the previous two assertions guarantee the claimed equivalence.\end{proof}


\begin{rem}
It is not reasonable to expect that $\mathsf{K}_{\mathsf{ac}}(\Inj{\Lambda})$ detects the gorensteinness of $\Lambda$. This comes, in part, from the fact that 
gorensteinness is not a singular invariant. Indeed, consider an Iwanaga-Gorenstein Artin algebra $A$ with $\gd A=\infty$, an Artin algebra $B$ with $\gd B<\infty$ and an $A$-$B$-bimodule $M$ such that $\pd_{A} M=\infty$. Then, we have an equivalence of singularity categories
    \[
    \mathsf{D}_{\mathsf{sg}}(\tiny\begin{pmatrix}
        A & M \\ 
        0 & B
    \end{pmatrix})\simeq \normalsize\mathsf{D}_{\mathsf{sg}}(A)
    \]
but the triangular matrix ring on the left-hand side is not Iwanaga-Gorenstein (see \cite[Theorem 3.3 and Theorem 4.1]{chen2}). As we have observed in Theorem \ref{properties} that the category $\mathsf{K}_{\mathsf{ac}}(\Inj{A})$ is always Gorenstein, for an Artin algebra $A$. It is furthermore known that not every Artin algebra is singular equivalent to an Iwanaga-Gorenstein algebra, see \cite{chen}. 
\end{rem}

For a ring $R$, it is well-known that $\gd R<\infty$ implies that $R$ is Gorenstein and if $R$ is Gorenstein then injectives generate for $R$, see \cite[Theorem 8.1]{rickard}. Furthermore, an Iwanaga-Gorenstein Artin algebra always has finite finitistic dimension (see \cite[Proposition 6.10]{auslander_reiten}), while injective generation for Artin algebras implies finiteness of the finitistic dimension \cite[Theorem 4.3]{rickard}. We will prove triangulated versions of these implications. 


\begin{thm} \label{theorem1}
The following implications hold for  a compactly generated triangulated category $\T$
\[\begin{xymatrix}{
\gd \T<\infty\ar@{=>}[r]& \T\ {\rm Gorenstein}\ar@{=>}[d]\ar@{=>}[r]&\T\ {\rm generated\ by\ injectives}\ar@{=>}[dl]^{(**)}\\ &\Findim \T<\infty&
}\end{xymatrix}\]
where the implication $(**)$ holds if  $\mathcal{T}^{\mathsf{b}}_i=\mathsf{broad}_{\Sigma}(S)$ for a set $S$.
%
%
%
%
%
\end{thm}
\begin{proof}
$\bullet\ \gd\T<\infty\Rightarrow \T\ {\rm Gorenstein}$: By applying $^{\padova}(-)$ to $\mathcal{T}^{\mathsf{b}}_p=\mathcal{T}^{\mathsf{b}}$ we infer $\mathcal{T}^{\mathsf{b}}=\mathcal{T}^{\mathsf{b}}_i$, see Lemma \ref{common facts of the subcategories}(iii). It follows in particular that $\mathcal{T}^{\mathsf{b}}_p=\mathcal{T}^{\mathsf{b}}_i$, i.e. $\mathcal{T}$ is Gorenstein. \\ 
 $\bullet\ \T\ {\rm Gorenstein}\Rightarrow \T\ {\rm generated\ by\ injectives}$: Simply observe that $\mathsf{Loc}(\mathcal{T}^{\mathsf{b}}_i)=\mathsf{Loc}(\mathcal{T}^{\mathsf{b}}_p)=\mathcal{T}$, where the first equality follows by gorensteinness and the second follows by $\mathcal{T}^{\mathsf{c}}\subseteq \mathcal{T}^{\mathsf{b}}_p$, see Lemma \ref{common facts of the subcategories}(i). \\ 
 $\bullet\ \T\ {\rm Gorenstein}\Rightarrow \Findim\T<\infty$: Since $\mathcal{T}^{\mathsf{c}}\subseteq \mathcal{T}^{\mathsf{b}}_p$ and $\mathcal{T}^{\mathsf{b}}_p=\mathcal{T}^{\mathsf{b}}_i$, it follows that $(\mathcal{T}^{\mathsf{b}}_i)^{\perp}\subseteq (\mathcal{T}^{\mathsf{c}})^{\perp}=\{0\}$ and in particular the intersection $(\mathcal{T}^{\mathsf{b}}_i)^{\perp}\cap \mathcal{T}^+$ is trivial. 

Finally, assume that $\mathcal{T}^{\mathsf{b}}_i=\mathsf{broad}_{\Sigma}(S)$ for some set $S$, and let us prove the implication $(\ast\ast)$. Since $S$ is a set, it follows by \cite[Theorem 2.3]{neeman3} that $(\mathsf{Loc}(S),S^{\perp})$ is a torsion pair. We have $\mathsf{Loc}(S)=\mathsf{Loc}(\mathcal{T}^{\mathsf{b}}_i)$ and $S^{\perp}=(\mathcal{T}^{\mathsf{b}}_i)^{\perp}$.
    Since injectives generate for $\mathcal{T}$, it follows that $(\mathcal{T}^{\mathsf{b}}_i)^{\perp}=\{0\}$. We have proved statements (a) and (b).
\end{proof}

Note that the assumption that $\mathcal{T}^{\mathsf{b}}_i=\mathsf{broad}_{\Sigma}(S)$ for a set $S$ is satisfied, for example, if $\mathcal{T}=\mathsf{D}(R)$ for a right noetherian ring $R$, since in that case $\mathcal{T}^{\mathsf{b}}_i=\mathsf{K}^\mathsf{b}(\Inj{R})=\mathsf{broad}_{\Sigma}(E)$, where $E$ is an injective cogenerator of $R$. The following example shows that this condition can be also found in the context of dg algebras.

\begin{exmp} \label{Prod vs Add}
Let $\T$ be an algebraic compactly generated triangulated category such that $\T^c$ is Krull-Schmidt and Hom-finite over a field $k$. Suppose that $x$ is a compact silting object in $\T$ which far-away generates $\T$. Then by Theorem \ref{silting far-away generating} we have that $\Tb_i=\mathsf{thick}(\mathsf{Prod}(x^\ast))$, where we can consider $x^*$ to be the Brown-Comenetz dual of $x$ with respect to the injective cogenerator $k$ of $\Mod{k}$, i.e.~$\mathsf{Hom}_\T(-,x^\ast)\cong \mathsf{Hom}_{k}(\mathsf{Hom}_\T(x,-),k)$. In particular, note that for any object $c$ of $\T$, we have that $\mathsf{Hom}_\T(c,x^\ast)\cong \mathsf{Hom}_{k}(\mathsf{Hom}_\T(x,c),k)$ is a finite-dimensional vector space over $k$. In particular, $x^\ast$ is an \textbf{endofinite} object of $\T$, i.e.~$\mathsf{Hom}_\T(c,x^\ast)$ has finite length over $\mathsf{End}_\T(x^\ast)$ for all $c$ compact. In particular, $x^\ast$ is then $\Sigma$-pure-injective and its indecomposable subobjects form a closed set of the Ziegler spectrum (see \cite[Theorem 2.13 and Appendix A]{conde_gorsky_marks_zvonareva}). It then follows from \cite[Theorem 4.10]{breaz_rafiliu} that $\mathsf{Prod}(x^\ast)=\mathsf{Add}(x^\ast)$ and that we then have $\Tb_i=\mathsf{thick}(\mathsf{Add}(x^\ast))$.

A concrete setup that realises the assumptions above is when we consider the derived category of a dg-algebra over a field $k$ for which $H^n(\Gamma)=0$ for all $n>0$ and such that the total cohomology $\oplus_{n\in\mathbb{Z}}H^n(\Gamma)$ is finite-dimensional over $k$, i.e~a connective proper dg algebra over $k$. This means that for such a dg algebra $\Gamma$, the implication $(**)$ holds for $\mathsf{D}(\Gamma)$. 
\end{exmp}

\begin{exmp}\label{bad example}
Let $\T=\mathsf{D}(R)$ where $R$ is the ring from Example \ref{counterexample}. Recall that $R$ is a commutative noetherian ring of infinite Krull dimension for which every finitely presented module has finite projective dimension. Being commutative and noetherian, if follows from \cite[Theorem 3.3]{rickard} that $\T$ is generated by injectives and that the implication $(\ast\ast)$ of the theorem above holds, hence showing that $\T$ has finite finitistic dimension. This is an unfortunate conclusion, given that we know that $R$ itself has infinite finitistic dimension. The reason for this mismatch is that the definition proposed by us and inspired by \cite[Theorem 4.4]{rickard} and \cite[Theorem B.1]{krause} is well adapted for Artin algebras (see also Theorem \ref{properties}), but not beyond that. 
\end{exmp}

\section{Regularity for $k$-linear categories}\label{Sec5}

\subsection{The subcategory of bounded finite objects} It is often the case that, as it has already happened, the triangulated categories that we work with are $k$-linear, where $k$ is commutative noetherian ring $k$. This is the case, for example, of the derived category of an algebra or of a scheme over $k$. Given a $k$-linear triangulated category $\mathcal{T}$ and a subcategory $\mathcal{X}$ of $\mathcal{T}$, we consider the following thick subcategories of $\mathcal{T}$\\ 
\[
\begin{aligned}
    \mathcal{X}^{\hfpadova}\coloneqq \{t\in\mathcal{T}: \ \forall x\in\mathcal{X}, \underset{n\in\mathbb{Z}}{\oplus}\mathsf{Hom}_{\mathcal{T}}(x,t[n])\in \smod k \}\ \ \ \ \  
    {^{\hfpadova}}\mathcal{X}\coloneqq \{t\in\mathcal{T}: \ \forall x\in\mathcal{X}, \underset{n\in\mathbb{Z}}{\oplus}\mathsf{Hom}_{\mathcal{T}}(t,x[n]) \in\smod k \}.
\end{aligned}
\]
It is clear by definition that $\mathcal{X}^{\hfpadova}\subseteq \mathcal{X}^{\padova}$ and $^{\hfpadova}\mathcal{X}\subseteq$$ ^{\padova}\mathcal{X}.$
These subcategories have been considered by many other authors under different names and in different setups, see for example \cite{adachi_mizuno_yang} and \cite{kuznetsov_shinder}.
Note that every triangulated category is $\mathbb{Z}$-linear by virtue of being additive. However, in this section we consider other $k$-linear structures as well, so as long as $k$ is commutative noetherian. 

Under the setup above, there is an extra distinguished triangulated category that we can study, which is intrinsic up to the choice of the ring $k$, which often will be clear from the context. 

\begin{defn} \label{bounded finite objects}
The subcategory of \textbf{bounded finite} objects of a $k$-linear triangulated category $\T$ is denoted by $\T^{\mathsf{b}}_c$ and it is defined by setting
$\mathcal{T}^{\mathsf{b}}_c\coloneqq (\mathcal{T}^{\mathsf{c}})^{\hfpadova}.$
\end{defn}

For later use we need the following. 

\begin{lem} \label{properties of hf orthogonals}
    Let $\mathcal{T}$ be a $k$-linear triangulated category and $\mathcal{X}$, $\mathcal{Y}$ classes of objects in $\mathcal{T}$. The following hold. 
    \begin{enumerate}
        \item $(\mathcal{X}*\mathcal{Y})^{\hfpadova}=\mathcal{X}^{\hfpadova}\cap \mathcal{Y}^{\hfpadova}$ and ${^{\hfpadova}(\mathcal{X}*\mathcal{Y})}={^{\hfpadova}}\mathcal{X}\cap {^{\hfpadova}}\mathcal{Y}$. 
        \item $\mathcal{X}^{\hfpadova}=\mathsf{thick}(\mathcal{X})^{\hfpadova}$ and ${^{\hfpadova}}\mathcal{X}={^{\hfpadova}}\mathsf{thick}(\mathcal{X})$. 
    \end{enumerate}
\end{lem}
\begin{proof}
The proof is the same as in Lemma \ref{triangulated_subcategories}, using the noetherianness of $k$ to prove (i).
\end{proof}

Observe also that the analogous of Lemma \ref{adjoints and far-away orthogonality} for $\?=\hfpadova$ also holds true and we shall use it when necessary. Let us now compute the subcategory of bounded finite objects for some triangulated categories.

\begin{prop} \label{bounded finite objects: examples}
Let $k$ be a commutative noetherian ring.
\begin{enumerate}
\item If $R$ is a Noether $k$-algebra and $\T=\mathsf{D}(R)$, then we have $\Tb_c=\mathsf{D}^{\mathsf{b}}(\smod{R})$.
\item If $k$ is artinian, $\Lambda$ is an Artin $k$-algebra and $\T=\mathsf{K}_{\mathsf{ac}}(\Inj{\Lambda})$, then we have $\Tb_c=\{0\}$.
\item If $k$ is artinian, $\Lambda$ is an Artin $k$-algebra and $\T=\mathsf{K}(\Inj{\Lambda})$, then we have $$\Tb_c=\{x\in \mathsf{K}^\mathsf{b}(\Inj{\Lambda})\colon \oplus_{n\in\mathbb{Z}} H^n(x)\in\smod{k}\}.$$
\item If $k$ is a field, $\Gamma$ is a dg $k$-algebra and $\T=\mathsf{D}(\Gamma)$, then we have 
$$\Tb_c=\{x\in \mathsf{D}(\Gamma)\colon \mathsf{dim}(\oplus_{n\in\mathbb{Z}} H^n(x))<\infty \}=:\mathsf{D}_{\mathsf{fd}}(\Gamma).$$
\end{enumerate}

\end{prop}
\begin{proof}
(i) For $\mathcal{T}=\mathsf{D}(R)$. By Lemma \ref{properties of hf orthogonals} we see that $\mathcal{T}^{\mathsf{b}}_c$ coincides with $R^{\hfpadova}$, which is precisely $\mathsf{D}^{\mathsf{b}}(\smod R)$ since $R$ is finitely generated over $k$.
   
(ii) For $\mathcal{T}=\mathsf{K}_{\mathsf{ac}}(\Inj{\Lambda})$, we have $\mathcal{T}^{\mathsf{b}}_c\subseteq \mathcal{T}^{\mathsf{b}}=\{0\}$ by Proposition \ref{examples2}.

(iii) For $\mathcal{T}=\mathsf{K}(\Inj \Lambda)$, consider the adjoint pair $(\mathsf{Q}_{\lambda},\mathsf{Q})$, see (\ref{rec}). The functor $\mathsf{Q}_{\lambda}$ preserves compact objects and, thus, the functor $\mathsf{Q}$ sends $\mathcal{T}^{\mathsf{b}}_c$ to $\mathsf{D}^{\mathsf{b}}(\smod \Lambda)$, which shows that $\mathcal{T}^{\mathsf{b}}_c$ is contained in 
$$\{x\in \mathsf{K}^\mathsf{b}(\Inj{\Lambda})\colon \oplus_{n\in\mathbb{Z}} H^n(x)\in\smod{k}\}.$$
The reverse inclusion is verified directly. 

 
 (iv) This follows from the fact that $\mathsf{D}(\Gamma)^{\mathsf{c}}=\mathsf{thick}(\Gamma)$ and $H^n(-)\cong \mathsf{Hom}_{\mathsf{D}(\Gamma)}(\Gamma,-[n])$.
\end{proof}

\begin{rem} \label{approximable}
Consider $\T$ as in Remark \ref{approx}, i.e.~$\mathcal{T}$ is compactly generated triangulated category generated by a single compact object $G$ and the t-structure generated by $G$ satisfies the ``approximability conditions'', see \cite[Definition 0.25]{neeman5}. Suppose that $\T$ is a $k$-linear category. In \cite[Definition 0.20]{neeman5}, Neeman considers the following subcategories
     \[
     \mathcal{T}^{-,\tau}_c\coloneqq \bigcap_{n>0}(\mathcal{T}^{\mathsf{c}}\ast \mathcal{T}^{\leq -n}) \text{  and  } \mathcal{T}^{\mathsf{b},\tau}_c\coloneqq \mathcal{T}^{-, \tau}_c\cap \mathcal{T}^{\mathsf{b}, \tau}. 
     \]
    where $\T^{\mathsf{b},\tau}$ is as defined in Remark \ref{approx}. As before, these subcategories turn out to be instrinsic and independent of the choice of generator $G$. Moreover, if $\mathsf{Hom}_{\mathcal{T}}(G,G[n])$ lies in $\smod k$ for all $n\in\mathbb{Z}$, by \cite[Theorem 0.4]{neeman5} we have that the subcategory $\mathcal{T}^{\mathsf{b},\tau}_c$ coincides with $\mathcal{T}^{\mathsf{b}}_c$ of Definition \ref{bounded finite objects}. 
     
\end{rem}

\subsection{Gorensteinness for dg algebras} As we will see in this subsection the subcategory of bounded finite objects helps us to establish a comparison between our notion of Gorenstein triangulated category for the derived category of a dg algebra $\Gamma$ with the one introduced in \cite{jin} for such dg algebras, whenever $\Gamma$ is proper and connective over a field $k$. Using Corollary \ref{T^b_p for dg algebras}, the derived category $\mathsf{D}(\Gamma)$ is Gorenstein if and only if 
\[
\mathsf{thick}(\Add \Gamma)=\mathsf{thick}(\Product \Gamma^*).
\]
The definition of Gorenstein dg algebra due to Jin \cite[Assumption 0.1(3)]{jin}, requires that $\mathsf{thick}(\Gamma)=\mathsf{thick}(\Gamma^*)$. The aim of this subsection is to compare these two notions; see Proposition \ref{gorensteinness vs gorensteinness}. We also refer to \cite{kuznetsov_shinder} for a related notion of Gorenstein dg categories and in particular to \cite[Section 6.2]{kuznetsov_shinder} where they compare their notion to that of \cite{jin}. We will first collect some necessary statements.

\begin{prop} \label{perfect modules over dg algebra}\label{thick of the dual} \label{homologically finite orthogonal of D_fd}
Let $\Gamma$ be a proper and connective dg algebra over a field $k$. Then the following hold.
\begin{enumerate}
\item There is an equality $\mathsf{per}(\Gamma)={^{\hfpadova}(\mathsf{D}_{\mathsf{fd}}(\Gamma))}$ as subcategories of $\mathsf{D}(\Gamma)$. In particular, we have that $(\mathsf{per}(\Gamma),\mathsf{D}_{\mathsf{fd}}(\Gamma))$ are a \textit{homologically finite far-away pair}, i.e. a pair $(\mathcal{X},\mathcal{Y})$ of subcategories of $\T$ for which $\Y=\X^{\hfpadova}$ and $\Y={}^{\hfpadova}\X$.
\item  There is an equality $\mathsf{thick}(\Gamma^*)=\mathsf{D}_{\mathsf{fd}}(\Gamma)^{\hfpadova}$ as subcategories of $\mathsf{D}(\Gamma)$.
\item  We have the following equalities between subcategories of $\mathsf{D}(\Gamma)$
    \[
    {^{\hfpadova}\mathsf{D}_{\mathsf{fd}}(\Gamma)}={^{\padova}\mathsf{D}_{\mathsf{fd}}(\Gamma)}\cap \mathsf{D}_{\mathsf{fd}}(\Gamma) \text{   and   }{\mathsf{D}_{\mathsf{fd}}(\Gamma)^{\hfpadova}}={\mathsf{D}_{\mathsf{fd}}(\Gamma)^{\padova}}\cap \mathsf{D}_{\mathsf{fd}}(\Gamma).
    \]
\end{enumerate}
\end{prop}

\begin{proof}
(i) This statement follows directly from \cite[Lemma 4.13]{adachi_mizuno_yang}.

(ii) We consider the duality over $k$, i.e.~the functor $d_k:=\mathbb{R}\mathsf{Hom}_{k}(-,k)\colon \mathsf{D}(\Gamma)\rightarrow\mathsf{D}(\Gamma^{\mathsf{op}})^{\mathsf{op}}$, which restricts to an equivalence $\mathsf{D}_{\mathsf{fd}}(\Gamma)\simeq \mathsf{D}_{\mathsf{fd}}(\Gamma^{\mathsf{op}})^{\mathsf{op}}$ (see for instance \cite[Lemma 3.5]{kuznetsov_shinder}). Since $\Gamma$ is proper, it follows that $\mathsf{D}_{\mathsf{fd}}(\Gamma)^{\hfpadova}\subseteq \mathsf{D}_{\mathsf{fd}}(\Gamma)$ and therefore the functor $d_k$ restricts to an equivalence $\mathsf{D}_{\mathsf{fd}}(\Gamma)^{\hfpadova}\simeq (\mathsf{D}_{\mathsf{fd}}(\Gamma^{\mathsf{op}})^{\mathsf{op}})^{\hfpadova}$.  The right-hand side of the latter equivalence is the opposite category of the subcategory ${^{\hfpadova}}(\mathsf{D}_{\mathsf{fd}}(\Gamma^{\mathsf{op}}))$ of $\mathsf{D}(\Gamma^{\mathsf{op}})$. Now, since $\Gamma$ is proper and connective, the same holds for $\Gamma^{\mathsf{op}}$ and, thus, using (i), we see that ${^{\hfpadova}}(\mathsf{D}_{\mathsf{fd}}(\Gamma^{\mathsf{op}}))=\mathsf{thick}(\Gamma^{\mathsf{op}})$. We conclude by observing that the inverse image of $\mathsf{thick}(\Gamma^{\mathsf{op}})^{\mathsf{op}}$ under the functor $d_k$ is $\mathsf{thick}(\Gamma^*)$, since $\Gamma^*=\mathbb{R}\mathsf{Hom}_k(_{\Gamma}\Gamma,k)$.

(iii) We only show the second equality, as both claims are similar. As observed in the proof of point (ii), we clearly have the inclusion $\mathsf{D}_{\mathsf{fd}}(\Gamma)^{\hfpadova}\subseteq \mathsf{D}_{\mathsf{fd}}(\Gamma)^{\padova}\cap \mathsf{D}_{\mathsf{fd}}(\Gamma)$. The converse follows from the fact that for any $x,y $ in $\mathsf{D}_{\mathsf{fd}}(\Gamma)$, the space $\mathsf{Hom}_{\mathsf{D}(\Gamma)}(x,y)$ is finite-dimensional over $k$. Indeed, this follows for instance from \cite[Proposition 2.1]{kalck_yang} which states that the t-structure induced by $\Gamma$ restricts to a bounded t-structure on $\mathsf{D}_{\mathsf{fd}}(\Gamma)$ with heart $\smod H^0(\Gamma)$ and that the thick closure of the heart is all of $\mathsf{D}_{\mathsf{fd}}(\Gamma)$.
\end{proof}

We can now prove the following.

\begin{prop} \label{gorensteinness vs gorensteinness}
Let $\Gamma$ be a proper connective dg algebra over a field $k$. Then $\mathsf{thick}(\Gamma)=\mathsf{thick}(\Gamma^\ast)$ if and only if $\mathsf{thick}(\mathsf{Add}(\Gamma))=\mathsf{thick}(\mathsf{Prod}(\Gamma))$. In other words, the dg algebra $\Gamma$ is Gorenstein \textnormal{(}in the sense of Jin\textnormal{)} if and only if the category $\mathsf{D}(\Gamma)$ is Gorenstein.
\end{prop}
\begin{proof}
    We write $\mathcal{T}$ for $\mathsf{D}(\Gamma)$ and assume first that $\Gamma$ is Gorenstein, i.e. that $\mathsf{per}(\Gamma)=\mathsf{thick}(\Gamma^*)$. Then we have 
    \[
    \mathcal{T}^{\mathsf{b}}_p=\mathsf{thick}(\Add \Gamma)=\mathsf{broad}_{\Sigma}(\mathsf{per}(\Gamma))=\mathsf{broad}_{\Sigma}(\mathsf{thick}({\Gamma}^*))=\mathsf{thick}(\Add {\Gamma}^*)=\mathsf{thick}(\Product {\Gamma}^*)=\mathcal{T}^{\mathsf{b}}_i, 
    \]
    where the first and last equalities follow from Corollary \ref{T^b_p for dg algebras}, the third follows by the assumption and the second to last follows from Example \ref{Prod vs Add}. This shows that $\mathcal{T}$ is Gorenstein. Assume now that $\mathcal{T}$ is Gorenstein. First off, we know from \cite[Lemma 5.2]{goodbody} that $\mathcal{T}$ is strongly generated, i.e.\ $\mathcal{T}=\overline{\langle x[\mathbb{Z}]\rangle}_n$ for some object $x\in\mathsf{D}_{\mathsf{fd}}(\Gamma)$ and some $n\geq 1$. In particular, it follows from Proposition \ref{bounded objects when T admits a strong gen} that $\mathcal{T}^{\mathsf{b}}=\mathsf{broad}_{\Sigma}(x)$ and by \cite[Proposition 9.8]{neeman5}, together with \cite[Theorem 0.4]{neeman5}, the fact that $\mathcal{T}$ is approximable (see for instance \cite[Remark 4.3]{neeman5}) and Remark \ref{approximable}, that $\mathcal{T}^{\mathsf{b}}_c=\langle x\rangle_n$ (see also \cite[Lemma 5.9]{Raedschelders_Stevenson}). As a consequence of the latter we have 
    \begin{equation} \label{broad for dg}
        \mathcal{T}^{\mathsf{b}}=\mathsf{broad}_{\Sigma}(\mathcal{T}^{\mathsf{b}}_c)
    \end{equation}
    as subcategories of $\mathcal{T}$. We then derive the following equalities
    \[
    \mathsf{per}(\Gamma)={^{\hfpadova}(\mathsf{D}_{\mathsf{fd}}(\Gamma))}={^{\padova}(\mathsf{D}_{\mathsf{fd}}(\Gamma))}\cap \mathsf{D}_{\mathsf{fd}}(\Gamma)={^{\padova}(\mathcal{T}^{\mathsf{b}})}\cap \mathsf{D}_{\mathsf{fd}}(\Gamma)
    \]
    where the first follows from Proposition \ref{perfect modules over dg algebra}(i), the second is Proposition \ref{homologically finite orthogonal of D_fd}(iii) and the third follows from (\ref{broad for dg}) together with Lemma \ref{triangulated_subcategories}(v). Dually, using Proposition \ref{thick of the dual}(ii), it follows that 
    \[
    \mathsf{thick}(\Gamma^*)=(\mathcal{T}^{\mathsf{b}})^{\padova}\cap \mathsf{D}_{\mathsf{fd}}(\Gamma). 
    \] 
    All in all, intersecting $\mathcal{T}^{\mathsf{b}}_p=\mathcal{T}^{\mathsf{b}}_i$ with $\mathsf{D}_{\mathsf{fd}}(\Gamma)$ implies that $\mathsf{thick}(\Gamma)=\mathsf{thick}(\Gamma^*)$. 
\end{proof}

\subsection{Regular triangulated categories} For $k$-linear triangulated categories, we suggest an additional homological condition, and a category that measures its failure.

\begin{defn} \label{regular triangulated categories}
    A $k$-linear triangulated category $\mathcal{T}$ is said to be \textbf{regular (over $k$)} if $\mathcal{T}^{\mathsf{b}}_c=\mathcal{T}^{\mathsf{c}}$. We define the \textbf{singularity $k$-category of $\T$}  to be $\T^{\mathsf{sg}}_k\coloneqq\mathsf{thick}(\mathcal{T}^{\mathsf{c}}\ast \mathcal{T}^{\mathsf{b}}_c)/(\mathcal{T}^{\mathsf{b}}_c\cap\T^{\mathsf{c}})$.
\end{defn}

Clearly the definition of the singularity category was set up to yield the following statement.

\begin{lem}
If $\T$ is a $k$-linear triangulated category, $\T$ is regular over $k$ if and only if $\T^{\mathsf{sg}}_k=0$.
\end{lem}
\begin{proof}
There is only a (slightly) non-trivial direction of this equivalence. The commutative diagram of inclusions
$$\xymatrix{&\T^{\mathsf{c}}\ar[rd]\\ (\mathcal{T}^{\mathsf{b}}_c\cap\T^{\mathsf{c}})\ar[ru]\ar[rd]&&\mathsf{thick}(\mathcal{T}^{\mathsf{c}}\ast \mathcal{T}^{\mathsf{b}}_c)\\ & \mathcal{T}^{\mathsf{b}}_c\ar[ru]}$$
yields an equivalence of categories if and only if every arrow in the diagram is an equivalence, proving that if $\T^{\mathsf{sg}}_k=0$ then $\T$ is regular.
\end{proof}

Based on our previous computations, we can state the following result (motivating the name \textit{regular}).

\begin{thm} \label{regularity for examples}
The following statements hold.
\begin{enumerate}
\item If $R$ is a commutative noetherian ring, then $\mathsf{D}(R)$ is regular \textnormal{(}over $R$\textnormal{)} if and only if $R$ is regular.
\item If $R$ is a commutative local noetherian ring, then $\mathsf{K}_{\mathsf{ac}}(\Inj R)$ is regular \textnormal{(}over $R$\textnormal{)} if and only if $\mathsf{K}_{\mathsf{ac}}(\Inj R)\simeq 0$ if and only if $R$ is regular.
\item If $\Lambda$ is an Artin $k$-algebra, then $\mathsf{D}(\Lambda)$ is regular \textnormal{(}over $k$\textnormal{)} if and only if $\mathsf{K}(\Inj{\Lambda})$ is regular \textnormal{(}over $k$\textnormal{)} if and only if $\mathsf{K}_{\mathsf{ac}}(\Inj{\Lambda})=0$ if and only if $\gld\Lambda<\infty$.
\end{enumerate}
\end{thm}
\begin{proof}
(i) Recall that a commutative noetherian ring $R$ is regular if and only if finitely generated modules have finite projective dimension (see for instance \cite[Theorem 5.94]{lam}).

(ii)  If $R$ is non-regular with maximal ideal $\mathfrak{m}$, then by \cite[Theorem 6.5]{avramov_veliche} we have 
   \begin{align}\label{hom2}
       \mathsf{Hom}_{\mathsf{D}_{\mathsf{sg}}(R)}(k(\mathfrak{m}),k(\mathfrak{m})[n])\neq 0
   \end{align}
    for all $n$ in $\mathbb{Z}$, where $k(\mathfrak{m})$ denotes the residue field of $R$. We consider $\T=\mathsf{K}_{\mathsf{ac}}(\Inj{R})$ and by viewing $\mathsf{D}_{\mathsf{sg}}(R)$ as a full subcategory of $\mathcal{T}^{\mathsf{c}}$ we infer that $k(\mathfrak{m})$ lies in $\mathcal{T}^{\mathsf{c}}$ but not in $\mathcal{T}^{\mathsf{b}}_c$ from the equation above. This means that $\mathsf{K}_{\mathsf{ac}}(\Inj R)$ is not regular. If now $R$ is regular, then by the Auslander–Buchsbaum theorem it follows that $\gd R<\infty$ and so $\mathsf{K}_{\mathsf{ac}}(\Inj R)\simeq 0$. 

(iii) This follows from Proposition \ref{bounded finite objects: examples}.
\end{proof}


The following remarks discuss the relation between our notion of regularity and other analogous notions available in the literature.

\begin{rem} \label{bounded finite objects for a dg algebra}
    Let $\Gamma$ be a dg algebra over a field $k$. Then for the $k$-linear category $\mathcal{T}=\mathsf{D}(\Gamma)$, we have
    \[
    \mathcal{T}^{\mathsf{b}}_c=\{x\in\mathsf{D}(\Gamma):  \oplus_{n\in\mathbb{Z}}H^n(x)\in\smod k\},
    \]
    see Proposition \ref{bounded finite objects: examples}. Observe then that $\Gamma$ is proper, i.e. $\oplus_{n\in\mathbb{Z}} H^n(\Gamma)$ is finitely generated over $k$, if and only if $\mathcal{T}^{\mathsf{c}}\subseteq \mathcal{T}^{\mathsf{b}}_c$. Further, when $\Gamma$ is smooth, i.e. $\Gamma\in\mathsf{per}(\Gamma\otimes_k{\Gamma}^{\mathsf{op}})$, then $\mathcal{T}^{\mathsf{b}}_c\subseteq \mathcal{T}^{\mathsf{c}}$; see the proof of \cite[Lemma 4.1]{keller}. In particular, if $\Gamma$ is proper and smooth, then $\T^{\mathsf{c}}=\T^\mathsf{b}_\mathsf{c}$, i.e. $\mathsf{D}(\Gamma)$ is regular over $k$. For a related result in the context of small dg-enhanced triangulated categories, we refer to \cite[Lemma 3.8]{kuznetsov_shinder}. 
\end{rem}


\begin{rem}
    It is perhaps more classical to regard a (small) triangulated category as being ``regular'' if it admits a strong generator \cite{bondal_van-den-bergh}. In our setting it would be reasonable to ask this condition on the compacts and this in principle differs from our definitions. In order to explain this, let us recall some general implications below, for $R$ a noetherian ring. 
\[
\begin{tikzcd}
	{\mathsf{D}^{\mathsf{b}}(R)=\mathsf{K}^{\mathsf{b}}(\Proj R)} && {\gd R<\infty} && {\mathsf{K}^{\mathsf{b}}(\proj R) \text{ admits a strong generator}} \\
	\\
	{\mathsf{D}^{\mathsf{b}}(\smod R)=\mathsf{K}^{\mathsf{b}}(\proj R)} && {R \text{ regular}}
	\arrow[Rightarrow, 2tail reversed, from=1-1, to=1-3]
	\arrow[Rightarrow, from=1-1, to=3-1]
	\arrow[Rightarrow, 2tail reversed, from=1-5, to=1-3]
	\arrow["{{/}}"{marking, allow upside down}, shift right=4, Rightarrow, from=3-1, to=1-1]
	\arrow["{{(*)}}", shift left=4, Rightarrow, from=3-1, to=1-1]
	\arrow["{{(R\text{ comm.})}}"', Rightarrow, 2tail reversed, from=3-3, to=3-1]
\end{tikzcd}
\]
    The implication marked with $(*)$ works for Artin algebras or local commutative noetherian rings, while an example for the implication that fails was given by Nagata (we have used this in Remark \ref{counterexample}). Regarding non-noetherian rings, there are rings satisfying $\mathsf{D}^{\mathsf{b}}(\smod R)\simeq\mathsf{K}^{\mathsf{b}}(\proj R)$, of infinite global dimension such that $\mathsf{D}^{\mathsf{b}}(\smod R)$ has Rouquier dimension 0, see \cite{stevenson}. The recent notion of \emph{local regularity} \cite{local_regularity} also differs from our definition of regularity; by \cite[Theorem 8.3]{local_regularity}, for a field $k$ and a finite group $G$ such that $\mathsf{char}(k)$ divides the order of $G$, the triangulated category $\mathsf{K}(\Inj kG)$ is locally regular. However, by Theorem \ref{regularity for examples}, $\mathsf{K}(\Inj kG)$ is regular in our sense if and only if $\gd kG<\infty$, which is never the case under the given assumptions. 
\end{rem}


\begin{rem}
    Let $\mathcal{D}$ be a small triangulated category, $G$ an object of $\mathcal{D}$ and consider $\mathcal{C}_G(\mathcal{D})$ to be the \emph{completion of $\mathcal{D}$} with respect to a ``$G$-good metric'' on $\mathcal{D}$ (i.e. a good metric equivalent to the one induced by $G$), see \cite{biswas_chen_manali-rahul_parker_zheng, neeman_metrics} for details. According to \cite{biswas_chen_manali-rahul_parker_zheng}, $\mathcal{D}$ is \emph{regular at $G$} if $\mathcal{D}=\mathcal{C}_{G}(\mathcal{D})$ and \emph{regular} if it is regular at a classical generator $G$. Assume now that $\mathcal{T}$ is a compactly generated triangulated category that has a compact generator $G$ such that $\mathsf{Hom}_{\mathcal{T}}(G,G[i])=0$ for $i\gg0$. Set $\mathcal{D}=\mathcal{T}^{\mathsf{c}}$ and consider $\mathcal{T}^{\mathsf{b},\tau}_{c}$ as in Remark \ref{approximable}, where $\tau$ is the t-structure of $\mathcal{T}$ induced by $G$. Then $\mathcal{D}$ is regular (at $G$) if and only if $\mathcal{D}=\mathcal{T}^{\mathsf{b},\tau}_c$. When $\mathcal{T}$ is approximable and $\mathsf{Hom}_{\mathcal{T}}(G,G[n])$ lies in $\smod k$ for all $n\in\mathbb{Z}$, we know (see Remark \ref{approximable}) that $\mathcal{T}^{\mathsf{b},\tau}_c=\mathcal{T}^{\mathsf{b}}_c$. To sum up, for such a triangulated category $\mathcal{T}$, we have that $\mathcal{T}^{\mathsf{c}}$ is regular in the sense of \cite{biswas_chen_manali-rahul_parker_zheng} if and only if $\mathcal{T}$ is regular in the sense of Definition \ref{regular triangulated categories}.
\end{rem}

The following proposition tells us that our notion of singularity category coincides with other notions available in the literature. We note that although the remark above tells us that the vanishing of the singularity category following \cite{biswas_chen_manali-rahul_parker_zheng} coincides with the vanishing of our proposed singularity category, it is not clear to us how they relate in the non-vanishing situation.

\begin{prop}\label{sing ex}
We have the following computations of singularity categories.
\begin{enumerate}
\item For a Noether $k$-algebra $R$ and $\T=\mathsf{D}(R)$, we have 
$\T_k^{\mathsf{sg}}=\mathsf{D}^{\mathsf{b}}(\smod{R})/\mathsf{K}^{\mathsf{b}}(\proj{R}).$
\item For a finite-dimensional algebra $\Lambda$ over a field $k$,
\begin{enumerate}
\item[\textnormal{(a)}] if $\T=\mathsf{D}(\Lambda)$, we have $\T_k^{\mathsf{sg}}=\mathsf{D}^{\mathsf{b}}(\smod{\Lambda})/\mathsf{K}^{\mathsf{b}}(\proj{\Lambda})$;
\item[\textnormal{(b)}] if $\T=\mathsf{K}(\Inj{\Lambda})$, we have $\T_k^{\mathsf{sg}}=\mathsf{D}^{\mathsf{b}}(\smod{\Lambda})/\mathsf{K}^{\mathsf{b}}(\inj{\Lambda})$.
\end{enumerate}
\item  For a dg algebra $\Gamma$ over a field $k$ and $\mathcal{T}=\mathsf{D}(\Gamma)$ we have
\begin{enumerate}
    \item[\textnormal{(a)}] if $\Gamma$ is proper \textnormal{(}see Remark \ref{bounded finite objects for a dg algebra}\ \!\textnormal{)} then $\mathcal{T}^{\mathsf{sg}}_k=\mathsf{D}_{\mathsf{fd}}(\Gamma)/\mathsf{per}(\Gamma)$;
    \item[\textnormal{(b)}] if $\Gamma$ is smooth \textnormal{(}see Remark \ref{bounded finite objects for a dg algebra}\ \!\textnormal{)} then $\mathcal{T}^{\mathsf{sg}}_k=\mathsf{per}(\Gamma)/\mathsf{D}_{\mathsf{fd}}(\Gamma)$.
\end{enumerate}
\end{enumerate}
\end{prop}
\begin{proof}
Statements (i) and (ii) follow from the computations in Proposition \ref{bounded finite objects: examples}, provided that we prove that for $\T=\mathsf{K}(\Inj\Lambda)$ we have $\T^{\mathsf{b}}_c=\mathsf{K}^{\mathsf{b}}(\mathsf{inj}\mbox{-}\Lambda)$. From the Proposition above cited, we see that $\T^{\mathsf{b}}_c$ is the subcategory of $\mathsf{K}^{\mathsf{b}}(\Inj\Lambda)\subseteq \mathsf{Im}(\mathsf{Q}_\rho)$ of objects with finitely generated total cohomology. Since the compact objects in $\T$ are those in $\mathsf{Q}_\rho(\mathsf{D}^{\mathsf{b}}(\smod\Lambda))$, we get
$$\T^{\mathsf{b}}_c=\mathsf{Q}_\rho(\mathsf{D}^\mathsf{b}(\smod\Lambda)^{\hfpadova})\cap \mathsf{K}^{\mathsf{b}}(\Inj\Lambda)=\mathsf{Q}_\rho(\mathsf{thick}(\Lambda^\ast))\cap \mathsf{K}^{\mathsf{b}}(\Inj\Lambda)=\mathsf{Q}_\rho(\mathsf{thick}(\Lambda^\ast))=\mathsf{K}^{\mathsf{b}}(\mathsf{inj}\mbox{-}\Lambda),$$
where the first equality follows from the fully faithfulness of $\mathsf{Q}_\rho$, the second equality holds by Proposition \ref{perfect modules over dg algebra}(ii) and the third and fourth equalities follows from the fact that $\Lambda^\ast$ is injective.

For statement (iii) observe that if $\Gamma$ is proper then $\mathsf{per}(\Gamma)\subseteq \mathsf{D}_{\mathsf{fd}}(\Gamma)$, and if $\Gamma$ is smooth, then $\mathsf{D}_{\mathsf{fd}}(\Gamma)\subseteq \mathsf{per}(\Gamma)$ (see Remark \ref{bounded finite objects for a dg algebra}). Hence the singularity category is as stated, in both cases (see also \cite{amiot}). 
\end{proof}

We will now compare regularity and the finiteness of global dimension for a $k$-linear triangulated category. This is an extension of Theorem \ref{theorem1}. 

\begin{thm} \label{regularity vs finite global dimension}
    Let $\T$ be a compactly generated $k$-linear triangulated category. The following statements hold.
    \begin{enumerate}
        \item  If $\mathcal{T}$ satisfies $(\mathcal{T}^{\mathsf{b}})^{\padova}=(\mathcal{T}^{\mathsf{b}}_c)^{\padova}$ and $\T$ is regular, then $\gd\mathcal{T}<\infty$. 
        \item If $\mathcal{T}^{\mathsf{c}}$ is Hom-finite over $k$,  $\mathcal{T}^{\mathsf{b}}_p\cap \mathcal{T}^{\mathsf{b}}_c\subseteq \mathcal{T}^{\mathsf{c}}$ and $\gd\mathcal{T}<\infty$, then $\T$ is regular.
    \end{enumerate}
\end{thm}
\begin{proof}
Assume that $\mathcal{T}$ is compactly generated, regular and satisfies $(\mathcal{T}^{\mathsf{b}})^{\padova}=(\mathcal{T}^{\mathsf{b}}_c)^{\padova}$. Applying $(-)^{\padova}$ to the equality $\mathcal{T}^{\mathsf{c}}=\mathcal{T}^{\mathsf{b}}_c$ shows that $\mathcal{T}^{\mathsf{b}}=(\mathcal{T}^{\mathsf{b}}_c)^{\padova}=(\mathcal{T}^{\mathsf{b}})^{\padova}=\mathcal{T}^{\mathsf{b}}_i$, meaning that $\gd\mathcal{T}<\infty$. Assume now that $\mathcal{T}^{\mathsf{c}}$ is Hom-finite over $k$, of finite global dimension and satisfies $\mathcal{T}^{\mathsf{b}}_p\cap \mathcal{T}^{\mathsf{b}}_c\subseteq \mathcal{T}^{\mathsf{c}}$. Then from the equality $\mathcal{T}^{\mathsf{b}}=\mathcal{T}^{\mathsf{b}}_p$, follows that $\mathcal{T}^{\mathsf{b}}\cap \mathcal{T}^{\mathsf{b}}_c=\mathcal{T}^{\mathsf{b}}_p\cap \mathcal{T}^{\mathsf{b}}_c$. We have $\mathcal{T}^{\mathsf{b}}_c\subseteq \mathcal{T}^{\mathsf{b}}$ and using the assumption, we infer that $\mathcal{T}^{\mathsf{b}}_c\subseteq \mathcal{T}^{\mathsf{c}}$. On the other hand, we always have $\mathcal{T}^{\mathsf{c}}\subseteq \mathcal{T}^{\mathsf{b}}_p$ and therefore $\mathcal{T}^{\mathsf{c}}\subseteq \mathcal{T}^{\mathsf{b}}$. Since $\mathcal{T}$ is locally Hom-finite, it follows that $\mathcal{T}^{\mathsf{c}}\subseteq \mathcal{T}^{\mathsf{b}}_c$, which completes the proof.    
\end{proof}

Note that we have checked the assumptions of the theorem in some settings. For example, the condition $(\mathcal{T}^{\mathsf{b}})^{\padova}=(\mathcal{T}^{\mathsf{b}}_c)^{\padova}$ is satisfied for $\T=\mathsf{D}(\Lambda)$, where $\Lambda$ is an Artin algebra, see Remark \ref{implication}. Also the condition $\mathcal{T}^{\mathsf{b}}_p\cap \mathcal{T}^{\mathsf{b}}_c\subseteq \mathcal{T}^{\mathsf{c}}$  is verified for the $\T=\mathsf{D}(\Lambda)$ and $\T=\mathsf{K}(\Inj \Lambda)$ for $\Lambda$ any finite dimensional $k$-algebra, just by observing the computation of the distinguished subcategories for each of these categories. 

We can now prove that for a proper connective dg algebra $\Gamma$, under some mild assumption on the base field, smoothness is detected intrinsically in $\mathsf{D}(\Gamma)$.

\begin{cor}
    Let $\Gamma$ be a proper connective dg algebra over a perfect field. Then $\Gamma$ is smooth if and only if $\mathsf{D}(\Gamma)$ has finite global dimension. 
\end{cor}
\begin{proof}
Using \cite[Corollary 3.12]{Raedschelders_Stevenson}, we may assume that $\Gamma$ is a finite dimensional dg algebra (i.e. given by a bounded complex of finite dimensional vector spaces on the nose). For such a dg algebra, since we assume $k$ to be a perfect field, it follows from \cite[Corollary 2.5]{goodbody2} that $\Gamma$ is smooth if and only if $\mathsf{D}_{\mathsf{fd}}(\Gamma)=\mathsf{per}(\Gamma)$. Moreover, we observe from the proof of Proposition \ref{gorensteinness vs gorensteinness} that for $\T=\mathsf{D}(\Gamma)$, we have both  $(\mathcal{T}^{\mathsf{b}})^{\padova}=(\mathcal{T}^{\mathsf{b}}_c)^{\padova}$ and $\mathcal{T}^{\mathsf{b}}_p\cap \mathcal{T}^{\mathsf{b}}_c\subseteq \mathcal{T}^{\mathsf{c}}$. It follows then, by Theorem \ref{regularity vs finite global dimension}, that $\mathsf{D}(\Gamma)$ is regular if and only if $\mathsf{D}(\Gamma)$ has finite global dimension.
\end{proof}

\section{Application to Recollements}\label{app to rec}
\label{recollements}

In this section we investigate the existence of recollements between (complete and cocomplete) triangulated categories and their various distinguished subcategories of Definition \ref{main_definition}. This problem was studied in \cite{angeleri_koenig_liu,angeleri_koenig_liu_yang} for derived categories of rings.

\subsection{A toolbox} In preparation of our results we need some useful results, summarised in the following lemma.

\begin{lem} 
\label{toolbox}
Let $\mathsf{F}\colon\mathcal{D}\rightarrow \mathcal{T}$ and $\mathsf{G}\colon\mathcal{T}\rightarrow \mathcal{D}$ be two triangle functors between triangulated categories with coproducts. Then the following statements hold.
\begin{enumerate}
\item \textnormal{(\!\!\cite[Lemma 4.2]{angeleri_koenig_liu_yang})} If $\mathsf{F}\colon \mathcal{D}\rightarrow \mathcal{T}$ is fully faithful and preserves coproducts, then for any $x$ such that $\mathsf{F}(x)$ lies in $\mathcal{T}^{\mathsf{c}}$, we have that $x$ lies in $\mathcal{D}^{\mathsf{c}}$. 
\item If $(\mathsf{F},\mathsf{G})$ is an adjoint pair and $\mathsf{G}$ preserves coproducts, then
\begin{enumerate}
\item[\textnormal{(a)}]  $\mathsf{F}$ preserves compact objects;
\item[\textnormal{(b)}] if $G$ is furthermore fully faithful, both $\mathcal{T}$ and $\mathcal{D}$ are compactly generated and $\mathsf{GF}$ preserves compacts, then $\mathsf{G}$ restricts to a functor $\mathcal{T}^{\mathsf{c}}\rightarrow\mathcal{D}^{\mathsf{c}}$.  
\end{enumerate}
\end{enumerate}
\end{lem}
\begin{proof}
We prove assertion (ii)(b). By (ii)(a), the functor $\mathsf{F}$ restricts to $\mathcal{D}^{\mathsf{c}}\rightarrow \mathcal{T}^{\mathsf{c}}$. We claim, in fact, that $\mathsf{F}$ maps a set $D$ of compact generators of $\mathcal{D}$, to a set of compact generators of $\mathcal{T}$. Indeed, let $x$ be an object in $\mathcal{T}$ such that $\mathsf{Hom}_{\mathcal{T}}(\mathsf{F}(d),x[n])=0$ for all $n$ and for every $d$ in $D$. By the adjunction we have 
    \[
    \mathsf{Hom}_{\mathcal{T}}(\mathsf{F}(d),x[n])\cong\mathsf{Hom}_{\mathcal{D}}(d,\mathsf{G}(x)[n]),
    \]
    and since $D$ is a set of compact generators, it follows that $\mathsf{G}(x)=0$. Further, since $\mathsf{G}$ is fully faithful, it follows that $x=0$. Consider now the subcategory $\mathcal{U}=\{x\in\mathcal{T}: \mathsf{G}(x)\in\mathcal{D}^{\mathsf{c}}\}$ of $\mathcal{T}$ and note that it is thick and contains $\mathsf{F}(D)$. We have $\mathsf{F}(D)\subseteq \mathcal{U}$ and so $\mathsf{thick}(\mathsf{F}(D))\subseteq \mathcal{U}$. Since $\mathsf{F}(D)$ is a set of compact generators of $\mathcal{T}$, it follows from \cite[Lemma 2.2]{neeman2} that $\mathsf{thick}(\mathsf{F}(D))=\mathcal{T}^{\mathsf{c}}$, from which the result follows. 
\end{proof}
For the next two results a we fix a recollement of triangulated categories $(\mathcal{X},\mathcal{T},\mathcal{Y})$.

\begin{equation} \label{R}
    \begin{tikzcd}
\mathcal{X} \arrow[rr, "\mathsf{i}"] &  & \mathcal{T} \arrow[rr, "\mathsf{e}"] \arrow[ll, "\mathsf{q}"', bend right] \arrow[ll, "\mathsf{p}", bend left] &  & \mathcal{Y} \arrow[ll, "\mathsf{l}"', bend right] \arrow[ll, "\mathsf{r}", bend left]
\end{tikzcd}
\end{equation}
We recall that a \textbf{recollement} (of triangulated categories) is a diagram of triangle functors as above such that $(\mathsf{q},\mathsf{i},\mathsf{p})$ and $(\mathsf{l},\mathsf{e},\mathsf{r})$ are adjoint triples, $\mathsf{i}$, $\mathsf{l}$ and $\mathsf{r}$ are fully faithful and $\mathsf{Ker}(\mathsf{e})=\mathsf{Im}(\mathsf{i})$. It is an easy consequence of the definition that $(\mathsf{Ker}(\mathsf{q}),\mathsf{Im}(\mathsf{i}))$ and $(\mathsf{Im}(\mathsf{i}),\mathsf{Ker}(\mathsf{r}))$ are (stable) t-structures with truncation functors given by the units and counits of the adjunctions of the given functors. We refer the reader to the original source \cite{beilinson_bernstein_deligne} for the basic material on recollements. We begin by discussing the restriction of some functors in a recollement to the subcategories of compact objects.

\begin{lem} \label{compacts_in_recollements}
    Let $\T$ be a complete and cocomplete triangulated category and assume the existence of a recollement \textnormal{(\ref{R})}. The following statements hold.
    \begin{enumerate}
    \item If  $\mathsf{i}$ restricts to a functor $\mathcal{X}^{\mathsf{c}}\rightarrow \mathcal{T}^{\mathsf{c}}$, then $\mathsf{e}$ restricts to a functor $\mathcal{T}^{\mathsf{c}}\rightarrow \mathcal{Y}^{\mathsf{c}}$. The converse holds if $\T$ is compactly generated.
    \item If both $\mathsf{i}$ and $\mathsf{e}$ restrict to functors between subcategories of compact objects, then $\mathsf{r}$ restricts to a functor between subcategories of compact objects if and only if so does $\mathsf{p}$. 
%
    \end{enumerate}
\end{lem}
\begin{proof}
    (a) Let $t$ be a compact object of $\mathcal{T}$ and consider the following triangle 
    \[
    \mathsf{le}(t)\rightarrow t\rightarrow \mathsf{iq}(t)\rightarrow \mathsf{le}(t)[1].
    \]
    By Lemma \ref{toolbox}(ii)(a) the functors $\mathsf{q}$ and $\mathsf{l}$ preserve compacts, since they are left adjoints to the coproduct preserving functors $\mathsf{i}$ and $\mathsf{e}$ respectively. If $\mathsf{i}$ restricts to $\mathcal{X}^{\mathsf{c}}\rightarrow \mathcal{T}^{\mathsf{c}}$, then $\mathsf{iq}(t)$ lies in $\mathcal{T}^{\mathsf{c}}$ and, thus, so does $\mathsf{le}(t)$ by the above triangle. It follows by Lemma \ref{toolbox}(i) that $\mathsf{e}(t)\in\mathcal{Y}^{\mathsf{c}}$, since $\mathsf{l}$ is fully faithful and coproduct-preserving. Assume now that $\mathcal{X}$ and $\mathcal{T}$ are compactly generated and that $\mathsf{e}$ restricts to $\mathcal{T}^{\mathsf{c}}\rightarrow \mathcal{Y}^{\mathsf{c}}$. Then $\mathsf{le}(t)$ lies in $\mathcal{T}^{\mathsf{c}}$ and, thus, so does $\mathsf{iq}(t)$ by the above triangle. It follows by Lemma \ref{toolbox}(ii)(b) that $\mathsf{i}$ restricts to compact objects.
    
    (b) Assume first that $\mathsf{p}(\mathcal{T}^{\mathsf{c}})\subseteq\mathcal{X}^{\mathsf{c}}$. Given $y$ in $\mathcal{Y}^{\mathsf{c}}$, consider the following triangle in $\mathcal{T}$
    \[
    \mathsf{ipl}(y)\rightarrow \mathsf{l}(y)\rightarrow \mathsf{rel}(y)\rightarrow \mathsf{ipl}(y)[1].
    \]
    We have that $\mathsf{l}(y)$ lies in $\mathcal{T}^{\mathsf{c}}$ and, by assumption, so does $\mathsf{ipl}(y)$. From the above triangle, it follows that $\mathsf{r}(y)\cong \mathsf{rel}(y)$ lies in $\mathcal{T}^{\mathsf{c}}$, showing one of the two claimed implications. Assume now that $\mathsf{r}(\mathcal{Y}^{\mathsf{c}})\subseteq \mathcal{T}^{\mathsf{c}}$. Given $t$ in $\mathcal{T}^{\mathsf{c}}$, we consider the following triangle of $\mathcal{T}$
    $$\mathsf{ip}(t)\rightarrow t\rightarrow \mathsf{re}(t)\rightarrow \mathsf{ip}(t)[1].$$
    Applying $\mathsf{q}$ to the latter gives a triangle and observing that $\mathsf{q}(t)$ lies in $\mathcal{X}^{\mathsf{c}}$, by the assumption we also have that $\mathsf{qre}(t)$ lies in $\mathcal{X}^{\mathsf{c}}$. We conclude that $\mathsf{p}(t)\cong \mathsf{qip}(t)$ lies in $\mathcal{X}^{\mathsf{c}}$, as wanted.
\end{proof}

The following lemma will also be useful in the following subsections.

\begin{lem} \label{inclusions_in_recollements}
Consider a recollement of the form \textnormal{(\ref{R})} and suppose that $\mathsf{l}(\mathcal{Y}^{\mathsf{b}})\subseteq \mathcal{T}^{\mathsf{b}}$ and $\mathsf{i}(\mathcal{X}^{\mathsf{c}})\subseteq \mathcal{T}^{\mathsf{b}}_p$. The following hold\textnormal{:} 
    \begin{enumerate}
        \item $\mathsf{e}(\mathcal{T}^{\mathsf{c}})\subseteq \mathcal{Y}^{\mathsf{b}}_p$.
        \item $\mathsf{q}(\mathcal{T}^{\mathsf{b}})\subseteq \mathcal{X}^{\mathsf{b}}$.
    \end{enumerate}
    
\end{lem}
\begin{proof}
   (i) For a compact object $t$ of $\mathcal{T}$, we have that $\mathsf{q}(t)$ lies in $\mathcal{X}^{\mathsf{c}}$ and, consequently, $\mathsf{iq}(t)$ lies in $\mathcal{T}^{\mathsf{b}}_p$ by Lemma \ref{adjoints and far-away orthogonality}(ii). Since $\mathcal{T}^{\mathsf{c}}$ is contained in $\mathcal{T}^{\mathsf{b}}_p$ by Lemma \ref{common facts of the subcategories}, we infer from the triangle
   $$\mathsf{le}(t)\rightarrow t\rightarrow \mathsf{iq}(t)\rightarrow \mathsf{le}(t)[1]$$
 that $\mathsf{le}(t)$ lies in $\mathcal{T}^{\mathsf{b}}_p$.  Finally, since $\mathsf{l}(\mathcal{Y}^{\mathsf{b}})\subseteq\mathcal{T}^{\mathsf{b}}$, we conclude by Lemma \ref{fully_faithful_functor}(iv)  that $\mathsf{e}(t)\in\mathcal{Y}^{\mathsf{b}}_p$.

    (ii) Since the functors $\mathsf{q}$ and $\mathsf{l}$ restrict to compact objects, it follows from Lemma \ref{adjoints and far-away orthogonality}(i) that the functors $\mathsf{i}$ and $\mathsf{e}$ restrict to bounded objects. Combined with the assumption that $\mathsf{l}$ sends bounded objects to bounded objects we have that for an object $t$ in $\mathcal{T}^{\mathsf{b}}$, $\mathsf{le}(t)$ lies in $\mathcal{T}^{\mathsf{b}}$. Considering the triangle 
    \[
    \mathsf{le}(t)\rightarrow t\rightarrow \mathsf{iq}(t)\rightarrow \mathsf{le}(t)[1],
    \]
we see that $\mathsf{iq}(t)$ also lies in $\T^{\mathsf{b}}$. Since $\mathsf{i}(\mathcal{X}^{\mathsf{c}})\subseteq \mathcal{T}^{\mathsf{b}}_p$ by assumption, and knowing from Proposition \ref{common facts of the subcategories}(i) that $\T^{\mathsf{b}}=(\T^{\mathsf{b}}_p)^{\padova}$, Lemma  \ref{fully_faithful_functor}(iii) applied to this inclusion shows that $\mathsf{q}(t)$ lies in $(\X^\mathsf{c})^{\padova}=\X^{\mathsf{b}}$.
\end{proof}

\subsection{Restricting recollements} \label{restricting recollements} Consider the following types of recollements, where $\mathcal{T}$, $\mathcal{X}$ and $\mathcal{Y}$ are assumed to be compactly generated triangulated categories. 
\[
\begin{tikzcd}
\mathcal{X} \arrow[rr, "\mathsf{i}"] &  & \mathcal{T} \arrow[rr, "\mathsf{e}"] \arrow[ll, "\mathsf{q}"', bend right] \arrow[ll, "\mathsf{p}", bend left] &  & \mathcal{Y} \arrow[ll, "\mathsf{l}"', bend right] \arrow[ll, "\mathsf{r}", bend left]
\end{tikzcd} \textnormal{($\mathrm{U}$)} \ \ \ 
\begin{tikzcd}
\mathcal{X}^{\mathsf{c}} \arrow[rr, "\mathsf{i}"] &  & \mathcal{T}^{\mathsf{c}} \arrow[rr, "\mathsf{e}"] \arrow[ll, "\mathsf{q}"', bend right] \arrow[ll, "\mathsf{p}", bend left] &  & \mathcal{Y}^{\mathsf{c}} \arrow[ll, "\mathsf{l}"', bend right] \arrow[ll, "\mathsf{r}", bend left]
\end{tikzcd} \textnormal{($\mathrm{C}$)} 
\]
\[
\begin{tikzcd}
\mathcal{X}^{\mathsf{b}} \arrow[rr, "\mathsf{i}"] &  & \mathcal{T}^{\mathsf{b}} \arrow[rr, "\mathsf{e}"] \arrow[ll, "\mathsf{q}"', bend right] \arrow[ll, "\mathsf{p}", bend left] &  & \mathcal{Y}^{\mathsf{b}} \arrow[ll, "\mathsf{l}"', bend right] \arrow[ll, "\mathsf{r}", bend left]
\end{tikzcd} \textnormal{($\mathrm{B}$)}
\]
We will study necessary and sufficient conditions for a recollement of the form (U) to restrict to one of the form (C) or of the form (B). The following is a triangulated version of \cite[Theorem 4.4]{angeleri_koenig_liu_yang}.

\begin{prop}
Let $\T$, $\X$ and $\Y$ be compactly generated triangulated categories. The following are equivalent\textnormal{:} 
    \begin{enumerate}
        \item a recollement of type \textnormal{($\mathrm{U}$)} restricts to a recollement of type \textnormal{($\mathrm{C}$)};
        \item the functors $\mathsf{i}$ and $\mathsf{p}$ restrict to compact objects;
        \item the functors $\mathsf{e}$ and $\mathsf{r}$ restrict to compact objects. 
    \end{enumerate}
\end{prop}
\begin{proof}
    (i)$\Longrightarrow$(ii) is evident, while the implications (ii)$\Longrightarrow$(iii) and (iii)$\Longrightarrow$(i) follow from Lemma \ref{compacts_in_recollements} and the fact that $\mathsf{q}$ and $\mathsf{l}$ always restrict to compact objects. 
\end{proof}

The following is a triangulated version of \cite[Proposition 4.8]{angeleri_koenig_liu_yang}.

\begin{prop} \label{restriction_to_type_B}
 Let $\T$, $\X$ and $\Y$ be compactly generated triangulated categories. The following are equivalent\textnormal{:}  
    \begin{enumerate}
        \item A recollement of type \textnormal{($\mathrm{U}$)} restricts to a recollement of type \textnormal{($\mathrm{B}$)}. 
        \item $\mathsf{i}(\mathcal{X}^{\mathsf{c}})\subseteq \mathcal{T}^{\mathsf{b}}_p$ and $\mathsf{l}(\mathcal{Y}^{\mathsf{b}})\subseteq \mathcal{T}^{\mathsf{b}}$. 
    \end{enumerate}
\end{prop}
\begin{proof}
Let us first show (i)$\Longrightarrow$(ii). We have $\mathsf{l}(\mathcal{Y}^{\mathsf{b}})\subseteq \mathcal{T}^{\mathsf{b}}$ by assumption, so we are left to show that $\mathsf{i}(\mathcal{X}^{\mathsf{c}})\subseteq \mathcal{T}^{\mathsf{b}}_p$. Since the functor $\mathsf{p}\colon\mathcal{T}\rightarrow \mathcal{X}$ restricts to bounded objects, it follows by Lemma \ref{adjoints and far-away orthogonality}(ii) that the functor $\mathsf{i}\colon\mathcal{X}\rightarrow \mathcal{T}$ restricts to bounded projective objects. By Lemma \ref{common facts of the subcategories}(i) we have $\mathcal{X}^{\mathsf{c}}\subseteq \mathcal{X}^{\mathsf{b}}_p$ which completes the claim.  We now prove (ii)$\Longrightarrow$(i). Since the functors $\mathsf{q}$ and $\mathsf{l}$ restrict to compact objects, it follows by Lemma \ref{adjoints and far-away orthogonality}(i) that the functors $\mathsf{i}$ and $\mathsf{e}$ restrict to bounded objects. Moreover, since $\mathsf{i}(\mathcal{X}^{\mathsf{c}})\subseteq \mathcal{T}^{\mathsf{b}}_p$, it follows by Lemma \ref{adjoints and far-away orthogonality}(i) and Lemma \ref{common facts of the subcategories} that $\mathsf{p}(\mathcal{T}^{\mathsf{b}})\subseteq \mathcal{X}^{\mathsf{b}}$. From Lemma \ref{inclusions_in_recollements}(i) and the assumption, it follows that $\mathsf{e}(\mathcal{T}^{\mathsf{c}})\subseteq\mathcal{Y}^{\mathsf{b}}_p$. Then, by the same argument as above, we derive that $\mathsf{r}(\mathcal{Y}^{\mathsf{b}})\subseteq \mathcal{T}^{\mathsf{b}}$. Lastly, again by Lemma \ref{inclusions_in_recollements} and the assumptions, it follows that $\mathsf{q}(\mathcal{T}^{\mathsf{b}})\subseteq \mathcal{X}^{\mathsf{b}}$ which completes the proof. 
\end{proof}

The following remark is of independent interest. 

\begin{rem}
    Recall from \cite[Appendix C.]{completing perfect complexes} that a \emph{morphic enchancement} of a triangulated category $\mathcal{T}$ is a triangulated category $\mathcal{U}$ with a fully faithful functor $\mathsf{i}\colon\mathcal{T}\rightarrow \mathcal{U}$ that fits in a recollement 
    \[
\begin{tikzcd}
\mathcal{T} \arrow[rr, "\mathsf{i}"] &  & \mathcal{U} \arrow[rr, "\mathsf{e}"] \arrow[ll, "\mathsf{q}"', bend right] \arrow[ll, "\mathsf{p}", bend left] &  & \mathcal{T} \arrow[ll, "\mathsf{l}"', bend right] \arrow[ll, "\mathsf{r}", bend left]
\end{tikzcd}
    \]
    such that $(\mathsf{p},\mathsf{l})$ is an adjoint pair. This recollement sits in an infinite ladder (see \cite[Proposition C.3.]{completing perfect complexes}) and it then follows from Proposition \ref{restriction_to_type_B} that $\mathcal{U}^{\mathsf{b}}$ is a morphic enchancement of $\mathcal{T}^{\mathsf{b}}$. 
\end{rem}

We now assume $\T$ to be compactly generated and $k$-linear over a commutative noetherian ring $k$. We investigate whether a recollement of type (U) restricts to a recollement as follows:
$$\begin{tikzcd}
\mathcal{X}^{\mathsf{b}}_c \arrow[rr, "\mathsf{i}"] &  & \mathcal{T}^{\mathsf{b}}_c \arrow[rr, "\mathsf{e}"] \arrow[ll, "\mathsf{q}"', bend right] \arrow[ll, "\mathsf{p}", bend left] &  & \mathcal{Y}^{\mathsf{b}}_c. \arrow[ll, "\mathsf{l}"', bend right] \arrow[ll, "\mathsf{r}", bend left]
\end{tikzcd} \textnormal{($\mathrm{B}'$)} 
$$
We will need an additional condition of $\T$ in order to prove our next statement. This condition is satisfied in many categories of interest and consists of the equality
\begin{equation} \label{reflexive}
    \mathcal{T}^{\mathsf{c}}={^{\hfpadova}}(\mathcal{T}^{\mathsf{b}}_c).
\end{equation}
Indeed, this condiiton is verified for the derived category of a finite dimensional algebra, see for instance \cite[Lemma 2.4]{angeleri_koenig_liu_yang}, for the derived category of a projective scheme over a field, see \cite[Lemma 7.49]{rouquier} and for the derived category of a proper differential graded algebra, see \cite[Lemma 4.13]{adachi_mizuno_yang}. Using it we can prove a triangulated version of \cite[Theorem 4.6]{angeleri_koenig_liu_yang}.

\begin{prop} \label{restricting to bounded finite}
Let $\T$, $\X$ and $\Y$ be compactly generated, $k$-linear triangulated categories over a commutative noetherian ring $k$. If $\mathsf{i}(\mathcal{X}^{\mathsf{c}})\subseteq \mathcal{T}^{\mathsf{c}}$ and $\mathsf{l}(\mathcal{Y}^{\mathsf{b}}_c)\subseteq \mathcal{T}^{\mathsf{b}}_c$ then a recollement of type \textnormal{($\mathrm{U}$)} restricts to a recollement of type \textnormal{($\mathrm{B}'$)}. Moreover, if $\mathcal{X}$ and $\mathcal{T}$ satisfy \textnormal{(\ref{reflexive})} then the converse also holds.
\end{prop}
\begin{proof}
Suppose that $\mathsf{i}(\mathcal{X}^{\mathsf{c}})\subseteq \mathcal{T}^{\mathsf{c}}$ and $\mathsf{l}(\mathcal{Y}^{\mathsf{b}}_c)\subseteq \mathcal{T}^{\mathsf{b}}_c$. From the assumption on the functor $\mathsf{i}$ it follows from Lemma \ref{compacts_in_recollements} that the functor $\mathsf{e}$ preserves compact objects. Therefore, the functors $\mathsf{i},\mathsf{e},\mathsf{r}$ and $\mathsf{p}$ restrict to bounded finite objects (being right adjoints to compact-preserving functors). Since we have $\mathsf{l}(\mathcal{Y}^{\mathsf{b}}_c)\subseteq \mathcal{T}^{\mathsf{b}}_c$ by assumption, we are left to show the inclusion $\mathsf{q}(\mathcal{T}^{\mathsf{b}}_c)\subseteq \mathcal{X}^{\mathsf{b}}_c$. Given $t$ in $\mathcal{T}^{\mathsf{b}}_c$, consider the triangle 
\[
\mathsf{le}(t)\rightarrow t\rightarrow \mathsf{iq}(t)\rightarrow \mathsf{le}(t)[1].
\]
By the considerations above, $\mathsf{le}(t)$ lies in $\mathcal{T}^{\mathsf{b}}_c$ and thus it follows that $\mathsf{iq}(t)$ lies in $\mathcal{T}^{\mathsf{b}}_c$. Since $\mathsf{i}$ is fully faithful, we infer that $\mathsf{q}(t)$ lies in $\mathcal{X}^{\mathsf{b}}_c$ from Lemma \ref{adjoints and far-away orthogonality}(i). For the converse, suppose that $\mathcal{X}$ and $\mathcal{T}$ satisfy (\ref{reflexive}). We then have $\mathsf{i}(\mathcal{X}^{\mathsf{c}})=\mathsf{i}({}^{\hfpadova}(\mathcal{X}^{\mathsf{b}}_c))\subseteq {}^{\hfpadova}(\mathcal{T}^{\mathsf{b}}_c)=\mathcal{T}^{\mathsf{c}}$, where the inclusion follows from the inclusion $\mathsf{p}(\mathcal{T}^{\mathsf{b}}_c)\subseteq \mathcal{X}^{\mathsf{b}}_c$ using Lemma \ref{adjoints and far-away orthogonality}(ii). 
\end{proof}

\subsection{Homological properties and recollements}
\label{subsect:regularandrecol} In this subsection, we compare regularity, injective generation and the finiteness of the finitistic dimension for categories in a recollement. 

We begin with a result regarding singularity categories and recollements in the spirit of similar well-known results in the literature; for instance \cite[Corollary 4.2]{jin_yang_zhou} will follow from this result using Proposition \ref{sing ex}. 

\begin{thm}
\label{prop:recolandsingularcats}
    Consider a recollement of compactly generated $k$-linear triangulated categories 
    \[
    \begin{tikzcd}
\mathcal{X} \arrow[rr, "\mathsf{i}"] &  & \mathcal{T} \arrow[rr, "\mathsf{e}"] \arrow[ll, "\mathsf{q}"', bend right] \arrow[ll, "\mathsf{p}", bend left] &  & \mathcal{Y} \arrow[ll, "\mathsf{l}"', bend right] \arrow[ll, "\mathsf{r}", bend left]
\end{tikzcd}
    \]
    satisfying $\mathcal{X}^{\mathsf{c}}\subseteq \mathcal{X}^{\mathsf{b}}_c, \mathcal{T}^{\mathsf{c}}\subseteq \mathcal{T}^{\mathsf{b}}_c$ and $\mathcal{Y}^{\mathsf{c}}\subseteq \mathcal{Y}^{\mathsf{b}}_c$. The following hold\textnormal{:}
    \begin{itemize}
        \item[\textnormal{(i)}] The functor $\mathsf{i}$ induces an equivalence $\mathcal{X}^{\mathsf{sg}}_k\xrightarrow{\simeq}\mathcal{T}^{\mathsf{sg}}_k$ if and only if $\mathcal{Y}$ is regular. 
        \item[\textnormal{(ii)}] The functor $\mathsf{e}$ induces an equivalence $\mathcal{T}^{\mathsf{sg}}_k\xrightarrow{\simeq}\mathcal{Y}^{\mathsf{sg}}_k$ if and only if $\mathsf{i}(\mathcal{X}^{\mathsf{c}})\subseteq \mathcal{T}^{\mathsf{c}}$ and $\mathcal{X}$ is regular. 
    \end{itemize}
\end{thm}
\begin{proof}
    (i) Assume first that $\mathcal{Y}$ is regular. Then, since the functor $\mathsf{l}$ preserves compact objects, it follows that $\mathsf{e}(\mathcal{T}^{\mathsf{b}}_c)\subseteq \mathcal{Y}^{\mathsf{b}}_c=\mathcal{Y}^{\mathsf{c}}$. By assumption $\mathcal{T}^{\mathsf{c}}\subseteq \mathcal{T}^{\mathsf{b}}_c$ and therefore $\mathsf{e}(\mathcal{T}^{\mathsf{c}})\subseteq\mathcal{Y}^{\mathsf{c}}$. We then know from Lemma \ref{compacts_in_recollements} that the functor $\mathsf{i}$ also preserves compact objects. Moreover, we have $\mathsf{i}(\mathcal{X}^{\mathsf{b}}_c)\subseteq \mathcal{T}^{\mathsf{b}}_c$ (since $\mathsf{q}$ preserves compacts) and therefore there is a short exact sequence $\mathcal{X}^{\mathsf{sg}}_k\rightarrow \mathcal{T}^{\mathsf{sg}}_k\rightarrow \mathcal{Y}^{\mathsf{sg}}_k$ up to direct summands (see for instance \cite[Lemma 3.1]{jin_yang_zhou}). Since $\mathcal{Y}$ is regular, the result follows. 

    Assume now that $\mathsf{i}$ induces as equivalence $\mathcal{X}^{\mathsf{sg}}_k\rightarrow \mathcal{T}^{\mathsf{sg}}_k$. Then, in particular, $\mathsf{i}(\mathcal{X}^{\mathsf{c}})\subseteq \mathcal{T}^{\mathsf{c}}$. It follows by Lemma \ref{compacts_in_recollements} that the functor $\mathsf{e}$ preserves compact objects and again by the same arguments as above, there is a short exact sequence $\mathcal{X}^{\mathsf{sg}}_k\rightarrow \mathcal{T}^{\mathsf{sg}}_k\rightarrow \mathcal{Y}^{\mathsf{sg}}_k$, up to direct summands.
    Since the functor  $\mathsf{i}\colon \mathcal{X}^{\mathsf{sg}}_k\rightarrow \mathcal{T}^{\mathsf{sg}}_k$ is assumed to be an equivalence, it follows that $\mathcal{Y}^{\mathsf{sg}}_k\simeq 0$, i.e. $\mathcal{Y}$ is regular.  \\ 
    (ii) Assume first that $\mathsf{i}(\mathcal{X}^{\mathsf{c}})\subseteq \mathcal{T}^{\mathsf{c}}$ and that $\mathcal{X}$ is regular. By Lemma \ref{compacts_in_recollements} it follows that the functor $\mathsf{e}$ preserves compact objects and since both $\mathsf{i}$ and $\mathsf{e}$ preserve bounded finite objects (being right adjoints to compact-preserving functors), it follows that there is a short exact sequence $\mathcal{X}^{\mathsf{sg}}_k\rightarrow \mathcal{T}^{\mathsf{sg}}_k\rightarrow \mathcal{Y}^{\mathsf{sg}}_k$ up to direct summands. Since $\mathcal{X}$ is assumed to be regular, it follows that $\mathcal{X}^{\mathsf{sg}}_k\simeq 0$, i.e.\ $\mathsf{e}\colon \mathcal{T}^{\mathsf{sg}}_k\rightarrow \mathcal{Y}^{\mathsf{sg}}_k$ is an equivalence. 

    Assume now that $\mathsf{e}$ induces an equivalence $\mathcal{T}^{\mathsf{sg}}_k\rightarrow \mathcal{Y}^{\mathsf{sg}}_k$. Then, in particular, $\mathsf{e}(\mathcal{T}^{\mathsf{c}})\subseteq \mathcal{Y}^{\mathsf{c}}$ and therefore by Lemma \ref{compacts_in_recollements} the functor $\mathsf{i}$ is compact-preserving. We deduce, like before, the existence of a short exact sequence $\mathcal{X}^{\mathsf{sg}}_k\rightarrow \mathcal{T}^{\mathsf{sg}}_k\rightarrow \mathcal{Y}^{\mathsf{sg}}_k$ up to direct summands and since $\mathcal{T}^{\mathsf{sg}}_k\rightarrow \mathcal{Y}^{\mathsf{sg}}_k$ is an equivalence, it follows that $\mathcal{X}^{\mathsf{sg}}_k\simeq 0$ and therefore $\mathcal{X}$ is regular.  
\end{proof}

\begin{rem}
    Note that using Proposition \ref{sing ex} , we can also obtain a triangulated analogue of \cite[Corollary 5.2]{jin_yang_zhou}. If we assume a recollement of compactly generated triangulated categories as in the proposition above, such that $\mathcal{X}^{\mathsf{b}}_c\subseteq \mathcal{X}^{\mathsf{c}}$, $\mathcal{T}^{\mathsf{b}}_c\subseteq \mathcal{T}^{\mathsf{c}}$ and $\mathcal{Y}^{\mathsf{b}}_c\subseteq \mathcal{Y}^{\mathsf{c}}$, then the following hold. 
    \begin{itemize}
        \item[\textnormal{(i)}] The functor $\mathsf{i}$ induces an equivalence $\mathcal{X}^{\mathsf{sg}}_k\xrightarrow{\simeq}\mathcal{T}^{\mathsf{sg}}_k$ if and only if $\mathsf{e}(\mathcal{T}^{\mathsf{c}})\subseteq \mathcal{Y}^{\mathsf{c}}$ and $\mathcal{Y}$ is regular. 
        \item[\textnormal{(ii)}] The functor $\mathsf{e}$ induces an equivalence $\mathcal{T}^{\mathsf{sg}}_k\xrightarrow{\simeq}\mathcal{Y}^{\mathsf{sg}}_k$ if and only if $\mathcal{X}$ is regular.  
    \end{itemize}
\end{rem}

Consider now the recollements of type ($\mathrm{U}$) and $(\mathrm{B})$ of Subsection \ref{restricting recollements}.

\[
\begin{tikzcd}
\mathcal{X} \arrow[rr, "\mathsf{i}"] &  & \mathcal{T} \arrow[rr, "\mathsf{e}"] \arrow[ll, "\mathsf{q}"', bend right] \arrow[ll, "\mathsf{p}", bend left] &  & \mathcal{Y} \arrow[ll, "\mathsf{l}"', bend right] \arrow[ll, "\mathsf{r}", bend left]
\end{tikzcd} \textnormal{($\mathrm{U}$)} \ \ \ 
\begin{tikzcd}
\mathcal{X}^{\mathsf{b}} \arrow[rr, "\mathsf{i}"] &  & \mathcal{T}^{\mathsf{b}} \arrow[rr, "\mathsf{e}"] \arrow[ll, "\mathsf{q}"', bend right] \arrow[ll, "\mathsf{p}", bend left] &  & \mathcal{Y}^{\mathsf{b}} \arrow[ll, "\mathsf{l}"', bend right] \arrow[ll, "\mathsf{r}", bend left]
\end{tikzcd} \textnormal{($\mathrm{B}$)}
\]

Under the existence of a recollement of bounded derived categories $(\mathsf{D}^{\mathsf{b}}(A), \mathsf{D}^{\mathsf{b}}(B), \mathsf{D}^{\mathsf{b}}(C))$ for rings $A,B$ and $C$, it is known that $\gd B<\infty$ if and only if $\gd A<\infty$ and $\gd C<\infty$. The same occurs for a recollement of unbounded derived categories (see \cite[Proposition 2.14]{angeleri_koenig_liu_yang}). Below we present a triangulated analogue for these results under the assumption that the bounded projective objects are bounded. 

\begin{prop} \label{regularity_and_recollements_for_bounded_objects}\label{comparing_regularity_in_type_U}
Suppose that $\X$, $\Y$ and $\T$ are triangulated categories with the property that bounded projective objects are bounded. Then the following statements hold.
\begin{enumerate}
\item If there is a recollement of type \textnormal{($\mathrm{B}$)}, then $\gd \mathcal{T}<\infty$ if and only if $\gd \mathcal{X}<\infty$ and $\gd\mathcal{Y}<\infty$.
\item If there is a recollement of type \textnormal{($\textnormal{U}$)} and if $\gd\mathcal{X}<\infty$ and $\gd\mathcal{Y}<\infty$, then $\gd\mathcal{T}<\infty$.
\end{enumerate}
\end{prop}
\begin{proof}
(i) By the same argument used in Lemma \ref{adjoints and far-away orthogonality}, our assumption guarantees that the functors of the given recollement restrict to a colocalisation sequence of the form
    \[
    \begin{tikzcd}
\mathcal{X}^{\mathsf{b}}_p \arrow[rr, "\mathsf{i}"] &  & \mathcal{T}^{\mathsf{b}}_p \arrow[rr, "\mathsf{e}"] \arrow[ll, "\mathsf{q}"', bend right] &  & \mathcal{Y}^{\mathsf{b}}_p. \arrow[ll, "\mathsf{l}"', bend right]
\end{tikzcd}
    \]
    Assume that $\gd\mathcal{T}<\infty$ and let $y$ be an object of $\mathcal{Y}^{\mathsf{b}}$. Then $\mathsf{l}(y)$ lies in $\mathcal{T}^{\mathsf{b}}=\mathcal{T}^{\mathsf{b}}_p$. Since $\mathsf{l}$ is fully faithful and $\mathsf{l}(\mathcal{Y}^{\mathsf{b}})\subseteq \mathcal{T}^{\mathsf{b}}$, it follows by Lemma \ref{fully_faithful_functor} that  $y$ lies in $\mathcal{Y}^{\mathsf{b}}_p$, i.e.~we have $\gd\mathcal{Y}<\infty$. The same argument applies for $\mathcal{X}$ using the functor $\mathsf{i}$. Assume now that $\mathcal{X}$ and $\mathcal{Y}$ have finite global dimension. Let $t$ be an object in $\mathcal{T}^{\mathsf{b}}$ and consider the triangle 
    $$\mathsf{le}(t)\rightarrow t\rightarrow \mathsf{iq}(t)\rightarrow \mathsf{le}(t)[1].$$ 
    We have that $\mathsf{e}(t)$ lies in $\mathcal{Y}^{\mathsf{b}}=\mathcal{Y}^{\mathsf{b}}_p$ and $\mathsf{q}(t)$ lies in $\mathcal{X}^{\mathsf{b}}=\mathcal{X}^{\mathsf{b}}_p$. Consequently, $\mathsf{le}(t)$ lies in $\mathcal{T}^{\mathsf{b}}_p$ and $\mathsf{iq}(t)$ lies in $\mathcal{T}^{\mathsf{b}}_p$. From the above triangle, we infer that $t$ belongs to $\mathcal{T}^{\mathsf{b}}_p$, i.e.\ $\mathcal{T}$ has finite global dimension.
    
    (ii) By (i), it is enough to show that when $\gd\mathcal{X}<\infty$ and $\gd\mathcal{Y}<\infty$, then the given recollement restricts to a recollement of type \textnormal{($\mathrm{B}$)}. By Proposition \ref{restriction_to_type_B} this is equivalent to $\mathsf{i}(\mathcal{X}^{\mathsf{c}})\subseteq \mathcal{T}^{\mathsf{b}}_p$ and $\mathsf{l}(\mathcal{Y}^{\mathsf{b}})\subseteq \mathcal{T}^{\mathsf{b}}$. We begin by showing the former. For $t$ a compact object of $\mathcal{T}$, we consider the following triangle
    \[
    \mathsf{le}(t)\rightarrow t\rightarrow \mathsf{iq}(t)\rightarrow \mathsf{le}(t)[1].
    \]
    Since $\mathsf{e}(t)\in\mathcal{Y}^{\mathsf{b}}=\mathcal{Y}^{\mathsf{b}}_p$, we infer that $\mathsf{le}(t)\in\mathcal{T}^{\mathsf{b}}_p$. It follows by the above triangle that $\mathsf{iq}(t)\in\mathcal{T}^{\mathsf{b}}_p$ for every $t\in\mathcal{T}^{\mathsf{c}}$. Since $\mathcal{X}^{\mathsf{c}}=\mathsf{thick}(\mathsf{q}(\mathcal{T}^{\mathsf{c}}))$, the claim follows. We will now show that $\mathsf{q}(\mathcal{Y}^{\mathsf{b}})\subseteq \mathcal{T}^{\mathsf{b}}$. Indeed, we have $\mathcal{Y}^{\mathsf{b}}=\mathcal{Y}^{\mathsf{b}}_p$ and therefore the inclusions $\mathsf{q}(\mathcal{Y}^{\mathsf{b}})=\mathsf{q}(\mathcal{Y}^{\mathsf{b}}_p)\subseteq \mathcal{T}^{\mathsf{b}}_p\subseteq \mathcal{T}^{\mathsf{b}}$, completing the proof. 
\end{proof}

We now discuss injective generation, showing a version of \cite[Propositions 6.5 and 6.6]{cummings}.

\begin{prop} \label{injective_generation_in_recollement}
    Suppose that $\X$, $\Y$ and $\T$ are triangulated categories in a recollement of type \textnormal{($\mathrm{U}$)}. 
    \begin{enumerate}
    \item If $\mathsf{Im}(\mathsf{i})\subseteq \mathsf{Loc}(\mathcal{T}^{\mathsf{b}}_i)$ and injectives generate for $\mathcal{Y}$, then injectives generate for $\mathcal{T}$. 
    \item  If $\mathsf{i}(\mathcal{X}^{\mathsf{b}}_i)\subseteq \mathcal{T}^{\mathsf{b}}_i$ and injectives generate for $\mathcal{X}$ and $\mathcal{Y}$, then injectives generate for $\mathcal{T}$.
    \item If $\mathsf{e}(\mathcal{T}^{\mathsf{b}}_i)\subseteq \mathcal{Y}^{\mathsf{b}}_i$ and injectives generate for $\mathcal{T}$, then injectives generate for $\mathcal{Y}$. 
    \end{enumerate}
\end{prop}
\begin{proof}
(i) Let $y$ be an object in $\mathcal{Y}^{\mathsf{b}}_i$ and note that $\mathsf{r}(y)$ is an object of $\mathcal{T}^{\mathsf{b}}_i$ by Lemma  \ref{adjoints and far-away orthogonality}(i). Consider the following triangle in $\mathcal{T}$ 
\[
 \mathsf{ler}(y)\rightarrow \mathsf{r}(y)\rightarrow \mathsf{iqr}(y)\rightarrow \mathsf{ler}(y)[1].
\]
    Since $\mathsf{Im}(\mathsf{i})\subseteq \mathsf{Loc}(\mathcal{T}^{\mathsf{b}}_i)$, we conclude that $\mathsf{ler}(y)$ lies in $\mathsf{Loc}(\mathcal{T}^{\mathsf{b}}_i)$. However, we have $\mathsf{ler}(y)\cong \mathsf{l}(y)$ and so the functor $\mathsf{l}$ maps $\mathcal{Y}^{\mathsf{b}}_i$ to $\mathsf{Loc}(\mathcal{T}^{\mathsf{b}}_i)$. Since injectives generate for $\mathcal{Y}$ and $\mathsf{l}$ is coproduct preserving, it follows that $\mathsf{Im}(\mathsf{l})\subseteq \mathsf{Loc}(\mathcal{T}^{\mathsf{b}}_i)$. For any $t$ in $\mathcal{T}$, consider the triangle 
    $$\mathsf{le}(t)\rightarrow t\rightarrow \mathsf{iq}(t)\rightarrow \mathsf{le}(t)[1].$$ 
    Both $\mathsf{le}(t)$ and $\mathsf{iq}(t)$ lie inside $\mathsf{Loc}(\mathcal{T}^{\mathsf{b}}_i)$, showing that also $t$ lies in $\mathsf{Loc}(\mathcal{T}^{\mathsf{b}}_i)$. 
    
    (ii) By the assumption and the fact that $\mathsf{i}$ preserves coproducts, it follows that $\mathsf{Im}(\mathsf{i})\subseteq \mathsf{Loc}(\mathcal{T}^{\mathsf{b}}_i)$. We can then deduce the claim from (i). 
    
    (iii) By combining the assumption with the fact that $\mathsf{e}$ preserves coproducts, it follows that $\mathsf{Im}(\mathsf{e})\subseteq \mathsf{Loc}(\mathcal{Y}^{\mathsf{b}}_i)$. But $\mathsf{Im e}=\mathcal{Y}$ since $\mathsf{e}$ is essentially surjective. 
\end{proof}

%

Regarding the finiteness of the finitistic dimension we prove a version of \cite[Theorem 1.1]{chen_xi}. 

\begin{prop} 
    Suppose that $\X$, $\Y$ and $\T$ are triangulated categories in a recollement of type \textnormal{($\mathrm{U}$)}. 
       \begin{enumerate}
        \item Assume that $\mathsf{p}$ admits a right adjoint $\mathsf{p}'$ such that $\mathsf{p}'(\mathcal{X}^+)\subseteq \mathcal{T}^+$ and $\mathsf{r}$ admits a right adjoint $\mathsf{r'}$ such that $\mathsf{r}'(\mathcal{T}^+)\subseteq\mathcal{Y}^+$. Then the following hold\textnormal{:} 
        \begin{enumerate}
            \item[\textnormal{(a)}] if $\Findim\mathcal{X}<\infty$ and $\Findim\mathcal{Y}<\infty$, then $\Findim\mathcal{T}<\infty$;
        \item[\textnormal{(b)}] if $\Findim\mathcal{T}<\infty$, then $\Findim\mathcal{X}<\infty$. 
        \end{enumerate}
        \item Assume that $\mathsf{l}(\mathcal{Y}^{\mathsf{b}})\subseteq \mathcal{T}^{\mathsf{b}}$ and $\mathsf{r}(\mathcal{Y}^+)\subseteq \mathcal{T}^+$. If $\Findim\mathcal{T}<\infty$, then $\Findim\mathcal{Y}<\infty$. 
    \end{enumerate}
\end{prop}
\begin{proof}
(i) Under the given assumptions, using Lemma \ref{adjoints and far-away orthogonality}, we have the existence of the functors:
   \[
   \begin{tikzcd}
\mathcal{X}^+ \arrow[rr, "\mathsf{i}"] \arrow[rr, "\mathsf{p}'"', bend right=60] &  & \mathcal{T}^+ \arrow[rr, "\mathsf{e}"] \arrow[ll, "\mathsf{p}", bend left] \arrow[rr, "\mathsf{r}'"', bend right=60] &  & \mathcal{Y}^+ \arrow[ll, "\mathsf{r}", bend left] & \mathcal{X}^{\mathsf{b}}_i \arrow[rr, "\mathsf{p}'"', bend right=60] &  & \mathcal{T}^{\mathsf{b}}_i \arrow[ll, "\mathsf{p}", bend left] \arrow[rr, "\mathsf{r}'"', bend right=60] &  & \mathcal{Y}^{\mathsf{b}}_i. \arrow[ll, "\mathsf{r}", bend left]
\end{tikzcd}
   \] 
   
 (a) Consider an object $t$ in $\mathcal{T}^{\mathsf{b}}_i$$^{\perp}\cap \mathcal{T}^+$. We first claim that $\mathsf{r}'(t)$ lies in $\mathcal{Y}^{\mathsf{b}}_i$$^{\perp}\cap \mathcal{Y}^+$. Indeed, given $y$ in $\mathcal{Y}^{\mathsf{b}}_i$, we have, for all $n$ in $\mathbb{Z}$,
   \[
   \mathsf{Hom}_{\mathcal{Y}}(y,\mathsf{r}'(t)[n])\cong \mathsf{Hom}_{\mathcal{T}}(\mathsf{r}(y),t[n])\cong 0
   \]
thus implying the claim. Since we assume $\mathcal{Y}$ to have finite finitistic dimension, it follows that $\mathsf{r}'(t)\cong 0$. Consequently, using the triangle 
   $$\mathsf{p}'\mathsf{p}(t)\rightarrow t\rightarrow \mathsf{r}\mathsf{r}'(t)\rightarrow \mathsf{p}'\mathsf{p}(t)[1],$$
   we infer that $\mathsf{p}'\mathsf{p}(t)\cong t$. Let now $x$ be an object in $\mathcal{X}^{\mathsf{b}}_i$. Then, for all $n\in \mathbb{Z}$ we have 
   \[
   0\cong \mathsf{Hom}_{\mathcal{T}}(\mathsf{p}'(x),t[n])\cong \mathsf{Hom}_{\mathcal{T}}(\mathsf{p}'(x),\mathsf{p}'\mathsf{p}(t))\cong \mathsf{Hom}_{\mathcal{X}}(x,\mathsf{p}(t)[n]). 
   \]
   Hence, we get that $\mathsf{p}(t)$ lies in $\mathcal{X}^{\mathsf{b}}_i$$^{\perp}\cap \mathcal{X}^+$. Since we assumed $\mathcal{X}$ to have finite finitisic dimension, it follows that $\mathsf{p}(t)\cong 0$ and therefore $t\cong \mathsf{p}'\mathsf{p}(t)\cong 0$. 

(b) Consider an object $x$ in $\mathcal{X}^{\mathsf{b}}_i$$^{\perp}\cap \mathcal{X}^+$. Then, for  $t$ in $\mathcal{T}^{\mathsf{b}}_i$ we have
   \[
   0\cong \mathsf{Hom}_{\mathcal{X}}(\mathsf{p}(t),x[n])\cong \mathsf{Hom}_{\mathcal{T}}(t,\mathsf{p}'(x)[n])
   \]
   for all $n$ in $\mathbb{Z}$. Hence, we get that $\mathsf{p}'(x)$ lies in $\mathcal{T}^{\mathsf{b}}_i$$^{\perp}\cap \mathcal{T}^+$. Since we assumed $\mathcal{T}$ to have finite finitistic dimension, we see that $\mathsf{p}'(x)\cong 0$ and since $\mathsf{p}'$ is fully faithful, we infer that $x\cong 0$. 

(ii) Under the given assumptions, using Lemma \ref{adjoints and far-away orthogonality}, follows the existence of the following functors 
   \[
   \begin{tikzcd}
\mathcal{T}^{\mathsf{b}}_i \arrow[rr, "\mathsf{e}"] &  & \mathcal{Y}^{\mathsf{b}}_i \arrow[ll, "\mathsf{r}", bend left] & \mathcal{T}^+ \arrow[rr, "\mathsf{e}"] &  & \mathcal{Y}^+ \arrow[ll, "\mathsf{r}", bend left]
\end{tikzcd}
   \]
   Consider an object $y$  of $\mathcal{Y}^{\mathsf{b}}_i$$^{\perp}\cap \mathcal{Y}^+$. For $t$ in $\mathcal{T}^{\mathsf{b}}_i$ we have 
   \[
   0\cong \mathsf{Hom}_{\mathcal{Y}}(\mathsf{e}(t),y[n])\cong \mathsf{Hom}_{\mathcal{T}}(t,\mathsf{r}(y)[n])
   \]
   for all $n$ in $\mathbb{Z}$ and, thus, $\mathsf{r}(y)$ lies in $\mathcal{T}^{\mathsf{b}}_i$$^{\perp}\cap\mathcal{T}^+$. Since we assumed $\mathcal{T}$ to have finite finitistic dimension, it follows that $\mathsf{r}(y)\cong 0$ and since $\mathsf{r}$ is fully faithful, $y\cong 0$. 
\end{proof}


\begin{thebibliography}{99}

\bibitem{adachi_mizuno_yang} 
    \textsc{T.~Adachi, Y.~Mizuno, D.~Yang}, \emph{Discreteness of silting objects and t-structures in triangulated categories}, Proc.\ Lond.\ Math.\ Soc.\ (3) 118 (2019), no.\ 1, 1–42.

\bibitem{aihara_iyama}
    \textsc{T.~Aihara, O.~Iyama}, \emph{Silting mutation in triangulated categories}, J.\ Lond.\ Math.\ Soc.\ (2) 85 (2012), no.\ 3, 633–668.


\bibitem{amiot} 
    \textsc{C.~Amiot}, \emph{Cluster categories for algebras of global dimension 2 and quivers with potential}, Ann.\ Inst.\ Fourier (Grenoble) 59 (2009), no.\ 6, 2525–2590.


\bibitem{anderson_fuller} 
    \textsc{F.~Anderson, K.~Fuller}, \emph{Rings and categories of modules}, Second edition Grad.\ Texts in Math., 13 Springer-Verlag, New York, 1992. x+376 pp.

\bibitem{angeleri_koenig_liu} 
    \textsc{L.~Angeleri H\"ugel, S.~Koenig, Q.~Liu}, \emph{On the uniqueness of stratifications of derived module categories}, J.\ Algebra 359 (2012), 120–137.
     

\bibitem{angeleri_koenig_liu_yang} 
     \textsc{L.~Angeleri H\"ugel, S.~Koenig, Q.~Liu, D.~Yang}, \emph{Ladders and simplicity of derived module categories}, J.\ Algebra 472, 15-66 (2017).

     \bibitem{angeleri_marks_vitoria}
     \textsc{L.~Angeleri H\"ugel, F.~Marks, J.~Vit\'oria}, \emph{Partial silting objects and smashing subcategories}, Math.~Z.~296, 887-900 (2020).

\bibitem{aoki} 
    \textsc{K.~Aoki}, \emph{Quasiexcellence implies strong generation}, J. Reine Angew. Math. 780 (2021), 133–138.

\bibitem{auslander_reiten} 
     \textsc{M.~Auslander, I.~Reiten}, 
     \emph{Applications of contravariantly finite subcategories}, Adv.\ Math.\ 86, No.\ 1, 111-152 (1991).

\bibitem{auslander_reiten_smalo} 
    \textsc{M.~Auslander, I.~Reiten, S.~Smalø}, \emph{Representation theory of Artin algebras}, Cambridge Stud.\ Adv.\ Math., 36
Cambridge University Press, Cambridge, 1995. xiv+423 pp.


\bibitem{avramov_veliche}
\textsc{L.L.~Avramov, O.~Veliche},
    \emph{Stable cohomology of local rings}, Adv.\ Math.\ 213 (2007), no.\ 1, 93–139.



\bibitem{beilinson_bernstein_deligne} 
    \textsc{A.~Beilinson, J.~Bernstein, P.~Deligne}, \emph{Faisceaux pervers}, Analysis and topology on singular spaces, I (Luminy, 1981), 5–171. Astérisque, 100 Société Mathématique de France, Paris, 1982.

\bibitem{beligiannis2} 
    \textsc{A.~Beligiannis}, 
    \emph{Cohen-Macaulay modules, (co)torsion pairs and virtually Gorenstein algebras}, J.\ Algebra 288 (2005), no.\ 1, 137–211.


\bibitem{local_regularity} 
\textsc{D.~Benson, S.~Iyengar, H.~Krause, J.~Pevtsova}, \emph{Locally dualisable modular representations and local regularity}, arxiv:2404.14672. 


\bibitem{biswas_chen_manali-rahul_parker_zheng} 
     \textsc{R.~Biswas, H.~Chen, K.~Manali Rahul, C.~Parker, J.~Zheng}, \emph{Bounded t-structures, finitistic dimensions, and singularity categories of triangulated categories}, arXiv:2401.00130.

\bibitem{bokstedt_neeman} 
     \textsc{M.~Bokstedt, A.~Neeman}, \emph{Homotopy limits in triangulated categories}, Compositio Math.\ 86 (1993), no.\ 2, 209–234.


\bibitem{bondal_van-den-bergh} 
     \textsc{A.~Bondal, M.~Van den Bergh}, \emph{Generators and representability of functors in commutative and
noncommutative geometry}, Moscow Math.\ J.\ 3 (2003), 1–36.

\bibitem{breaz_rafiliu}
\textsc{S.~Breaz, C.~Rafiliu}, \emph{Comparing $\mathsf{Add}(M)$ with $\mathsf{Prod}(M)$}, J.~Algebra 682 (2025), 804--823.

\bibitem{buchweitz} 
    \textsc{R.-O.~Buchweitz}, \emph{Maximal Cohen–Macaulay modules and Tate-cohomology over Gorenstein rings}, University of Hannover, 1986. 


\bibitem{burke_neeman_pauwels}
     \textsc{J.~Burke, A.~Neeman, B.~Pauwels}, \emph{Gluing approximable triangulated categories}, Forum Math.\ Sigma 11 (2023), Paper No.\ e110, 18 pp.


\bibitem{CanonacoNeemanStellari}
\textsc{A.~Canonaco, A.~Neeman, P.~Stellari}, {\em Weakly approximable triangulated categories and enhancements: a survey}, Boll.\ Unione Mat.\ Ital.\ 18 (2025), no.\ 1, 109--134.



\bibitem{CanonacoNeemanStellariHaesemeyer}
\textsc{A.~Canonaco, A.~Neeman, P.~Stellari}, with an appendix by \textsc{C.~Haesemeyer}, {\em The passage among subcategories of weakly approximable triangulated categories}, arXiv:2402.04605.

\bibitem{chen_xi} 
   \textsc{H.~Chen, C.~Xi},
   \emph{Recollements of derived categories III: finitistic dimensions}, J.\ Lond.\ Math.\ Soc.\ (2) 95 (2017), no.\ 2, 633–658.

\bibitem{chen2} 
   \textsc{X.-W.~Chen}, \emph{Singularity categories, Schur functors and triangular matrix rings}, Algebr. Represent. Theory 12, No. 2-5, 181-191 (2009).

\bibitem{chen} 
   \textsc{X.-W.~Chen} 
   \emph{The singularity category of an algebra with radical square zero}, Doc. Math. 16 (2011), 921–936.

\bibitem{chen_li_zhang_zhao}  
   \textsc{X.-W.~Chen, Z.-W.~Li, X.~Zhang, Z.~Zhao}, \emph{A non-vanishing result on the singularity category}, Proc.\ Amer.\ Math.\ Soc.\ 152 (2024), no.\ 9, 3765–3776.



\bibitem{conde_gorsky_marks_zvonareva}
\textsc{T.~Conde, M.~Gorsky, F.~Marks, A~Zvonareva}, \emph{A functorial approach to rank functions}, J.~reine angew.~Math.~811 (2024), 135--181.

\bibitem{cummings}
   \textsc{C.~Cummings}, \emph{Ring Constructions and Generation of the Unbounded Derived Module Category}, Algebr.\ Represent.\ Theory 26 (2023), no.\ 1, 281–315.

\bibitem{goodbody2} 
    \textsc{I.~Goodbody}, \emph{Reflecting perfection for finite-dimensional differential graded algebras}, Bull.\ Lond.\ Math.\ Soc.\ 56 (2024), no.\ 12, 3689–3707.

\bibitem{goodbody} 
     \textsc{I.~Goodbody}, \emph{Approximable Triangulated Categories and Reflexive DG-categories}, arXiv:2411.09461v1. 

\bibitem{happel}
\textsc{D.~Happel}, \emph{On Gorenstein algebras}, in: Representation theory of finite groups and finite-dimensional algebras (Bielefeld, 1991), Progr.~Math.~95, 389--404 (1991).

\bibitem{hoshino_kato_miyachi} 
\textsc{M.~Hoshino, Y.~Kato, J.-I.~Miyachi}, \emph{On t-structures and torsion theories induced by compact objects}, J. Pure
Appl. Algebra 167 (1) (2002), 15–35.


\bibitem{iyengar_krause} 
    \textsc{S.~Iyengar, H.~Krause}, \emph{Acyclicity versus total acyclicity for complexes over noetherian rings}, Doc.\ Math.\ 11, 207-240 (2006).

\bibitem{jin} 
    \textsc{H.~Jin}, \emph{Cohen–Macaulay differential graded modules and negative Calabi–Yau configurations}, Adv.\ Math.\ 374 (2020), 107338, 59 pp.


\bibitem{jin_yang_zhou}
    \textsc{H.~Jin, D.~Yang, G.~Zhou}, \emph{A localisation theorem for singularity categories of proper dg algebras}, arxiv:2302.05054.

\bibitem{jorgensen2} 
     \textsc{P.~Jørgensen}, \emph{The homotopy category of complexes of projective modules}, Adv.\ Math.\ 193 (2005), no.\ 1, 223–232.

\bibitem{jorgensen} 
     \textsc{P.~Jørgensen}, \emph{Finite flat and projective dimension}, Comm.\ Algebra 33 (2005), no.\ 7, 2275–2279.


\bibitem{kalck_yang} 
    \textsc{M.~Kalck, D.~Yang}, \emph{Relative singularity categories I: Auslander resolutions}, Adv. Math. 301 (2016), 973–1021.

     

\bibitem{keller0} 
     \textsc{B.~Keller}, \emph{Deriving DG categories}, Ann.~Sci.~Éc.~Norm.~Supér.~27, 63--102 (1994).

\bibitem{keller} 
     \textsc{B.~Keller}, \emph{Calabi-Yau triangulated categories} Trends in representation theory of algebras and related topics, 467–489.\ EMS Ser.\ Congr.\ Rep.\ European Mathematical Society (EMS), Zürich, 2008

    \bibitem{koenig} 
      \textsc{S.~Koenig}, \emph{Tilting complexes, perpendicular categories and recollements of derived module categories
of rings}, J.\ Pure Appl.\ Algebra 73, No.\ 3, 211-232 (1991).


\bibitem{kontsevich} 
    \textsc{M.~Kontsevich}, \emph{Noncommutative motives}, Talk at the Institute for Advanced Study on the occasion of the 61st birthday of Pierre Deligne, October 2005.\ Video available at http://video.ias.edu/Geometry-and-Arithmetic.

\bibitem{krause2} 
      \textsc{H.~Krause}, \emph{The stable derived category of a noetherian scheme}, Compos.\ Math.\ 141 (2005), no.\ 5, 1128–1162.

\bibitem{completing perfect complexes} 
    \textsc{H.~Krause}, \emph{Completing perfect complexes. With appendices by Tobias Barthel and Bernhard Keller}, Math.\ Z.\ 296, No.\ 3-4, 1387-1427 (2020).


\bibitem{krause} 
\textsc{H.~Krause}, \emph{The finitistic dimension of a triangulated category}, Proc.\ Amer.\ Math.\ Soc., Series B 11 (2024), 570--578.


\bibitem{kuznetsov_shinder}
      \textsc{A.~Kuznetsov, E.~Shinder}, \emph{Homologically finite-dimensional objects in triangulated categories}, Sel.\ Math., New Ser.\ 31, No.\ 2, Paper No.\ 27, 45 p.\ (2025).


\bibitem{lam}
      \textsc{T.~Y.~Lam}, \emph{Lectures on modules and rings}, ser.\ Graduate Texts in Mathematics.\
Springer-Verlag, New York, 1999, vol.\ 189, pp. xxiv+557


\bibitem{manali-rahul} 
    \textsc{K.~Manali Rahul}, \emph{Representability theorems via metric techniques}, arXiv:2504.11768.

\bibitem{bounded co t structures}
    \textsc{O.~Mendoza Hernández, E.~Sáenz Valadez, V~Santiago Vargas, M.~Souto Salorio}, \emph{Auslander-Buchweitz context and co-t-structures}, Appl.\ Categ.\ Structures 21 (2013), no.\ 5, 417–440.

\bibitem{nagata}
      \textsc{M.~Nagata}, \emph{Local rings}, Interscience Tracts in Pure and Applied Mathematics, 13, Interscience Publishers, New York, 1962.

\bibitem{nakamura} 
    \textsc{T.~Nakamura}, \emph{Indecomposable pure-injective objects in stable categories of Gorenstein-projective modules over Gorenstein orders}, arXiv:2209.15630. With an Appendix by Rosanna Laking.

\bibitem{neeman2}
    \textsc{A.~Neeman}, \emph{The connection between the K–theory localisation theorem of Thomason, Trobaugh and Yao, and the smashing subcategories of Bousfield and Ravenel}, Ann.\ Sci.\ Ecole Normale Superieure 25 (1992), 547–566.


\bibitem{neeman_metrics} 
    \textsc{A.~Neeman}, \emph{Metrics on triangulated categories}, J.\ Pure Appl.\ Algebra 224 (2020), no.\ 4, 106206.


\bibitem{neeman3}
\textsc{A.~Neeman}, \emph{The t-structures generated by objects}, Trans.\ Amer.\ Math.\ Soc.\ 374, No.\ 11, 8161-8175 (2021).

\bibitem{neeman4} 
    \textsc{A.~Neeman}, \emph{Strong generators in $\mathsf{D}^{perf}(X)$
 and $\mathsf{D}^b_{coh}(X)$}, Ann.\ Math.\ (2) 193, No.\ 3, 689-732 (2021).


\bibitem{neeman5}
     \textsc{A.~Neeman}, \emph{Triangulated categories with a single compact generator and two Brown representability theorems}, arxiv:1804.02240.


\bibitem{NSZ} 
    \textsc{P.~Nicolás, M.~Saorin, A.~Zvonareva}, \emph{Silting theory in triangulated categories with coproducts}, J. Pure Appl. Algebra 223 (2019), no. 6, 2273–2319.

\bibitem{OPS}
   \textsc{S.~Oppermann, C.~Psaroudakis, T.~Stai}, \emph{Change of rings and singularity categories}, Adv.\ Math.\ 350 (2019), 190–241.


\bibitem{orlov}
     \textsc{D.~O.~Orlov}, \emph{Triangulated categories of singularities and D-branes in Landau-Ginzburg
models}, Tr.\ Mat.\ Inst.\ Steklova 246, Algebr.\ Geom.\ Metody, Svyazi i Prilozh.\ (2004), 240–262;
translation in Proc.\ Steklov Inst.\ Math.\ (2004), no.\ 3, 227–248.

\bibitem{orlov2} 
     \textsc{D.~O.~Orlov}, \emph{Triangulated categories of singularities and equivalences between Landau-Ginzburg models}, Sb.\ Math.\ 197, No.\ 12, 1827-1840 (2006); translation from Mat.\ Sb.\ 197, No.\ 12, 117-132 (2006).

\bibitem{PV} {\sc C.~Psaroudakis}, {\sc J.~Vit\'oria}, \emph{Realisation functors in tilting theory}, Math. Z. 288 (3) (2018), 965--1028.

     
     \bibitem{Raedschelders_Stevenson} \textsc{T.~Raedschelders, G.~Stevenson}, \emph{Proper connective differential graded algebras and their geometric realizations}, Eur.\ J.\ Math.\ 8 (2022), S574–S598.
     
     \bibitem{ravenel}
     \textsc{D.~C.~Ravenel}, \emph{Localization with respect to certain periodic homology theories}, Amer.\ J.\ of Math.\ 105 (1984), 351–414.

\bibitem{rickard_morita} 
       \textsc{J.~Rickard}, \emph{Morita theory for derived categories}, J.\ London Math.\ Soc.\ (2) 39 (1989), no.\ 3, 436–456.

\bibitem{rickard}
       \textsc{J.~Rickard}, \emph{Unbounded derived categories and the finitistic dimension conjecture}, Adv.\ Math.\ 354 (2019), 106735, 21 pp.

\bibitem{rouquier} 
      \textsc{R.~Rouquier}, 
      \emph{Dimensions of triangulated categories}, J.\ K-Theory 1 (2008), no.\ 2, 193-256.  

\bibitem{stevenson} 
      \textsc{G.~Stevenson}, 
      \emph{Rouquier dimension versus global dimension}, J.\ Pure Appl.\ Algebra 229 (2025).


\bibitem{schwede_shipley} 
     \textsc{S.~Schwede, B.~Shipley}, \emph{Stable model categories are categories of modules}, Topology 42, 103–153 (2003).

\bibitem{toen} 
      \textsc{B.~Toën}, \emph{Homotopy finiteness of proper and smooth dg-algebras}, Proc.\ Lond.\ Math.\ Soc.\ (3) 98, No.\ 1, 217-240 (2009).

\end{thebibliography}
\end{document}